\documentclass{article}
\usepackage{geometry}
\usepackage{graphicx}

\usepackage{url}

\usepackage{subcaption}

\usepackage{mathrsfs}

\usepackage{amsthm}
\usepackage{amsmath,bm}
\usepackage{amssymb}

\newtheorem{theorem}{Theorem}
\newtheorem{definition}{Definition}
\newtheorem{remark}{Remark}
\newtheorem{example}{Example}
\newtheorem{lemma}{Lemma}
\newtheorem{corollary}{Corollary}
\newtheorem{proposition}{Proposition}
\newtheorem{claim}{Claim}

\usepackage{xcolor} 

\usepackage{tikz}
\usetikzlibrary{fit,shapes,arrows,positioning,decorations.pathmorphing}

\usepackage{bbm}

\newcommand{\Ob}{\mathcal{O}}

\newcommand{\D}{\mathrm{d}}
\newcommand{\Legendre}{\mathrm{P}}

\DeclareSymbolFont{mathroman}{OT1}{cmr}{m}{n}
\DeclareMathSymbol{e}{\mathalpha}{mathroman}{`e}
\DeclareMathSymbol{d}{\mathalpha}{mathroman}{`d}
\DeclareMathSymbol{i}{\mathalpha}{mathroman}{`i}
\DeclareMathSymbol{T}{\mathalpha}{mathroman}{`T}
\DeclareMathSymbol{J}{\mathalpha}{mathroman}{`J}
\DeclareMathSymbol{Y}{\mathalpha}{mathroman}{`Y}
\DeclareMathSymbol{H}{\mathalpha}{mathroman}{`H}

\newcommand\revisions[1]{#1}%\textcolor{red}{#1}}

\usepackage[affil-sl]{authblk}
\usepackage{fancyhdr}
\author[1]{G. Maierhofer}
\affil[1]{Laboratoire Jacques-Louis Lions\\
	Sorbonne University, France\\georg.maierhofer@sorbonne-universite.fr \vspace{1em}}
\author[2]{A. Iserles}	\author[2]{N. Peake}	
\affil[2]{Department of Applied Mathematics and Theoretical Physics\\
	University of Cambridge\\ UK}	

\usepackage[]{fancyhdr}
\begin{document}

\title{Recursive moment computation in Filon methods and application to high-frequency wave scattering in two dimensions}
% Short title for running heads:
%\shorttitle{A Filon--Clenshaw--Curtis method for high-frequency wave scattering}

%\author[G. Maierhofer, A. Iserles and N. Peake]{%
%{\sc
%G. Maierhofer\thanks{Corresponding author. Email: \revisions{georg.maierhofer@sorbonne-universite.fr}}\\[2pt]
%\revisions{\textit{Laboratoire Jacques-Louis Lions, Sorbonne University, 4 Place Jussieu, 75005 Paris, France}}\\[4pt]
%A. Iserles\\[2pt]
%\textit{DAMTP, Centre for Mathematical Sciences, University of Cambridge, Wilberforce Road, Cambridge, CB3 0WA, UK}\\[4pt]
%and\\[4pt]
%N. Peake} \\[2pt]
%DAMTP, Centre for Mathematical Sciences, University of Cambridge, Wilberforce Road, Cambridge, CB3 0WA, UK
%}

% Short list of authors for running heads:
%\shortauthorlist{G. Maierhofer, A. Iserles and N. Peake}

\maketitle
\ \vspace{-11.5cm}\ \ \begin{center}\textit{\small This article has been accepted for publication in IMA Journal of Numerical Analysis published by Oxford University Press. The present version is the unreviewed author’s original version.}\end{center}\ \\\vspace{8.5cm}

\begin{abstract}
	% Body of abstract:
	{We study the efficient approximation of highly oscillatory integrals using Filon methods. A crucial step in the implementation of these methods is the accurate and fast computation of the Filon quadrature moments. In this work we demonstrate how recurrences can be constructed for a wide class of oscillatory kernel functions, based on the observation that many physically relevant kernel functions are in the null space of a linear differential operator whose action on the Filon interpolation basis is represented by a banded (infinite) matrix. We discuss in further detail the application to two classes of particular interest, integrals with algebraic singularities and stationary points and integrals involving a Hankel function. We provide rigorous stability results for the moment computation for the first of these classes and demonstrate how the corresponding Filon method results in an accurate approximation at truly frequency-independent cost. For the Hankel kernel, we derive \revisions{error estimates which describe} the convergence behaviour of the method in terms of frequency and number of Filon quadrature points. Finally, we show how Filon methods with recursive moment computation can be applied to compute efficiently integrals arising in hybrid numerical-asymptotic collocation methods for high-frequency wave scattering on a screen.
	}

\end{abstract}

%\section{To do overall}
%\questions{\begin{itemize}
		%		\item Thank annonymous reviewers for their appropriate contributions!
		%		\item Check if we addressed all reviewers concerns!
		%		\item take care of references!
		%		\item Ask Arieh if the Hilber space formulation is ok, correct?!
		%\end{itemize}}
		
		\section{Introduction}
		\label{sec:introduction}
		The efficient numerical approximation of highly oscillatory integrals is an essential step in the simulation of many physical systems involving high-frequency phenomena. Although efficient methods for the computation of certain highly oscillatory integrals had been discovered as early as the first half of the twentieth century by Louis Napoleon George \cite{filon1930iii}, and thorough research over the past decades has lead to an immense increase in efficiency and applicability of such methods \cite{deano2017computing}, many open problems remain. Amongst them is the computation of Filon quadrature moments, the so-called `moment-problem', which can be given in simple explicit form only for isolated instances of interpolation bases and oscillators. This hinders the direct application of Filon methods to several important classes of integrals which involve complicated oscillatory kernel functions, including integrals arising in hybrid numerical-asymptotic boundary integral methods for high-frequency wave scattering where oscillatory basis functions need to be integrated against a singular, oscillatory Green's function \cite{chandler2012numerical}.
		
		In the present work we address this problem by providing a method for the construction of recursive relations satisfied by the Filon quadrature moments, which leads to a very efficient strategy for finding the moments in a range of settings. This is a continuation of extensive work on Filon methods over the past two decades which was started by \cite{iserles2004,iserles2005} and \cite{iserlesnorsett2004} who were the first to provide a detailed asymptotic error analysis and an extension of the ideas presented by \cite{filon1930iii} specifically describing the favourable asymptotic properties of Filon methods in the high-frequency regime. Improvements to the Filon method by reducing the asymptotic error were introduced by \cite{iserlesnorsett2005efficient} through including information about the derivative values of the amplitude function $f$ and resulted in the development of the extended Filon method which was studied in greater detail by \cite{Gao2017a,Gao2017b}. With the goal to understand the quadrature error uniformly also for small frequencies, it was shown by \cite{melenk2010} that the error analysis of Filon methods for non-stationary oscillators can essentially be reduced to the study of the interpolation error of the amplitude function at the relevant quadrature points. Both \cite{melenk2010} (based on analyticity properties in a neighbourhood of the domain of integration) and \cite{Dominguez2011} (based on the regularity of the amplitude in certain periodic Sobolev spaces) use this observation to provide error estimates that are explicit in the frequency of oscillations as well as the number of interior quadrature points in the non-stationary case.
		
		The interest in extending these Filon methods from simple linear oscillators to more general kernels has led to work by \cite{olver2006,olver_2007} who described a moment-free version of the Filon method that is applicable to algebraic singularities and stationary points. A different type of approach involves recursive moment computation, which has been successfully applied to a number of individual cases in the context of Clenshaw--Curtis interior points: for integrals involving Bessel functions of linear arguments by \cite{piessens1983modified}, for \revisions{exponential oscillators with linear phase functions} by \cite{Dominguez2011} and for \revisions{exponential oscillators with linear phase and a logarithmic amplitude singularity by} \cite{dominguez2014}. \revisions{Our present work can be seen as a generalisation of these previous recursive approaches.} In a related context of computing indefinite integrals over oscillatory and singular functions using a Levin-type method, \cite{keller1999,keller2007method} described a method for the recursive computation of Chebyshev coefficients for functions that satisfy a linear differential equation with polynomial coefficients. This is based on earlier work by \cite{Lewanowicz1991} on the recursive computation of Jacobi coefficients of special functions satisfying similar differential equations. The ideas underpinning these final three studies, namely that the null space of certain differential operators can be related to expansions in a Hilbert basis whose coefficients satisfy recurrences, are closely related to Thm.~\ref{thm:recurrence_by_duality} in our present work.
		
		In recent years, high-frequency wave scattering has provided strong motivation for further advances in the development of highly oscillatory quadrature. \cite{dominguez2013filon} constructed a composite (graded) version of the Filon method that can be applied to arbitrary algebraic and logarithmic singularities, and which has been successfully applied to hybrid numerical-asymptotic methods in wave scattering by \cite{chandler2012numerical}, \cite{kim2012} and \cite{parolin2015}. While this method is already significantly better than traditional quadrature, the flexibility of this graded method comes at the price of \revisions{losing} some of the favourable asymptotic properties of Filon methods, and we will see in the present work how this may be overcome in certain cases by the construction of a direct Filon method for the corresponding integrals. An alternative approach to computing highly oscillatory integrals is numerical steepest descent which was introduced by \cite{huybrechs2006evaluation}. Numerical steepest descent has recently been applied by \cite{gibbs2019fast} to wave scattering problems on multiple screens \revisions{in the case where the screens are aligned} (see \cite{software_hnabemlab,software_pathfinder}). These results serve as a reference for the application of our methods to a collocation method in high-frequency wave scattering in \S\ref{sec:numerical_examples_scattering_on_screen}.% \questions{[Possibly update the reference here!]}. While numerical steepest descent in general outperforms Filon methods in terms of asymptotic order and computational expense, Filon methods have the advantage that they are more robust towards changes in the behaviour of $f$, and in particular require no analyticity assumptions on the integrand as opposed to the requirement of at least local analyticity for numerical steepest descent. Furthermore, as long as moments can be computed accurately, Filon methods permit uniform error estimates that are explicit both in the number of interior interpolation points and in the frequency of the integral (see \cite{Dominguez2011} and \cite{Gao2017b}).
		
		%\subsection{Overview and main results}
		
		The structure and main results of this manuscript are as follows: We begin with a general description of the extended Filon method as introduced by \cite{Gao2017a,Gao2017b} in \S\ref{sec:extended_filon_method}. This is setting the scene for the `moment-problem' in Filon methods, and specifically for our first main result, Thm.~\ref{thm:recurrence_by_duality}, which we prove in \S\ref{sec:moment_computation_by_duality}. The theorem provides a set of sufficient conditions for Filon moments to satisfy recurrences, and is based on the observation that many relevant interpolation bases are in fact a (scaled) Hilbert basis of a weighted $L^2$-space and that several relevant oscillators satisfy certain differential equations. Following two instructive examples, we focus on applying this methodology to Filon--Clenshaw--Curtis methods in \S\ref{sec:application_to_integrals_with_singularities} \& \S\ref{sec:application_to_wave_scattering}.
		
		Specifically, in \S\ref{sec:application_to_integrals_with_singularities} we construct a direct Filon--Clenshaw--Curtis method for integrals with either a stationary point or an algebraic singularity. In this section our main results are the rigorous stability analysis for the corresponding moment recurrences in theorems \ref{thm:full_stability_quadratic_oscillator} \& \ref{thm:square_root_stability_algebraic_singularity}. Although we focus our attention to the initial stability regime (which is most relevant for practical computations) we also indicate how one may use Oliver's algorithm \cite{oliver1968numerical} for the stable computation of the tail (which is mostly of theoretical interest since in practice the computational advantage of Filon methods over classical quadrature exists only when the number of required moments is smaller than the frequency of the oscillator $N\lesssim \omega$). Numerical examples are included in \S\ref{sec:numerical_examples_algebraic_singularities} demonstrating the advantage of this direct application of the Filon method over composite versions.
		
		The second major application of our methodology is described in \S\ref{sec:application_to_wave_scattering}, where we consider the direct construction of Filon methods for hybrid numerical-asymptotic collocation methods for high-frequency wave scattering on a screen. The first step in this construction is the proof of a Filon paradigm in Prop.~\ref{prop:filon_paradigm_for_alpha_zero}, which is a simple, but non-trivial result describing the asymptotic behaviour of the integral over the combination of a linear exponential oscillator and a Hankel function, the latter of which has a frequency dependent singularity in the domain of integration. This facilitates the study of error estimates that are explicit in both frequency and number of interior points in Corollary~\ref{cor:nu_explicit_error_estimates}. Although the stability for the relevant moment recurrences is non-tractable for analytic study, \revisions{these recurrences provide an extremely efficient way to compute the Filon moments in practice} when combined with the expressions for initial moments found in Lemma~\ref{lem:expression_for_initial_moments_linear_oscillator}. We evaluate the practical performance of the method based on an example of a hybrid numerical-asymptotic method describing the scattering of a Gaussian beam by a finite plate in \S\ref{sec:numerical_examples_scattering_on_screen}.
		
		Our results are summarised and an outlook towards future research directions is provided in the concluding remarks in \S\ref{sec:conclusions}.

		\section{The extended Filon method}\label{sec:extended_filon_method}
		We begin with a review of the extended Filon method as introduced by \cite{iserlesnorsett2005efficient} and \cite{Gao2017a,Gao2017b} based on the following generic form of a one-dimensional oscillatory integral:
		\begin{align*}
			I_\omega[f]=\int_{a}^b f(x)h_\omega(x)dx,\quad -\infty<a<b<\infty.
		\end{align*}
		Here the kernel function $h_\omega(x)$ is an $\omega$-oscillatory function, \revisions{by which we broadly mean a function whose oscillations depend on $\omega$. One may think of the example $h_\omega(x)=\exp(i\omega g(x))$ with a suitable choice of $g$, although in \S\ref{sec:application_to_wave_scattering} we will also study the possibility when $h_\omega(x)$ is expressed in terms of certain special functions.} Filon quadrature methods are designed to approximate $I_\omega[f]$ with good accuracy and at uniform cost when $\omega\gg1$.  \cite{iserlesnorsett2005efficient} observed that for $h_\omega(x)=\exp(i\omega g(x))$ with $g'(x)\neq 0, x\in(a,b),$ the asymptotic expansion of $I_\omega[f]$ for large $\omega$ depends only on the values $\mathcal{S}=\{f^{(j)}(a),f^{(j)}(b)\big|j=0,1,\dots\}$ and thus proposed to construct a Filon quadrature method by  computing $\mathcal{Q}_\omega^{[\nu,s]}[f]:=I_\omega[p],$ where $p$ is an interpolating polynomial of degree $2s+\nu+1$ satisfying the Hermite-type interpolation conditions
		\begin{align}\begin{split}\label{eqn:extended_Filon_interpolation_problem}
				\revisions{p^{(j)}(a)=f^{(j)}(a),\,\,\,p^{(j)}(b)=f^{(j)}(b),}\,\,\, j=0,\dots, s\,\,\,\,\text{\ and\ }\,\,\,\, p(c_l)=f(c_l),\quad l=1,\dots,\nu,\end{split}
		\end{align}
		for some specified interior interpolation points $a=c_0<c_1<\cdots<c_\nu<c_{\nu+1}=b$. Since the asymptotic behaviour of $I_\omega[f]$ is determined by the values $\mathcal{S}$, one can show that the asymptotic error of this quadrature method is
		\begin{align*}
			\left|\mathcal{Q}_\omega^{[\nu,s]}[f]-I_\omega[f]\right|=\mathcal{O}\left(\omega^{-s-2}\right),\quad \omega\rightarrow\infty,
		\end{align*}
		i.e. it can be made to decay at an arbitrary algebraic rate in $\omega$ so long as $f$ possesses a sufficient number of derivatives on $[a,b]$ (cf. \eqref{eqn:asymptotic_error_algebraic_singularities} and Prop.~\ref{prop:filon_paradigm_for_alpha_zero}). This idea extends more generally also to oscillators with stationary points and to higher dimensions: as long as the derivative values of $p$ match those of $f$ \revisions{up to certain order on a specified set of points, the asymptotic error of the Filon quadrature method constructed analogously to above decays in $\omega$. For integrals of the form $\int_\Omega f(\mathbf{x})\exp(i\omega g(\mathbf{x}))\,d x$ where $\Omega\subset \mathbb{R}^n$ is a polytope and $f,g$ are sufficiently differentiable, this set of points is known to consist of vertices of the polytope, stationary points (points $\mathbf{x}\in\Omega$ where $\bm{\nabla}g(\mathbf{x})=0$) and hidden stationary points (points $\mathbf{x}\in\partial\Omega$ where $\bm{\nabla}g(\mathbf{x})$ is orthogonal to $\partial\Omega$). For more details we refer the reader to \cite[\S2.4 \& \S 4]{deano2017computing}.}
		\subsection{The Achilles' heel of Filon methods: Moment computation}
		In practice the interpolation problem \eqref{eqn:extended_Filon_interpolation_problem} is solved by finding the coefficients of $p$ with respect to a given set of interpolation basis functions $\left\{\phi_n\right\}_{n=0}^{\nu+2s+1}$. These are typically polynomials on finite intervals, but we shall see an example below where $\phi_n$ can be taken to be Fourier modes (see example \ref{example:periodic_integral}). Although formally there is no difference in the particular choice of finite degree polynomial basis, the conditioning and (as explained in \S\ref{sec:fast_interpolation_at_FCC_points}) the speed of the interpolation algorithm are affected and, depending on the choice of interior nodes $c_{\revisions{l}}, \revisions{l}=1,\dots \nu$, in several cases it is advantageous to express the interpolating polynomial in terms of its coefficients with respect to a specific basis of orthogonal polynomials. Following the solution of the interpolation problem \eqref{eqn:extended_Filon_interpolation_problem} the polynomial $p$ is thus expressed in the form
		\begin{align*}
			p(x)=\sum_{n=0}^{\nu+2s+1}p_n\phi_n(x)
		\end{align*}
		and we compute the Filon quadrature as \revisions{$\mathcal{Q}_\omega^{[\nu,s]}[f]=\sum_{n=0}^{\nu+2s+1}p_nI_\omega[\phi_n]$}. Hence an important step in the Filon method is to evaluate the Filon quadrature moments $\{I_\omega[\phi_n]\}_{n=0}^{\nu+2s+1}$. For a general oscillator $h_\omega(x)$ this task is extremely tricky, mainly because explicit expressions for the moments are only rarely available or given in terms of special functions that are computationally expensive to evaluate. Thus it is fair to say that \textit{the computation of moments is the Achilles' heel of Filon methods}.
		
		\subsection{Fast interpolation at Filon--Clenshaw--Curtis points}\label{sec:fast_interpolation_at_FCC_points}
		Before embarking on a more thorough study of recursive moment computation, let us consider the choice of interior interpolation points in greater detail. A good choice of interior points $c_n$ is determined by a number of competing goals: accuracy for $\omega\gg 1$, uniform accuracy, simplicity of coefficients, and, for large values of $\nu$, minimization of computation cost for the interpolation problem. This aspect was studied in \cite[\S 4.2]{deano2017computing}, and it was first shown by \cite{Dominguez2011} that for finite intervals $[a,b]$ a particularly interesting choice of intermediate interpolation points are shifted Clenshaw--Curtis points: Suppose without loss of generality $a=-1,b=1$, then the Clenshaw--Curtis points are given by
		\begin{align}\label{eqn:definition_CC_points}
			c_{\revisions{l}}=\cos(\revisions{l}\pi/(\nu+1)), \quad \revisions{l}=0,\dots,\nu+1.
		\end{align}
		It is well-known that for $s=0$ the solution of the interpolation problem \eqref{eqn:extended_Filon_interpolation_problem} can be expressed as a finite linear combination of Chebyshev polynomials $T_n$ using a single application of a \textit{Discrete Cosine Transform}, DCT-I, bringing the cost of the interpolation part of the Filon method to just $\mathcal{O}(\nu\log\nu)$ operations \revisions{(cf. \cite{trefethen2008gauss} and references therein)}. The work by \cite{Dominguez2011} is particularly close to the present manuscript as it proposed a recursive approach for computing the Chebyshev moments $I_\omega[T_n]$ for a linear exponential oscillator $\exp(i\omega x)$. Our present work can be seen as a generalisation of this recursive approach. The advantageous interpolation properties of interior Chebyshev points carry over to non-zero values of $s$ as was shown by \cite{Gao2017a}: the interpolation problem \eqref{eqn:extended_Filon_interpolation_problem} for $p$ in a basis of Chebyshev polynomials can be solved very efficiently by a single application of DCT-I, and the solution of a small auxiliary linear system of size $2s\times 2s$, bringing the overall cost of interpolation to $\mathcal{O}(\nu\log\nu +\nu s+s^3)$. For full details on this procedure we refer the reader to the original work \cite{Gao2017a}.
		
		It is possible, however, for the asymptotic behaviour of the integral $I_\omega[f]$ to also depend on interior values of $f$, for instance in the presence of a stationary point of $h_\omega$. Subject to a few minor modifications, which we describe in Appendix A, the aforementioned procedure can also be used to solve efficiently the following interpolation problem: Compute $q$, the unique polynomial of degree $3s+\nu+1$ such that
		\revisions{\begin{align}\begin{split}\label{eqn:interpolation_problem_p2}
					q^{(j)}(0)&=f^{(j)}(0),\,\,\quad j=1,\dots,s,\\
					\,\, q^{(j)}(\pm1)&=f^{(j)}(\pm1),\,\,\,\, j=0,\dots,s,\\
					\,\,\,\,\text{\ and\ }\,\,\,\,q(c_l)&=f(c_l),\,\,\quad\,\,\,\,\, l=1,\dots,\nu,\end{split}
		\end{align}}
		where $c_{\revisions{l}}$ are as in \eqref{eqn:definition_CC_points} and this time we take $\nu$ to be odd, to ensure that $0=c_{(\nu-1)/2}$ is amongst the interior points. Indeed, as was the case for \eqref{eqn:extended_Filon_interpolation_problem}, the Chebyshev coefficients of $q$ can again be found in
		$\mathcal{O}(\nu\log\nu+\nu s+s^3)$ operations. We will use this result when considering integrals with stationary points and algebraic singularities at the origin in \S \ref{sec:application_to_integrals_with_singularities}.
		
		\section{Recursive moment computation in Filon methods}\label{sec:moment_computation_by_duality}
		
		%\questions{To do this section:
			%\begin{itemize}
			%	\item Check notation is consistent - especially with differential operators L and basis functions phi.
			%\end{itemize}}
			As remarked above a crucial step in the Filon method is the accurate and efficient computation of the quadrature moments $I_\omega[\phi_n]$. In this section we present a constructive result that can be used to find recurrences for these quadrature moments in a range of settings, by regarding them as the coefficients in a Hilbert basis formed by the (appropriately scaled) interpolation basis. This generally yields a highly efficient way for computing Filon moments, provided initial conditions for the recurrence can be found.
			
			We note that a related result for the recursive computation of Chebyshev coefficients of functions satisfying linear ODEs with polynomial coefficients is given in Lemma 2.4 by \cite{keller2007method} and we also highlight similar work by \cite{Lewanowicz1991} for the recursive computation of Jacobi coefficients of functions satisfying linear ODEs with polynomial coefficients. The general constructive result is given in Thm.~\ref{thm:recurrence_by_duality}, but we shall relax some of the assumptions in the sequel. We provide a rigorous stability analysis of some of these types of recurrences for Filon--Clenshaw--Curtis methods in \S\ref{sec:application_to_integrals_with_singularities}. 
			
			\revisions{In the following we will, for some measurable weight function $W:[a,b]\rightarrow[0,\infty)$ such that $W(x) dx$ is a non-trivial Borel measure, denote by $L^{2}([a,b],W)$ the usual space of functions which are square integrable against the weight function $W$, together with the inner product 
				\begin{align*}
					\left( f,g\right)_{L^2([a,b],W)}=\int_a^b f(x) \overline{g(x)}W(x)dx.
				\end{align*}
				In the statement of the central theorem it will be helpful to refer to the following spaces of functions:
				\begin{definition}
					Given an orthonormal set of basis functions $\{\phi_n\}_{n\in\mathcal{I}}$ ($\mathcal{I}=\mathbb{N}$ or $\mathbb{Z}$) of $L^2([a,b],W)$, we denote by $H^{s}([a,b],W)$ for any $s\geq0$ the space
					\begin{align*}
						H^{s}([a,b],W):=\left\{f\in L^2([a,b],W)\,\Big|\, \|f\|_{H^{s}([a,b],W)}<\infty\right\},
					\end{align*}
					where we introduced the notation
					\begin{align*}
						\|f\|_{H^{s}([a,b],W)}:=\left(\sum_{m\in\mathcal{I}}[m]^{2s}\left|\left( f,\phi_m\right)_{L^2([a,b],W)}\right|^2\right)^{\frac{1}{2}},\quad [m]=\begin{cases}
							1,&m=0,\\
							|m|,&m\neq 0.
						\end{cases}
					\end{align*}
				\end{definition}\label{def:generalised_sobolev_spaces}
				Analogous to the theory of Sobolev spaces on periodic domains it is straightforward to check that $H^{s}([a,b],W)$ together with $\|\,\cdot\,\|_{H^{s}([a,b],W)}$ is a Hilbert space, that for any $s>t\geq 0$, $H^{s}([a,b],W)\subset H^{t}([a,b],W)$, and that $H^{0}([a,b],W)=L^2([a,b],W)$. Let us denote, for $s>0$, by $H^{-s}([a,b],W)$ the continuous dual space of $H^{s}([a,b],W)$. Since the dual space of $L^2([a,b],W)$ is represented by $L^2([a,b],W)$, we have a natural embedding $L^2([a,b],W)\subset H^{-s}([a,b],W)$ for any $s>0$. Similar to Sobolev spaces on periodic domains we then have a natural expression for the norm of $f\in L^2([a,b],W) \subset H^{-s}([a,b],W)$
				\begin{align*}
					\|f\|_{H^{-s}([a,b],W)}:=\sup_{g\in H^{s}([a,b],W)\setminus\{0\}}\left|\langle g,f\rangle_{s\times-s}\right|=\left(\sum_{m\in\mathcal{I}}[m]^{-2s}\left|\left( f,\phi_m\right)_{L^2([a,b],W)}\right|^2\right)^{\frac{1}{2}}
				\end{align*}
				where in the above we denoted by $\langle\,\cdot\,,\,\cdot\,\rangle_{s\times-s}:H^{s}([a,b],W)\times H^{-s}([a,b],W)\rightarrow \mathbb{C}$ the natural duality pairing. Note that for $f\in H^{s}([a,b],W)\cap L^2([a,b],W), g\in H^{-s}([a,b],W)\cap L^2([a,b],W)$ we have
				\begin{align}\label{eqn:special_expression_duality_pairing}
					\langle f,g\rangle_{s\times-s}=\left(f,g\right)_{L^{s}([a,b],W)}.
				\end{align}
				A simple argument then shows that $H^{-s}([a,b],W)$ is the completion of $L^2([a,b],W)$ with respect to the norm $\|\cdot\|_{H^{-s}([a,b],W)}$ i.e. that $L^2([a,b],W)$ is a dense subspace. These arguments are analogous to the theory of standard Sobolev spaces on periodic domains which can be found for instance in \cite[]{saranen2002periodic}. With these tools we are now able to provide a sufficient condition that ensures existence of recurrences for the moments of the Filon method.}
			
			\begin{theorem}\label{thm:recurrence_by_duality}
				Let $\{\phi_n\}_{n\in\mathcal{I}}$ ($\mathcal{I}=\mathbb{N}$ or $\mathbb{Z}$) be a \revisions{a complete orthonormal set of basis functions} of $L^2([a,b],W)$ \revisions{where $-\infty< a<b<\infty$ and $W:[a,b]\rightarrow[0,\infty)$ is a measurable weight function such that $W(x)dx$ is a non-trivial Borel measure}. Consider the moments expressed in the form
				\begin{align*}
					\sigma_n=\int_a^b \phi_n(x) h_\omega(x)W(x)dx.
				\end{align*}
				We assume the existence of a linear differential operator $\mathcal{L}_\omega$ of order $s\in\mathbb{N}$ such that $\mathcal{L}_\omega:H^{t+s}([a,b],W)\rightarrow H^{t}([a,b],W)$ \revisions{is bounded for some $t\geq 0$, that $h_\omega\in H^{t+s}([a,b],W)$ and $\mathcal{L}_\omega h_\omega=0$.} \revisions{Suppose further that} %$\{\phi_n\}_{n\in\mathcal{I}}\subset H^{\revisions{t+}s}([a,b],W)$ and that 
				the action of $\mathcal{L}_\omega$  on the conjugate of the basis functions is given by a banded (infinite) matrix $B_{mn}$ with bandwidth $k$, such that
				\begin{align*}
					\mathcal{L}_\omega\overline{\phi_n}&=\sum_{m\in\mathcal{I}} B_{nm}\overline{\phi_m}.
				\end{align*} 
				Then the moments satisfy a $k+1$-term recurrence relation, $\sum_{n\in\mathcal{I}} B^T_{mn}\sigma_n=0$, which together with $k$ initial or boundary conditions uniquely determines all moments.
			\end{theorem}
			\begin{proof}
				%Let us denote by $(\cdot,\cdot)_{L^2}$ the inner product on $L^2([a,b],W)$, with the convention
				%\begin{align*}
				%	\left( f,g\right)_{L^2}=\int_a^b f(x) \overline{g(x)}W(x)dx.
				%\end{align*}
				%We also denote by $\langle f,g\rangle$ the action of $g\in H^{-s}([a,b],W)$ on $f\in H^{s}([a,b],W)$. 
				\revisions{Let us} consider the adjoint map $\mathcal{L}_{\omega}^*:H^{-t}([a,b],W)\rightarrow H^{-s-t}([a,b],W)$. Clearly, the set of complex conjugates $\{\overline{\phi_n}\}_{n\in\mathcal{I}}$ is also a Hilbert basis for \revisions{$L^2([a,b],W)$}. \revisions{As noted above,} \revisions{$L^2([a,b],W)$} is dense in $H^{-s-t}\revisions{([a,b],W)}$, meaning in particular that
				\begin{align*}
					\mathrm{cl}_{H^{-s-t}\revisions{([a,b],W)}}\left[\mathrm{span}\left(\overline{\phi_n}\right)_{n\in\mathcal{I}}\right]=\revisions{H^{-s-t}([a,b],W)}.
				\end{align*}
				Thus we can write $\mathcal{L}_{\omega}^* \overline{\phi_n}=\lim_{N\rightarrow\infty} \sum_{|m|\leq N} a_{nm}\overline{\phi_m}$ for some $a_{nm}\in\mathbb{C}$ 
				and since $\langle \mathcal{L}_\omega\overline{\phi_m},\overline{\phi_n}\rangle_{\revisions{t\times -t}}=\langle \overline{\phi_m},\mathcal{L}_{\omega}^*\overline{\phi_n}\rangle_{\revisions{t+s\times -t-s}}$ we have $a_{nm}=\overline{B_{mn}}$, i.e. $\mathcal{L}_{\omega}^* \overline{\phi_n}=\sum_{m\in\mathcal{I}} \overline{B_{mn}}\,\overline{\phi_m}$. Hence the action of the adjoint $\mathcal{L}_{\omega}^*$ on $\overline{\phi_{n}}$ is given in terms of the banded infinite matrix $\overline{B}^T$. We conclude the proof by noting that the moments are the coefficients of $h_\omega$ with respect to the basis $\{\overline{\phi_n}\}_{n\in\mathcal{I}}$. Therefore, for all $m\in\mathcal{I}$,
				\begin{align*}
					\sum_{n\in\mathcal{I}}B_{mn}^T\sigma_n&=\sum_{n\in\mathcal{I}}B_{mn}^T\left(h_\omega,\overline{\phi_n}\right)_{L^2\revisions{([a,b],W)}}=\langle h_\omega,\sum_{n\in\mathcal{I}}\overline{B_{nm}}\,\overline{\phi_n}\,\rangle_{\revisions{t+s\times-s-t}}\\&=\left\langle h_\omega,\mathcal{L}_{\omega}^* \overline{\phi_m}\right\rangle_{\revisions{t+s\times-s-t}}=\left\langle \mathcal{L}_\omega h_\omega,\overline{\phi_m}\right\rangle_{\revisions{t\times-t}}=0,
				\end{align*}
				where we used \revisions{\eqref{eqn:special_expression_duality_pairing} and} the fact that $B_{mn}$ is banded, so that all the sums are over a finite number of non-zero terms.
			\end{proof}
			\revisions{\begin{remark}\label{rmk:integration_by_parts}The central step in the above proof is to see that the formal operator which is represented by the matrix $\overline{B}^T$ corresponds to the continuous adjoint of the differential operator $\mathcal{L}_\omega$ with respect to the appropriate spaces as defined above. With this in mind we may, in essence, regard the final step in the proof
					\begin{align*}
						\left\langle h_\omega,\mathcal{L}_{\omega}^* \overline{\phi_m}\right\rangle_{\revisions{t+s\times-s-t}}=\left\langle \mathcal{L}_\omega h_\omega,\overline{\phi_m}\right\rangle_{\revisions{t\times-t}}
					\end{align*}
					as an integration-by-parts argument. In general, integration-by-parts would introduce a contribution from the boundary of the domain at $x=a,b$, and these terms are implicitly accounted for in $\mathcal{L}_{\omega}^*$. However, due to the representation of this operator by $\overline{B}^T$ these boundary terms do not affect the banded matrix representation. If our basis $\{\phi_n\}_{n\in\mathcal{I}}$ includes the constant function, i.e. without loss of generality if $\phi_0(x)=(\int_a^bW(y)dy)^{-1}$ then the boundary terms (which are constants) would introduce a column of infinitely many non-zero constants in $\overline{B}^T$, i.e. $\overline{B}^T_{0m}\neq 0$ for infinitely many $m\in\mathcal{I}$. Thus the matrix would no longer be banded. Therefore, whenever our basis includes a constant function (which is the case for all examples considered in this manuscript) the banded representation of $\mathcal{L}_\omega$ means that $\mathcal{L}_\omega$ is constructed such that all boundary terms in the above integration-by-parts step vanish.
			\end{remark}}
			\revisions{\begin{remark}In Def.~\ref{def:generalised_sobolev_spaces} and the statement and proof of Thm~\ref{thm:recurrence_by_duality} the only properties we required about the domain $[a,b]$ were the existence of a topology and the notion of a derivative ($\mathcal{L}_\omega$ is a differential operator). Therefore, the above holds equally if we replaced $[a,b]$ by an arbitrary closed subset of a topological vector space. In particular, Thm.~\ref{thm:recurrence_by_duality} can be proved analogously if we replace $[a,b]$ by the entire real line $(-\infty,\infty)$, a periodic interval $[0,2\pi)$ or a general closed subset $\mathcal{D}\subset \mathbb{R}^n$.
			\end{remark}}
			Note in several important cases it is possible to find initial conditions in terms of special functions, or, alternatively, in terms of simple integrals that can be approximated efficiently (for instance exponentially decaying integrals as in Lemma~\ref{lem:expression_for_initial_moments_linear_oscillator}). Moreover, the choice of $\mathcal{L}_\omega$ is not unique, but in practice it is often possible to spot a simple choice by inspection, leading to a low-order recurrence. Let us begin by illustrating the result with a simple example where recurrences for moments are already well-known:
			
			\begin{example}\label{example:periodic_integral}
				Consider $f\in \mathrm{L}_2\revisions{([0,2\pi))}\cap C_\mathrm{per}\revisions{([0,2\pi))}$ and the oscillatory integral
				\begin{align*}
					I_\omega[f]:=\int_{0}^{2\pi}e^{i\omega \cos x}f(x) dx.
				\end{align*}
				In this case a natural interpolation basis is the Fourier basis $\left\{\frac{1}{\sqrt{2\pi}}e^{in x}\right\}_{n\in\mathbb{Z}}$, which has good interpolation properties on equispaced points, and which is also a Hilbert basis for $L^2\revisions{([0,2\pi))}$. The oscillator in the weighted space is $h_\omega(x)=\sqrt{2\pi} e^{i\omega \cos x}$, satisfying
				\begin{align*}
					\mathcal{L}_\omega h_\omega=0,\quad \mathcal{L}_\omega=\frac{d}{dx}+i\omega\sin x.
				\end{align*}
				
				\revisions{Let us check carefully that the conditions of Thm.~\ref{thm:recurrence_by_duality} are satisfied. To begin with, we note that in the present setting the spaces $H^{s}(\revisions{[0,2\pi)},1)$ mentioned above restrict to the standard Sobolev spaces on the periodic domain $\revisions{[0,2\pi)}$. By the Sobolev embedding theorem, $H^{2}(\revisions{[0,2\pi)},1)\subset C^{1}_\mathrm{per}(\revisions{[0,2\pi)})$, i.e. any function in $H^2$ is at least once continuously differentiable (this is true in fact for $f\in H^{s}$ any $s>3/2$, but for us the weaker observation suffices). Moreover, one can easily check that $\mathcal{L}_\omega:H^{2}(\revisions{[0,2\pi)},1)\rightarrow H^{1}(\revisions{[0,2\pi)},1)$ is bounded, and that $h_\omega\in C^{\infty}_{\mathrm{per}}(\revisions{[0,2\pi)})\subset H^{2}(\revisions{[0,2\pi)},1)$.} Moreover,
				\begin{align*}
					\mathcal{L}_\omega \overline{\phi}_n=\left(\frac{d}{dx}+i\omega\sin x\right)\frac{1}{\sqrt{2\pi}}e^{-inx}=\frac{\omega}{2}\overline{\phi_{n-1}}-in \overline{\phi_n}-\frac{\omega}{2}\overline{\phi_{n+1}}.
				\end{align*}
				Thus by Thm.~\ref{thm:recurrence_by_duality} we deduce that the Filon moments, $\sigma_n=I_\omega[\phi_n]$ must satisfy the following recurrence
				\begin{align}\label{eqn:Bessel_recurrence}
					-\frac{\omega}{2}\sigma_{n-1}-i n \sigma_n+\frac{\omega}{2}\sigma_{n+1}=0.
				\end{align}
				This recurrence provides a highly efficient way of computing the moments, and we have actually recovered a well-known relation: In the present case the moments can be expressed in terms of Bessel functions of the first kind, $\mathrm{J}_n$ (cf. the integral expression \cite[Eq.~9.1.21]{abramowitz1965handbook})
				\begin{align*}
					\sigma_n=\sqrt{2 \pi} e^{\frac{i \pi n}{2}} \mathrm{J}_{n}(\omega)
				\end{align*}
				and the recurrence \eqref{eqn:Bessel_recurrence} is equivalent to the Bessel recurrence satisfied by $\mathrm{J}_n$ \cite[Eq.~9.1.27]{abramowitz1965handbook}.
			\end{example}
			The next example concerns a case where, to the best of our knowledge, recurrences are not yet readily available in the literature:
			\begin{example}\label{example:Legendre_recurrences_quadratic_oscillator}
				For our second example we consider an integral over $[-1,1]$ with a quadratic oscillator,
				\begin{align*}
					I_\omega[f]:=\int_{-1}^1 e^{i\omega x^2}f(x)dx.
				\end{align*}
				One possible choice of interpolation basis is the use of Legendre polynomials (when $s=0$ in \eqref{eqn:interpolation_problem_p2}). This choice is guided by the idea that interpolating $f$ at Legendre points optimizes the order of the method when $\omega=0$, as described by \cite[\S 4.2.1]{deano2017computing}. Thus, we choose $\phi_n=\tilde{\Legendre}_n:=\sqrt{n+\frac12}\,\Legendre_n,\,\,\revisions{ n=0,1\dots,}$ where $\Legendre_n$ are Legendre polynomials with the standard normalisation $\Legendre_n(1)=1$ and $\tilde{\Legendre}_n$ are scaled such that they form a Hilbert basis for $L^2([-1,1])$. The oscillator $h_\omega(x)=\exp(i\omega x^2)$ satisfies
				\begin{align*}
					\mathcal{L}_\omega h_\omega=0,\quad \mathcal{L}_\omega=(x^2-1)\frac{d}{dx}-2xi\omega (x^2-1).
				\end{align*}
				\revisions{Let us confirm that $\mathcal{L}_\omega$ and $h_\omega$ satisfy the assumptions of Thm.~\ref{thm:recurrence_by_duality}.} We have chosen $\mathcal{L}_\omega$ specifically with the following two identities in mind \cite[Eqs. 22.8.5 \& 22.7.10]{abramowitz1965handbook}:
				\begin{align}\label{eqn:banded_identities_Legendre_polys1}
					\frac{x^2-1}{n}\frac{d}{dx}\Legendre_n(x)&=\frac{n+1}{2n+1}\Legendre_{n+1}(x)-\frac{n+1}{2n+1}\Legendre_{n-1}(x), \quad n\geq 1, \text{\ and\ }(x^2-1)\frac{d}{dx}\Legendre_0(x)=0,\\\label{eqn:banded_identities_Legendre_polys2}
					x\Legendre_n(x)&=\frac{n+1}{2n+1}\Legendre_{n+1}(x)+\frac{n}{2n+1}\Legendre_{n-1}(x), \quad n\geq 1, \text{\ and\ }x\Legendre_0(x)=\Legendre_{1}(x).
				\end{align}
				The identities \eqref{eqn:banded_identities_Legendre_polys1}-\eqref{eqn:banded_identities_Legendre_polys2} ensure that the action of $\mathcal{L}_\omega$ on the basis $\left\{\tilde{\Legendre}_n\right\}_{n=0}^\infty$ can indeed be represented by a banded infinite matrix. In fact, one may use \eqref{eqn:banded_identities_Legendre_polys1}-\eqref{eqn:banded_identities_Legendre_polys2} in an analogous way to find a suitable differential operator for any oscillators of the form $\exp(i\omega q(x))$ when $q(x)$ is a polynomial. \revisions{To understand the continuity properties of $\mathcal{L}_\omega$ we observe that for any function $f\in H^{4}([-1,1],1)$ we know there is a constant $C_f>0$ such that
					\begin{align*}
						\left|\left(f,\tilde{\Legendre}_n\right)_{L^{2}([-1,1],1)}\right|\leq C_f n^{-4}.
					\end{align*}
					Since $\sup_{x\in[-1,1]}|\tilde{\Legendre}_n|\leq \sqrt{n+1/2}$ we therefore find that the sum 
					\begin{align*}
						\sum_{n=0}^\infty \left(f,\tilde{\Legendre}_n\right)_{L^{2}([-1,1],1)}\tilde{\Legendre}_n(x)
					\end{align*}
					converges absolutely uniformly. Moreover we have the following identity for w, by boundedness of $\Legendre_n$ that 
					converges uniformly absolutely. From \eqref{eqn:banded_identities_Legendre_polys1} \& \eqref{eqn:banded_identities_Legendre_polys2} combined we have
					\begin{align*}
						\frac{d}{dx}\Legendre_{n+1}(x)=(n+1)\Legendre_n(x)+x\frac{d}{dx}\Legendre_n(x), \quad \frac{d}{dx}\Legendre_0(x)=0,
					\end{align*}
					whence it follows by induction $\sup_{x\in[-1,1]}|\tilde{\Legendre}_n|\leq \sqrt{n+1/2} n(n+1)/2$. This means also
					\begin{align*}
						\sum_{n=0}^\infty \left(f,\tilde{\Legendre}_n\right)_{L^{2}([-1,1],1)}\frac{d}{dx}\tilde{\Legendre}_n(x)
					\end{align*}
					converges absolutely uniformly and we thus have for any element $f\in H^{4}([-1,1],1)$:
					\begin{align*}
						\mathcal{L}_\omega f(x)=\sum_{n=0}^\infty \left(f,\tilde{\Legendre}_n\right)_{L^{2}([-1,1],1)}\mathcal{L}_\omega \tilde{\Legendre}_n(x).
					\end{align*}
					From \eqref{eqn:banded_identities_Legendre_polys1} \& \eqref{eqn:banded_identities_Legendre_polys2} we thus conclude that $\mathcal{L}_\omega:H^{4}([-1,1],1)\rightarrow H^{3}([-1,1],1)$ is bounded. Since $h_\omega\in C^{\infty}([-1,1])$ its Legendre coefficients decay faster than any polynomial and it immediately follows that $h_\omega\in H^{s}([-1,1],1)$ for all $s\geq 0$. Therefore the assumptions of Thm.~\ref{thm:recurrence_by_duality} are satisfied and the result} allows us to construct (after a few steps of algebra) the following recurrence,
				\begin{align*}
					\begin{split}
						&-\frac{2 i (n-2) (n-1) n \omega }{\sqrt{2 n-5} (2 n-3) (2 n-1) \sqrt{2 n+1}}\revisions{\sigma}_{n-3}+\frac{n \left(2 i \omega(n^2-3) +(2n-3)(2n+3)(n-1)\right)}{(2 n-3) \sqrt{2 n-1}
							\sqrt{2 n+1} (2 n+3)}\revisions{\sigma}_{n-1}\\
						&+\frac{(n+1) \left(2i\omega(n^2+2n-1)-(2n-1)(n+2)(2n+5)\right)}{(2 n-1) \sqrt{2 n+1} \sqrt{2 n+3} (2 n+5)}\revisions{\sigma}_{n+1}-\frac{2 i (n+1) (n+2) (n+3) \omega }{\sqrt{2 n+1} (2 n+3) (2 n+5) \sqrt{2 n+7}}\revisions{\sigma}_{n+3}=0
					\end{split}
				\end{align*}
				valid for $n\geq 3$, where the moments are $\sigma_n=I_\omega[\tilde{\Legendre}_n],\, n\geq 0$. Additionally, the first column of the matrix representation $B_{nm}$ of $\mathcal{L}_\omega$ gives rise to the extra condition
				\begin{align*}
					\sigma_4&=\revisions{\frac{\sqrt{5}}{3}}\frac{5(21 i+2\omega)}{24\omega}\sigma_{2}+\revisions{\frac{1}{3}}\frac{7}{12}\sigma_0,
				\end{align*}
				which means that the moments $\sigma_{2n}, n\geq 0,$ can be computed from just two initial conditions for which we have the following expressions:
				\begin{align*}
					\sigma_0&=\frac{e^{i\frac{\pi}{4}}}{\sqrt{2\omega}}\left(\gamma\left(\frac{1}{2},-i\omega\right)\right),\,
					\sigma_2=\sqrt{\frac{5}{2}}\left[\frac{3}{2}\frac{e^{i\omega}}{i\omega}-\left(\frac{3}{4i\omega}+\frac{1}{2}\right)\frac{e^{i\frac{\pi}{4}}}{\sqrt{\omega}}\left(\gamma\left(\frac{1}{2},-i\omega\right)\right)\right],
				\end{align*}
				where $\gamma(\cdot,\cdot)$ is the lower incomplete Gamma function \cite[Eq.~6.5.2]{abramowitz1965handbook}. Of course, for all moments of odd order, $\sigma_{2n+1}=0,\, n\geq 0,$ since $\exp(i\omega x^2)$ is an even function. We note that many efficient methods exist for computing the incomplete Gamma function \cite{cody1976,gautschi1979computational}, so the above expressions constitute a suitable way of initiating the recurrence.
			\end{example}
			%\begin{remark}
			%In order to give the reader confidence in the above expressions we have verified the recurrence numerically using Wolfram Mathematica 12. This applies to all non-trivial formulae in the present work.
			%\end{remark}
			
			In a similar spirit to Example \ref{example:Legendre_recurrences_quadratic_oscillator} one may choose other orthogonal polynomials as interpolation \revisions{bases} (and their zeros as corresponding interior nodes) in an attempt to maximize the \revisions{classical order of the quadrature when $\omega=0$ with the goal to ensure that the resulting Filon method has good convergence properties for all $\omega\geq 0$ (cf. \cite[\S4.2.1]{deano2017computing} and \cite{Gao2017b}).} In many cases one can use a similar approach to the above and exploit the three-term recurrence of orthogonal polynomials to extract recurrences for the moments in this manner.
			\subsection{Recursive moment computation for Filon--Clenshaw--Curtis methods}\label{sec:recursive_moment_computation_for_FCC}
			In the remainder of this paper we shall focus our attention to Filon--Clenshaw--Curtis methods, motivated by fast interpolation properties as described in \S\ref{sec:fast_interpolation_at_FCC_points}. We have seen in the previous section how one may find a recursion for the Filon quadrature moments when considering an integral of the form
			\begin{align*}
				I_\omega[\revisions{f}]=\int_{-1}^1f(x)h_\omega(x)\frac{dx}{\sqrt{1-x^2}},
			\end{align*}
			using an interpolation basis of normalised Chebyshev polynomials $\phi_n(x)=\sqrt{s_n}\,T_n(x),$ where $s_0=1/\pi,s_n=2/\pi,n\geq1$, are chosen such that$\left\{\sqrt{s_n}\,T_n\right\}_{n=0}^\infty$ forms a Hilbert basis for $L^2([-1,1],(1-x^2)^{-1/2})$. In Thm.~\ref{thm:recurrence_by_duality} we made the assumption that $h_\omega\in H^{s}([a,b],(1-x^2)^{-1/2})$ where $s\in\mathbb{N}$ is the order of the linear differential operator which maps $h_\omega$ to zero. It turns out that it can be desirable to relax this assumption. To see why, let us consider an integral of the form
			\begin{align*}
				\tilde{I}_\omega[f]=\int_{-1}^1f(x)\tilde{h}_\omega(x)dx,
			\end{align*}
			where to begin with we take $\tilde{h}_\omega(x)\in C^{1}([-1,1])$. The enormous speed up achieved by the use of Clenshaw--Curtis points in the Filon method makes those an excellent choice for interior interpolation points even when the weight function does not match (in a spirit similar to classical quadrature where Clenshaw--Curtis points can be preferable to optimal Legendre points as noted by \cite{trefethen2008gauss}). This means we need to compute the moments
			\begin{align*}
				\tilde{I}_\omega[\phi_n]=\sqrt{s_n}\int_{-1}^1T_n(x)\tilde{h}_\omega(x)dx.
			\end{align*}
			To ensure the interpolation basis is a Hilbert basis such that we can apply a methodology similar to Thm.~\ref{thm:recurrence_by_duality} it is thus appropriate to write $\tilde{I}_\omega[f]$ in the form
			\begin{align}\label{eqn:extracted_weight_function_FCC}
				\tilde{I}_\omega[\revisions{f}]=\int_{-1}^1f(x)h_\omega(x)\frac{dx}{\sqrt{1-x^2}},
			\end{align}
			where $h_\omega(x)=\sqrt{1-x^2}\tilde{h}_\omega(x)$. Clearly, the extra factor $\sqrt{1-x^2}$ weakens the regularity of $h_\omega(x)$, and we no longer expect $h_\omega(x)\in C^{1}([-1,1])$. This regularity was used in the proof of Thm.~\ref{thm:recurrence_by_duality} at the point where we showed that the formal adjoint defined component-wise by $\mathcal{L}_{\omega}^*\overline{\phi_n}=\sum_{m\in\mathcal{I}}\overline{B_{mn}}\,\overline{\phi_m}$ satisfied
			\begin{align}\label{eqn:moving_adjoint_to_other_side}
				\langle h_\omega,\mathcal{L}_{\omega}^*\overline{\phi_n}\rangle\rangle_{\revisions{t+s\times-s-t}}=\revisions{\langle\mathcal{L}_\omega h_\omega,\overline{\phi_n}\rangle_{\revisions{t\times-t}}}.
			\end{align}
			We can overcome this by constructing $\mathcal{L}_\omega$ in a suitable way: Let us follow the convention $T_{-n}(x)=T_n(x)$, then one can show using standard trigonometric identities:
			\begin{lemma}[\cite{abramowitz1965handbook}, Eqs. 22.7.4 \& 22.8.3]\label{lem:banded_operators_Chebyshev_polynomials}For all $n\geq \mathbb{Z}$:
				\begin{align*}
					xT_n(x)&=\frac{1}{2}T_{n-1}(x)+\frac{1}{2}T_{n+1}(x),\,\,\text{and}\,\,
					(1-x^2)T_n'(x)=\frac{n}{2}T_{n-1}(x)-\frac{n}{2}T_{n+1}(x).
				\end{align*}
				In particular, the actions of $x,(1-x^2)d/dx$ on $\left\{\sqrt{s_n}\,T_n\right\}_{n=0}^\infty$ are both banded, with bandwidth 3.
			\end{lemma}
			These operators ensure \eqref{eqn:moving_adjoint_to_other_side} holds even in cases when $\sqrt{1-x^2}h_\omega(x)$ is not sufficiently regular to satisfy the assumptions of Thm.~\ref{thm:recurrence_by_duality}. Indeed, consider the operator $\mathcal{L}=(1-x^2)d/dx$ and suppose $h_\omega$ is $C^1([a,b])$. Then by simple integration by parts we have:

			\begin{align}\begin{split}\label{eqn:general_integration_by_parts_to_take_adjoint}
					\int_{-1}^1&T_n(x)\mathcal{L}\left(\sqrt{1-x^2}h_\omega(x)\right)\frac{dx}{\sqrt{1-x^2}}=\lim_{\epsilon\rightarrow 0^+}\int_{-1+\epsilon}^{1-\epsilon}T_n(x)(1-x^2)\frac{d}{dx}\left(\sqrt{1-x^2}h_\omega(x)\right)\frac{dx}{\sqrt{1-x^2}}\\
					&=\lim_{\epsilon\rightarrow 0^+}\left[T_n(x)(1-x^2)h_\omega(x)\right]_{-1+\epsilon}^{1-\epsilon}-\lim_{\epsilon\rightarrow 0^+}\int_{-1+\epsilon}^{1-\epsilon}\frac{d}{dx}\left(\sqrt{1-x^2}T_n(x)\right)\left(\sqrt{1-x^2}h_\omega(x)\right)\frac{dx}{\sqrt{1-x^2}}\\
					&=\revisions{-}\int_{-1}^{1}\frac{d}{dx}\left(\sqrt{1-x^2}T_n(x)\right)\left(\sqrt{1-x^2}h_\omega(x)\right)\frac{dx}{\sqrt{1-x^2}}=\int_{-1}^{1}\mathcal{L}^* T_n(x)\left(\sqrt{1-x^2}h_\omega(x)\right)\frac{dx}{\sqrt{1-x^2}},
				\end{split}
			\end{align}
			meaning any operator formed as a combination of polynomial multiplication and $\mathcal{L}$ will still satisfy \eqref{eqn:moving_adjoint_to_other_side}. In fact the same integration by parts argument can be applied if $\tilde{h}_\omega$ satisfies an ordinary differential equation with polynomial coefficients that has a simple singularity in the interior of the domain. \revisions{This construction of $\mathcal{L}_\omega$ such that the boundary terms in integration-by-parts vanish reflects our observations from Remark~\ref{rmk:integration_by_parts}.} We shall demonstrate the principle on two types of integrals in greater detail: Integrals with stationary points/algebraic singularities in \S\ref{sec:application_to_integrals_with_singularities}, and integrals involving Hankel functions and hybrid numerical-asymptotic basis functions in \S\ref{sec:application_to_wave_scattering}.
			
			\section{Application to integrals with algebraic singularities and stationary points}\label{sec:application_to_integrals_with_singularities}
			Consider the case of algebraic singularities or stationary points at $x=0$, i.e. integrals of the form
			\begin{align*}
				I^{(1)}_\omega[f]&=\int_{-1}^{1}f(x)e^{i\omega x^r}dx,\,\, r\in\mathbb{N},\, r\geq2,\,\,\,\,\text{and}\,\,\,\,
				I^{(2)}_\omega[f]=\int_{-1}^{1}\mathrm{sgn}(x)|x|^\alpha e^{i\omega x}f(x)dx,\,\,\alpha\in(-1,1).
			\end{align*}
			By using the simple change of variable $y=x^r$ the integral $I^{(1)}_\omega[f]$ can be brought into the form $I^{(2)}_\omega[f]$. In fact, by the inverse function theorem, an integral with a general oscillator $h_\omega(x)=\exp(i\omega g(x))$ with $g^{(r)}(0)=0, g^{(r+1)}(0)\neq 0, g'(x)\neq 0,\,\forall x\neq 0,$ can also be brought into the above forms, by substituting $g(x)=y^r$, or, equivalently, as noted by \cite{olver_2007}, by choosing an interpolation basis that is in the span of $\{\mathrm{sgn}(x)g'(x)|g(x)|^{(n+1-r)/r}\}_{n=0}^{\nu+2s+1}$. These types of integrals were considered in the Filon context by \cite{olver_2007} and \cite{dominguez2013filon}. To illustrate the main ideas we focus on the integral $I_\omega^{(2)}[f]$. Here the natural basis described by \cite{olver2006,olver_2007} essentially reduces to a monomial interpolation basis, $x^n$, and the central observation is that its moments can be expressed explicitly \revisions{in terms of} the lower incomplete gamma function $\gamma$:
			\begin{align}\label{eqn:explicit_expression_standard_moments_algebraic_singularities}
				I_\omega^{(2)}[x^n]=\left(-i\omega\right)^{-1-n-\alpha}\left(\gamma(1+n+\alpha,-i\omega)\right)+\left(i\omega\right)^{-1-n-\alpha}\left(\gamma(1+n+\alpha,i\omega)\right), \quad n\geq 0.
			\end{align}
			This approach is particularly suitable when only a small number $\nu$ of interior interpolation points and hence moments are required. However, if we choose $\nu$ at moderate or large size relative to $s$ the cost of interpolating with standard polynomials increases rapidly. Yet resolving to fast interpolation at Clenshaw--Curtis points is seemingly prevented by the well-known exponential instability of computing $I^{(2)}_\omega[T_n]$ through directly expanding $T_n$ in terms of $x^n$. \cite{dominguez2013filon} approached the problem from a slightly different perspective, which also applies to integrals of the form $I^{(2)}_\omega[f]$, choosing a mesh $-1=x_{-M}<\cdots<x_{-1}<0<x_1<\cdots<x_M=1$ that is graded towards the algebraic singularity/stationary point at $x=0$. The method then evaluates the integrals using the classical non-singular Filon method on each subinterval $[x_{l},x_{l+1}], -M\leq l\leq -2,1\leq l\leq M-1$ and sets the approximation of the integral on $[x_{-1},x_{1}]$ equal to zero, when $-1<\alpha<0$. While this approach is quite flexible, as it avoids having to know the exact type of singularity at $x=0$, this flexibility comes at the price of asymptotic sub-optimality, because we know from the method of stationary phase, for any $-1<\alpha<0$, and $\epsilon>0$ fixed:
			\begin{align}\label{eqn:missmatch_asymptotic_behaviour_dominguez}
				\int_{-\epsilon}^{\epsilon}\text{sgn}(x)|x|^\alpha e^{i\omega x}f(x)dx\sim	I^{(2)}_\omega[f],\quad \text{as\ }\omega\rightarrow\infty,
			\end{align}
			i.e. we have asymptotic concentration near the singularity. Thus the method proposed by \cite{dominguez2013filon} leads, as $\omega$ increases, to an absolute error that is of the same size as the original integral (this is demonstrated in practical examples in \S\ref{sec:numerical_examples_algebraic_singularities}). \revisions{One can mitigate this effect by modifying the grading as $\omega$ increases, leading to so-called `adaptive Filon methods' which have been explored by \cite{gao2018adaptive}. However, as demonstrated in \cite{gao2018adaptive} (see also the discussion in \cite[\S 4.4]{deano2017computing}) this approach can be challenging because the grading needs to be carefully designed to accommodate the specific type of singularity at hand (\cite{gao2018adaptive} have so far only studied the case when $\alpha=0$ and $\alpha=-0.5$). Moreover, if the interior points are allowed to depend on $\omega$, the computational advantages of fast interpolation at special points (e.g. at Clenshaw--Curtis points) are entirely lost.}
			
			A resolution of \revisions{these} approaches can be found based on the methodology from Thm.~\ref{thm:recurrence_by_duality}: We can apply a direct version of the Filon method to integrals of the form $I^{(2)}_\omega[f]$ and still interpolate at Clenshaw--Curtis points as in \eqref{eqn:interpolation_problem_p2}, by computing the Chebyshev moments $I^{(2)}_\omega[T_n]$ accurately and efficiently using a recurrence initialised with exact expressions in \eqref{eqn:explicit_expression_standard_moments_algebraic_singularities}. Indeed, one can easily check using Watson's lemma \cite[pp. 263--265]{bender2013advanced} that if $q$ satisfies the interpolation conditions \eqref{eqn:interpolation_problem_p2}, and $f\in C^{s+2}([-1,1])$ then the direct Filon quadrature $\mathcal{Q}^{[\nu,s]}_\omega[f]:=I^{(2)}_\omega[q]$
			satisfies
			\begin{align}\label{eqn:asymptotic_error_algebraic_singularities}
				\left|I^{(2)}_\omega[f]-\mathcal{Q}^{[\nu,s]}_\omega[f]\right|=\mathcal{O}\left(\omega^{-(s+2)-\min\{0,\alpha\}}\right),\quad \omega\rightarrow\infty.
			\end{align}
			We note that an explicit derivation of \eqref{eqn:asymptotic_error_algebraic_singularities} shows that there is $C_{\alpha,s}>0$, dependent on $\alpha,s$, such that
			\begin{align}\label{eqn:Filon_error_as_interpolation_error_algebraic_singularities}
				\left|I^{(2)}_\omega[f]-\mathcal{Q}^{[\nu,s]}_\omega[f]\right|\leq C_{\alpha,s}\min_{0\leq j\leq s}\omega^{-j-2}\left(\omega^{-\alpha}\|f^{(j)}-q^{(j)}\|_{L^{\infty}([-1,1])}+\|f^{(j+1)}-q^{(j+1)}\|_{L^{\infty}([-1,1])}\right)
			\end{align}
			thus allowing us to account for the dependency on $\nu$ by studying the quality of interpolation of $f$ by $q$. This principle is the same as for a non-stationary oscillator as described by \cite[\S1.1]{melenk2010}. In the interest of brevity we omit the details of this derivation here, and instead refer to an analogous argument for a Hankel oscillator which we provide in Prop.~\ref{prop:filon_paradigm_for_alpha_zero} and Cor. \ref{cor:nu_explicit_error_estimates}. Having understood that the direct application of a Filon method truly matches the asymptotic behaviour of the integral it remains to compute the moments $\sigma_n=I^{(2)}_\omega[\phi_n]=I^{(2)}_\omega[\sqrt{s_n}\,T_n]$. In line with \eqref{eqn:extracted_weight_function_FCC} we write $h_\omega(x)=\sqrt{1-x^2}\mathrm{sgn}(x)|x|^\alpha \exp\left(i\omega x\right)$. We can now follow the recipe of Thm.~\ref{thm:recurrence_by_duality}: A suitable differential operator is given by
			\begin{align*}
				\mathcal{L}_\omega=x(1-x^2)\frac{d}{dx}+x^2-\alpha (1-x^2)-i\omega x(1-x^2).
			\end{align*}
			which is such that $\mathcal{L}_\omega h_\omega(x)=0$ pointwise for all $x\neq0$ (and including $x=0$ when $\alpha\geq0$). Moreover, in the same way as in \eqref{eqn:general_integration_by_parts_to_take_adjoint} we can use integration by parts (this time also excluding a small neighbourhood of $x=0$) to take the adjoint of $\mathcal{L}_\omega$, thus ensuring that \eqref{eqn:moving_adjoint_to_other_side}, and hence the conclusion of Thm.~\ref{thm:recurrence_by_duality}, hold.
			%\begin{align*}
			%	\int_{-1}^1T_n(x)x(1-x^2)&\frac{d}{dx}\left(\sqrt{1-x^2}\mathrm{sgn}(x)|x|^\alpha e^{i\omega x}\right)\frac{dx}{\sqrt{1-x^2}}\\
			%	&=\lim_{\epsilon\rightarrow 0}\left[\int_{-1+\epsilon}^{-\epsilon}+\int_{\epsilon}^{1-\epsilon}\right]x\sqrt{1-x^2}T_n(x)\frac{d}{dx}\left(\sqrt{1-x^2}\mathrm{sgn}(x)|x|^\alpha e^{i\omega x}\right)dx\\
			%	&=\lim_{\epsilon\rightarrow 0}\left[(1-x^2)T_n(x)\mathrm{sgn}(x)|x|^{\alpha+1} e^{i\omega x}\right]_{-1+\epsilon}^{-\epsilon}\\
			%	&\quad +\lim_{\epsilon\rightarrow 0}\left[(1-x^2)T_n(x)\mathrm{sgn}(x)|x|^{\alpha+1} e^{i\omega x}\right]_{-1+\epsilon}^{-\epsilon}\\
			%	&\quad -\lim_{\epsilon\rightarrow 0}\left[\int_{-1+\epsilon}^{-\epsilon}+\int_{\epsilon}^{1-\epsilon}\right](1-x^2)\frac{d}{dx}\left(x\sqrt{1-x^2}T_n(x)\right)\mathrm{sgn}(x)|x|^\alpha e^{i\omega x}\frac{dx}{\sqrt{1-x^2}}\\
			%	&=-\int_{-1}^1(1-x^2)\frac{d}{dx}\left(x\sqrt{1-x^2}T_n(x)\right)\mathrm{sgn}(x)|x|^\alpha e^{i\omega x}\frac{dx}{\sqrt{1-x^2}}.
			%\end{align*}
			After a few steps of algebra this results in the following recurrence, where for ease of notation we introduced $\tilde{\sigma}_n:=\sigma_n/\sqrt{s_n}$ and follow the convention $\sigma_{-n}=\sigma_n$:
			\begin{align}\label{eqn:recurrence_algebraic_singularities_chebyshev_moments}
				\tilde{\sigma}_{n-3}+\frac{2(-(n-3)+\alpha)}{	i\omega}\tilde{\sigma}_{n-2}-\tilde{\sigma}_{n-1}+\frac{4-4\alpha}{	i\omega}\tilde{\sigma}_n-\tilde{\sigma}_{n+1}+\frac{2(n+3+\alpha)}{i\omega}\tilde{\sigma}_{n+2}+ \tilde{\sigma}_{n+3}=0, \quad \forall n\in\mathbb{Z}.
			\end{align}
			This means in particular that
			%\questions{Do we need the first 3 conditions?!} This means in particular
			%\begin{align*}
			%	\frac{4-4\alpha}{	i\omega}\tilde{\sigma}_0-2\sigma_{1}+2\frac{2(3+\alpha)}{i\omega}\tilde{\sigma}_{2}+ 2\tilde{\sigma}_{3}&=0,\quad
			%	-\tilde{\sigma}_{0}+\frac{8}{	i\omega}\tilde{\sigma}_1+\frac{2(4+\alpha)}{i\omega}\tilde{\sigma}_{3}+ \tilde{\sigma}_{4}=0\\
			%	\frac{2\alpha}{	i\omega}\tilde{\sigma}_{0}+\frac{4-4\alpha}{	i\omega}\tilde{\sigma}_2-\tilde{\sigma}_{3}+\frac{2(5+\alpha)}{i\omega}\tilde{\sigma}_{4}+ \tilde{\sigma}_{5}&=0
			%\end{align*}
			the initial values $\tilde{\sigma}_0,\tilde{\sigma}_1=\tilde{\sigma}_{-1},\tilde{\sigma}_2=\tilde{\sigma}_{-2}$ are sufficient in order to compute the moments using \eqref{eqn:recurrence_algebraic_singularities_chebyshev_moments}. For those we have the explicit expressions in terms of the lower incomplete gamma function based on \eqref{eqn:explicit_expression_standard_moments_algebraic_singularities}:
			\begin{align*}
				\tilde{\sigma}_0&=I^{(2)}_\omega[\revisions{1}],\quad \tilde{\sigma}_1=I^{(2)}_\omega[\revisions{x}],\quad \tilde{\sigma}_1=2I^{(2)}_\omega[x^2]-I^{(2)}_\omega[\revisions{1}].
			\end{align*}
			
			\subsection{Stability analysis of the recurrences}\label{sec:stability_analysis_of_the_recurrences}
			In this section we seek to understand the stability of moment recurrences for the integrals $I_\omega^{(1)},I_\omega^{(2)}$. Similar to work by \cite{piessens1983modified} and \cite{Dominguez2011} we find there is a balance between $N$, the number of required moments, and $\omega$ which results in two regions of different behaviour: 
			\begin{itemize}
				\item \textbf{The initial regime, when $\bm{N\ll\omega}$:} Here the recurrences lead at worst to algebraic instabilities, which are moderate relative to the decay of the interpolation coefficients for sufficiently smooth $f$. We provide rigorous results for two cases of interest in \S\ref{sec:initial_stability_results_algebraic_singularities} below.
				%Comment that the results for this part are related to Bessels recurrence!
				\item \textbf{The tail of the recurrences, when $\bm{N\gg\omega}$:} This is mostly of theoretical interest, since in practice, if we require $N\sim \omega$ moments, Gaussian quadrature applied to the full integral will be of comparable cost and error to the Filon method thus there is no necessity to resort to Filon methods any longer. Nevertheless, in \S\ref{sec:tail_stability_algebraic_singularities} we provide an indication of the behaviour of the recurrences in this regime which can be used in practice to compute them stably using Oliver's algorithm \cite{oliver1968numerical}. \revisions{In order to apply Oliver's algorithm typically an approximation of moments for large indices is required, i.e. approximate values for $\tilde{\sigma}_{N},\dots,\tilde{\sigma}_{N+M}$ for some $N\gg \omega, M\in\mathbb{N}$. These can be approximated using the asymptotic expansion of $\tilde{\sigma}_N$ as $N\rightarrow\infty$ for fixed $\omega$ which is obtained from the method of stationary phase \cite{bender2013advanced}. This was done elegantly in the case of an exponential oscillator without stationary point by \cite{Dominguez2011}.} 
			\end{itemize}
			\subsection{Stability results for the initial regime}\label{sec:initial_stability_results_algebraic_singularities}
			We note, firstly, that for a linear oscillator, i.e. integrals of the form $\int_{-1}^1f(x)e^{i\omega x}dx$, it was shown by \cite{Dominguez2011} that \revisions{an equivalent} recurrence to the one found through the application of Thm.~\ref{thm:recurrence_by_duality} is algebraically stable for $n<\omega$. Our first stability result shows that a very similar analysis can be applied to the case of a simple stationary point at $x=0$, i.e. integrals of the form $I_\omega[f]=\int_{-1}^1f(x)e^{i\omega x^2}dx$.
			Although this case is covered by the recurrence \eqref{eqn:recurrence_algebraic_singularities_chebyshev_moments} we can find a simpler version by noting that $\mathcal{L}_\omega h_\omega=0$ where
			\begin{align*}
				h_\omega(x)=\sqrt{1-x^2}e^{i\omega x^2},\quad \mathcal{L}_\omega=(1-x^2)\frac{d}{dx}-i\omega2x(1-x^2)+x.
			\end{align*}
			The resulting recurrence satisfied by the moments is (again writing $\tilde{\sigma}_n=\sigma_n/\sqrt{s_n}$ and $\tilde{\sigma}_{-n}=\tilde{\sigma}_n$):
			\begin{align}\label{eqn:recurrence_quadratic_oscillator_chebyshev_moments1}
				\tilde{\sigma}_{n-3}+\left(-1-\frac{2(n-2)}{i\omega}\right)\tilde{\sigma}_{n-1}+\left(-1+\frac{2(n+2)}{i\omega}\right)\tilde{\sigma}_{n+1}+\tilde{\sigma}_{n+3}&=0,\quad n\in\mathbb{Z}%\geq 3\\\label{eqn:recurrence_quadratic_oscillator_chebyshev_moments2}
				%\left(-1-\frac{-2}{i\omega}\right)\tilde{\sigma}_{0}+\frac{6}{i\omega}\tilde{\sigma}_{2}+\tilde{\sigma}_{4}&=0
			\end{align}
			and we see like in example \ref{example:Legendre_recurrences_quadratic_oscillator}, by the symmetry of the kernel, $\tilde{\sigma}_{2n+1}=0,\,n\geq 0$. We also have the following initial conditions in terms of the lower incomplete Gamma function:% \questions{[We might have to rescale the initial value of $\sigma_0$]}
			\begin{align*}
				\tilde{\sigma}_0&=\frac{e^{i\frac{\pi}{4}}}{\sqrt{\omega}}\left(\gamma\left(\frac{1}{2},-i\omega\right)\right),\quad \tilde{\sigma}_2=2\frac{e^{i\omega}}{i\omega}-\left(1+\frac{1}{i\omega}\right)\frac{e^{i\frac{\pi}{4}}}{\sqrt{\omega}}\left(\gamma\left(\frac{1}{2},-i\omega\right)\right).
			\end{align*}
			\begin{theorem}\label{thm:full_stability_quadratic_oscillator}
				Suppose the moments $\check{\tilde{\sigma}}_n$ are computed using \eqref{eqn:recurrence_quadratic_oscillator_chebyshev_moments1} with slightly perturbed initial conditions: $\check{\tilde{\sigma}}_0=\tilde{\sigma}_0+\epsilon_0$, $\check{\tilde{\sigma}}_2=\tilde{\sigma}_2+\epsilon_2$, for some $|\epsilon_0|,|\epsilon_2|<\epsilon$. Then, for any $n$ with $2n+1<\omega$,
				\begin{align*}
					|\check{\tilde{\sigma}}_{2n}-\tilde{\sigma}_{2n}|< \frac{8n\omega^{\frac{1}{2}}}{3\left(\omega^2-(2n+1)^2\right)^{\frac{1}{4}}}\left(2+\frac{1}{\omega}\right)\epsilon.
				\end{align*}
			\end{theorem}
			\begin{proof} The proof of this result is similar to \cite[Thm.~5.1]{Dominguez2011} and can be found in Appendix B.
			\end{proof}
			This result tells us that we can reliably compute the first $N$ moments from two initial conditions using \eqref{eqn:recurrence_quadratic_oscillator_chebyshev_moments1}, and, as long as $N<C\omega$ for some constant $C<1$, any error in initial conditions grows no faster than linearly in $N$.
			
			We now consider the recurrence \eqref{eqn:recurrence_algebraic_singularities_chebyshev_moments} for a general value of $\alpha>-1$: We can still guarantee at worst linear growth of initial perturbations, but this time our rigorous analysis applies to the slightly narrower regime $N+1<\min\{C\sqrt{\omega},\omega\}$ for some $C>0$.
			\begin{theorem}\label{thm:square_root_stability_algebraic_singularity}
				Suppose the moments $\check{\tilde{\sigma}}_n$ are computed using \eqref{eqn:recurrence_algebraic_singularities_chebyshev_moments}
				%\begin{align}\label{eqn:recurrence_algebraic_singularity_chebyshev_moments}
				%	\check{\tilde{\sigma}}_{n-3}+\frac{2(-(n-3)+\alpha)}{	i\omega}\check{\tilde{\sigma}}_{n-2}-\check{\tilde{\sigma}}_{n-1}+\frac{4-4\alpha}{	i\omega}\check{\tilde{\sigma}}_n-\check{\tilde{\sigma}}_{n+1}+\frac{2(n+3+\alpha)}{i\omega}\check{\tilde{\sigma}}_{n+2}+ \check{\tilde{\sigma}}_{n+3}=0,\quad n\geq 0,
				%\end{align}
				with the perturbed initial conditions $\check{\tilde{\sigma}}_0=\tilde{\sigma}+\epsilon_0$, $\check{\tilde{\sigma}}_1=\check{\tilde{\sigma}}_{-1}=\tilde{\sigma}_1+\epsilon_1$, $\check{\tilde{\sigma}}_2=\check{\tilde{\sigma}}_{-2}=\tilde{\sigma}_2+\epsilon_2$, $|\epsilon_j|<\epsilon$ for some $\epsilon>0$, and assume $\check{\tilde{\sigma}}_3=\check{\tilde{\sigma}}_{-3}$. Then, whenever $	n+1<\min\{C\sqrt{\omega},\omega\}$ for a given $C>0$, we have
				\begin{align*}
					|\check{\tilde{\sigma}}_n-\tilde{\sigma}_n|\leq \frac{(K_0+nK_1)}{2}\epsilon \left(\frac{K_2\omega^{\frac{1}{2}}}{K_2\omega^{\frac{1}{2}}-1}\exp\left(\frac{C}{K_2-\omega^{-\frac{1}{2}}}\right)+1\right)
				\end{align*}
				where the constants $K_0,K_1,K_2$ are independent of $n$ and are given by 
				\begin{align*}
					K_0= \frac{2\sqrt{\omega}}{\sqrt{\omega-C^2}},\quad K_1=\frac{\omega+2+|\alpha|}{\sqrt{\omega^2-C^2\omega}},\quad K_2=\frac{\left(\omega-C^2\right)^{\frac{1}{4}}}{\omega^{\frac{1}{4}}\sqrt{2|\alpha|+2}}.
				\end{align*}
			\end{theorem}
			\begin{proof} The proof of this result is given in Appendix C.
			\end{proof}
			Note that the above constants have simple limits as $\omega\rightarrow\infty$ which means that for $\omega$ sufficiently large we can simplify the upper bound:
			\begin{corollary} For any $\delta>0,\,C>0$, there is $\omega_0>0$ such that whenever the assumptions of Thm.~\ref{thm:square_root_stability_algebraic_singularity} are satisfied, and $\omega\geq \omega_0$, the error is bounded above by
				\begin{align*}
					|\check{\tilde{\sigma}}_n-\tilde{\sigma}_n|\leq \frac{((2+\delta)+(1+\delta)n)}{2}\epsilon \left(\exp\left(C\sqrt{2|\alpha|+2}(1+\delta)\right)+1\right).
				\end{align*}
			\end{corollary}
			\revisions{\begin{remark}The stability results in Thms.~\ref{thm:full_stability_quadratic_oscillator} \& \ref{thm:square_root_stability_algebraic_singularity} assume that the additions and multiplications used to compute the moments recursively from the \textit{homogeneous} equations \eqref{eqn:recurrence_algebraic_singularities_chebyshev_moments} \& \eqref{eqn:recurrence_quadratic_oscillator_chebyshev_moments1} are done exactly. This means the results describe the stability of the recurrences with respect to perturbations in the initial conditions, i.e. an error in the approximation of the initial moments. This analysis can be modified to additionally account for errors in floating point arithmetic that may occur in the additions and multiplications performed at each recursive step, by using a discrete variation of constants argument analogous to \eqref{eqn:discrete_variation_of_constants_quadratic_oscillator} in the proof of Thm.~\ref{thm:full_stability_quadratic_oscillator}. A similar argument was employed by \cite{Dominguez2011} to study the stability of an \textit{inhomogeneous} recurrence for moments in the Filon--Clenshaw--Curtis method.		\end{remark}}
			\subsection{Change of behaviour of homogeneous solutions and Oliver's algorithm}\label{sec:tail_stability_algebraic_singularities}
			The above results suggest that, as $n$ increases, there will be change in the behaviour of homogeneous solutions to \eqref{eqn:recurrence_algebraic_singularities_chebyshev_moments} and \eqref{eqn:recurrence_quadratic_oscillator_chebyshev_moments1} and that, for sufficiently large $n$, some of the solutions will exhibit super-algebraic growth. Understanding when exactly this transition occurs for general recurrences with non-constant coefficients is an open problem, however, based on numerical experiments, we find that the following heuristic argument provides a reasonably accurate practical indication of the location of this change of behaviour: Our Ansatz is that the change from algebraic to super-algebraic regime occurs when $n\propto\omega$. Thus we let $n=C_{n,\omega}\omega$ in \eqref{eqn:recurrence_algebraic_singularities_chebyshev_moments}, set $\lambda=\sigma_{n+1}/\sigma_{n}$ and we assume that for $-3\leq j\leq 3$:
			\begin{align*}
				\sigma_{n+j}/\sigma_{n}\sim\lambda^j \quad\text{as\ }n\rightarrow\infty.
			\end{align*}
			Plugging into \eqref{eqn:recurrence_quadratic_oscillator_chebyshev_moments1} and matching the leading order terms in $n$ yields the condition
			\begin{align*}
				\lambda^{-3}+\left(-1+2iC_{n,\omega}\right)\lambda^{-1}+\left(-1-2iC_{n,\omega}\right)\lambda+\lambda^3=0,
			\end{align*}
			which has solutions $\lambda=\pm1,\pm\sqrt{iC_{n,\omega}\pm\sqrt{1-C_{n,\omega}^2}}$. Thus these solutions for $\lambda$ are in modulus no larger than one if and only if $n/\omega=C_{n,\omega}<1$. This prediction matches our rigorous result in Thm.~\ref{thm:full_stability_quadratic_oscillator} which showed no larger than linear growth in that regime. For $n/\omega=C_{n,\omega}>1$ two of those values of $\lambda$ are in modulus greater than one, thus indicating that there may be two out of six linearly independent solutions to \eqref{eqn:recurrence_quadratic_oscillator_chebyshev_moments1} that exhibit super-algebraic growth in this regime.
			
			A similar heuristic argument can be applied to \eqref{eqn:recurrence_algebraic_singularities_chebyshev_moments}, which reduces, \revisions{after matching the leading order terms in $n$,} to the condition
			\begin{align}\label{eqn:heuristic_argument_for_algebraic_singularities}
				\revisions{\lambda^{-3}+2iC_{\omega,n}\lambda^{-2}-\lambda^{-1}-\lambda-2iC_{\omega,n}\lambda^{2}+\lambda^3=0.}
			\end{align}
			The solutions are now $\lambda=i^j, j=1,\dots 4,$ and $\lambda=i C_{n,\omega}\pm\sqrt{1-C_{n,\omega}^2}$. This means we expect algebraic behaviour in the regime $n/\omega=C_{n,\omega}<1$, which suggests that the results in Thm.~\ref{thm:square_root_stability_algebraic_singularity} might extend to larger values of $n$ than we are currently able to prove. This near-linear growth until $n\approx \omega$ is indeed observed in practice as we show in Fig.~\ref{fig:norm_growth_algebraic_singularities}. Finally, when $n/\omega=C_{n,\omega}>1$ one of the solutions for $\lambda$ is in modulus greater than one, which indicates that we might expect to have one out of six linearly independent solutions exhibiting super-algebraic growth in this regime. Of course, the moments $\sigma_n$ decay algebraically as $n\rightarrow\infty$ for any fixed $\omega$. This suggests that the tail (i.e. moments with $n>\omega$) can be computed stably using Oliver's algorithm \cite{oliver1968numerical} with five initial and one endpoint value, \revisions{the latter of which can be approximated by an asymptotic expression for $\sigma_n$ as $n\rightarrow\infty$ (as was done for a linear oscillator by \cite{Dominguez2011})}. Numerical experiments support this observation, but in the interest of brevity those are omitted from the present work. Instead we shall provide numerical evidence that supports the above argument of algebraic stability for \eqref{eqn:recurrence_algebraic_singularities_chebyshev_moments} when $n<C\omega$ for some $C<1$.
			
			\subsection{Numerical examples and comparison to previous work}\label{sec:numerical_examples_algebraic_singularities}
			One way to verify this numerically is by expressing the recurrence \eqref{eqn:recurrence_algebraic_singularities_chebyshev_moments} in the equivalent form 
			\begin{align}\label{eqn:algebraic_recurrence_as_matrix_product}
				\bm{x}_{N+1}=A_N\bm{x}_N=A_N A_{N-1}\cdots A_0\bm{x}_0=:\prod_{n=1}^NA_n\bm{x}_1, \quad\forall N\in\mathbb{Z},
			\end{align}
			where $\bm{x}_n=(x_{n+2},\dots,x_{n-3})^T$ and $A_n,\, n\geq 0,$ are $6\times 6$ matrices given by
			\begin{align}\label{eqn:matrix_form_algebraic_recurrence_as_matrix_product}
				A_n=\begin{pmatrix}
					-\frac{2(n+3+\alpha)}{i\omega}&1&\frac{4\alpha-4}{i\omega}&1&\frac{2(n-3)-\alpha}{i\omega}&-1\\
					&&\bm{I_{5}}&&&0
				\end{pmatrix},
			\end{align}
			where $\bm{I_5}$ is the $5\times5$ identity matrix. The growth of an arbitrary homogeneous solution to \eqref{eqn:recurrence_algebraic_singularities_chebyshev_moments} with given initial conditions $x_{-2},\dots,x_{3}$ is then bounded above by the norm $\left\|\prod_{n=1}^NA_n\right\|\|\bm{x}_1\|$. In Figure \ref{fig:norm_growth_algebraic_singularities} we plot this quantity for various values of $\omega$ and $\alpha$. As guaranteed by Thm.~\ref{thm:square_root_stability_algebraic_singularity}, the initial regime exhibits only linear growth, which changes to super-algebraic growth near $n\approx\omega$ as has been predicted by the heuristic argument \eqref{eqn:heuristic_argument_for_algebraic_singularities}.
			
			\begin{figure}[h!]\vspace{-0.25cm}
				\centering
				\begin{subfigure}[h]{0.49\linewidth}
					\centering
					\includegraphics[width=0.95\linewidth]{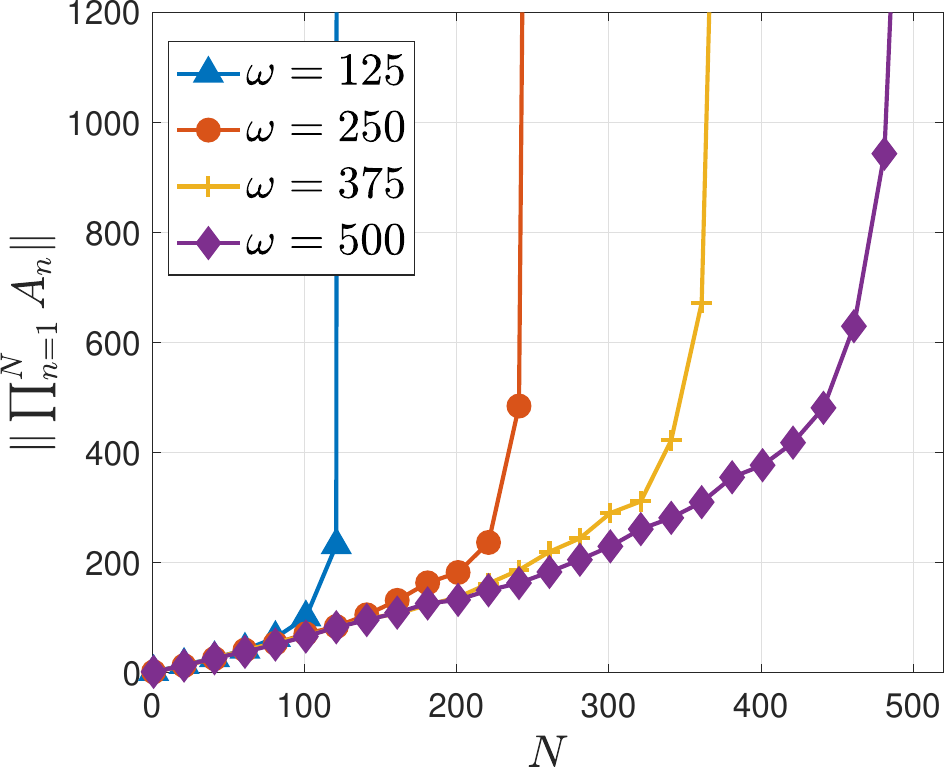}
					\caption{$\alpha=-0.2$}
					\label{fig:normgrowthreal1}
				\end{subfigure}%
				\begin{subfigure}[h]{0.49\linewidth}
					\centering
					\includegraphics[width=0.95\linewidth]{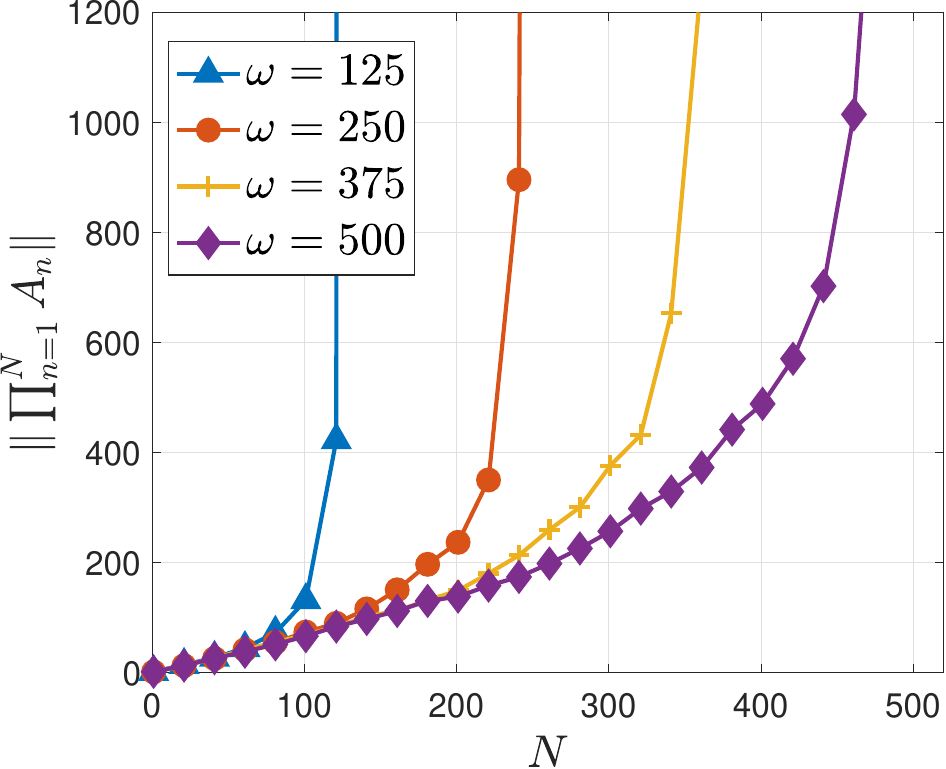}
					\caption{$\alpha=-0.7$}
					\label{fig:normgrowthreal2}
				\end{subfigure}
				\caption{The growth of solutions to \eqref{eqn:recurrence_algebraic_singularities_chebyshev_moments} as measured by $\left\|\prod_{n=1}^NA_j\right\|$. The initial linear growth transitions to super-algebraic growth nearby $n\approx \omega$ as predicted heuristically in \eqref{eqn:heuristic_argument_for_algebraic_singularities}.}
				\label{fig:norm_growth_algebraic_singularities}
			\end{figure}\newpage
			
			In our second numerical example we evaluate the practical performance of the direct Filon method with recursive moment computation for $I_\omega^{(2)}$. In this example we let $\alpha=-0.25$ and $f(x)=x/(1+x^2)+1/(1+x^4)$, i.e. we approximate the integral
			\begin{align*}
				I_\omega^{(2)}[f]=\int_{-1}^1 \mathrm{sgn}(x)|x|^{-0.25}e^{i\omega x}\left(\frac{x}{1+x^2}+\frac{1}{1+x^4}\right)dx.
			\end{align*}
			We begin by considering the absolute error of the direct Filon method
			\begin{align*}
				\mathcal{E}_{[s,\nu]}[f]&=\left|I^{(2)}_{\omega}[f]-\mathcal{Q}_\omega^{[s,\nu]}[f]\right|, \text{\ where\ }\mathcal{Q}_\omega^{[s,\nu]}[f]=I^{(2)}_\omega[q], \text{\ and\ }q\text{\ satisfies \eqref{eqn:interpolation_problem_p2}} .
			\end{align*}
			Figure \ref{fig:asymptotic_order_our_filon} shows this absolute error for the range $\omega\in[16,400]$ with fixed values of $\nu=5,s=0,1,2$. The black dash-dotted curves correspond to the asymptotic orders $\Ob(\omega^{-k-2+0.25})$ and confirm \eqref{eqn:asymptotic_error_algebraic_singularities}.% The error of the Filon methods remain small both for large as well as for moderate frequencies $\omega$.
			
			In Fig.~\ref{fig:cpu_time_comparison} we compare the efficiency of the direct method (for $s=0$) with the composite Filon method described by \cite{dominguez2013filon} and with a simple graded Clenshaw--Curtis approach. Both of the latter methods define a mesh that is graded towards the singularity at $x=0$ with $x_j=\mathrm{sgn}(j)|j/M|^{r}$ for $j=-M,\dots, M,$ and on each of the intervals $[x_j,x_{j+1}], j=-M,\dots,-2,1,\dots,M,$ the integral is approximated by the classical Filon method as in \cite{Dominguez2011} for the composite Filon method, and by Clenshaw--Curtis quadrature in the graded Clenshaw--Curtis case. These `sub-methods' on each $[x_j,x_{j+1}]$ come with an additional parameter $\nu$, where $\nu+2$ is the number of quadrature points on $[x_j,x_{j+1}]$ analogously to \eqref{eqn:extended_Filon_interpolation_problem} with $s=0$. For both methods the integral on $[x_{-1},x_1]$ is approximated by zero and the contributions are summed to provide an overall approximation to $I_\omega^{(2)}[f]$.
			
			In the figure we compare the minimum CPU time each of the methods required in order to compute the integral to a fixed relative error of $10^{-7}$. According to \cite{dominguez2013filon} if we choose $r>(\nu+2)/(1+\alpha)$ the composite Filon method converges as $M\rightarrow\infty$. Thus we fix $\nu=4, r=8.1$ and proceed by increasing $M$ from $M=10$ until we reach the desired relative error with a certain choice $M_0(\omega)$. The CPU time plotted in Fig.~\ref{fig:cpu_time_comparison} is the time the method took to evaluate the integral with the fixed setting $M=M_0(\omega)$. We repeat the process for each frequency $\omega$ and proceed similarly for the graded Clenshaw--Curtis method. According to \eqref{eqn:Filon_error_as_interpolation_error_algebraic_singularities} the direct Filon--Clenshaw--Curtis method converges as $\nu\rightarrow\infty$ and so for this case we start with $\nu=3$ and proceed by increasing $\nu$ until we achieve the desired relative error with some $\nu_0(\omega)$ before plotting the CPU time it took to evaluate the method with $\nu=\nu_0(\omega)$ and repeating the process for each frequency.
			
			\begin{figure}[h!]
				\centering\vspace{-0.25cm}
				\begin{subfigure}[h]{0.49\linewidth}
					\centering
					\includegraphics[width=0.95\textwidth]{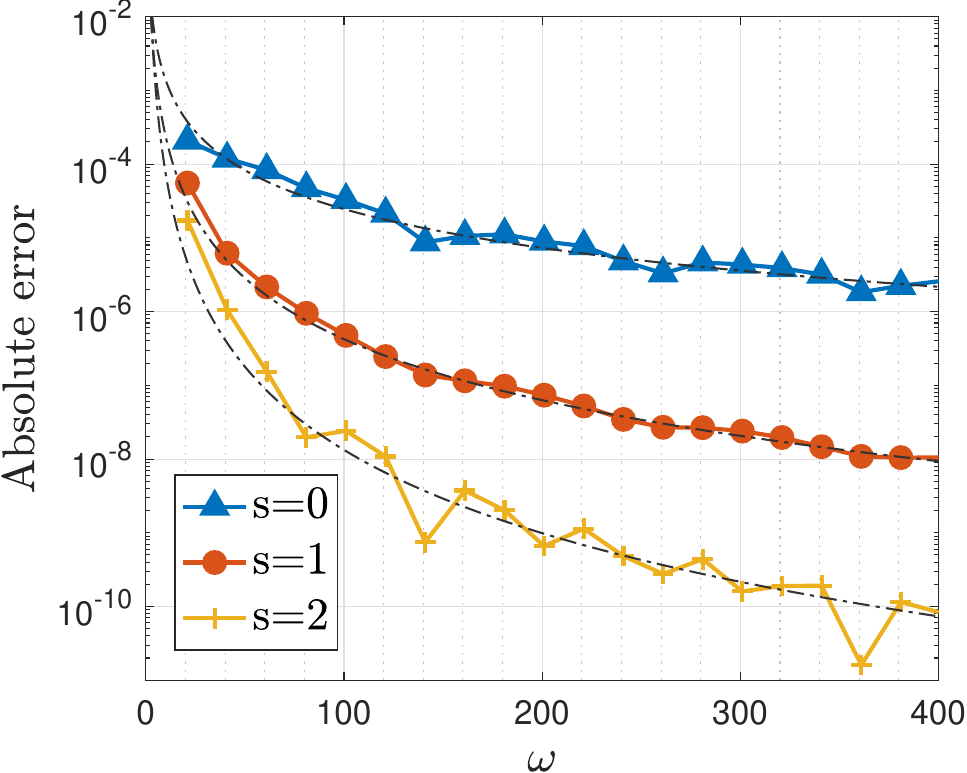}
					\caption{$\mathcal{E}_{[s,\nu]}[f]$, \textcolor{black}{for a constant number of interior\\ evaluations $\nu=5$.\hspace{-1cm}}}
					\label{fig:asymptotic_order_our_filon}
				\end{subfigure}%
				\begin{subfigure}[h]{0.49\linewidth}
					\centering
					\includegraphics[width=0.95\textwidth]{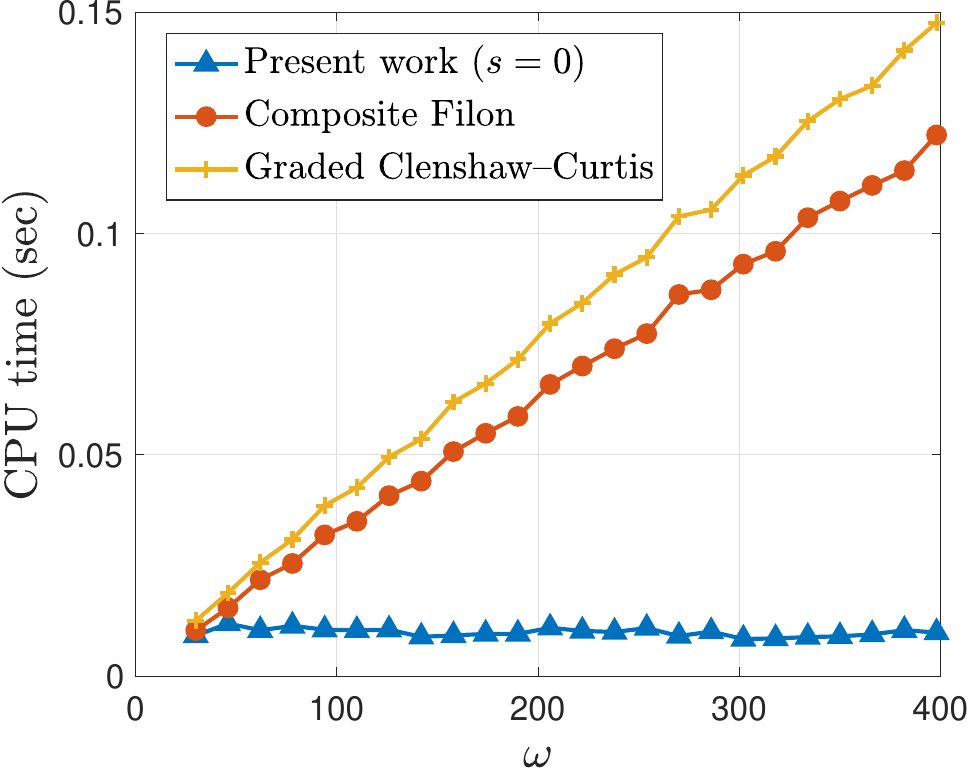}
					\caption{Min. CPU time to achieve relative error $\leq10^{-7}$.\\\,}
					\label{fig:cpu_time_comparison}
				\end{subfigure}
				\caption{Comparison of our direct Filon method with recursive moment computation, with the literature. On the left we show the absolute error for our present method, and on the right a comparison of the minimum CPU time necessary to achieve a fixed relative accuracy using our present work in comparison to the composite Filon method \cite{Dominguez2011} and a graded Clenshaw--Curtis method.}
				%\label{fig:asymptotic_order_our_filon}
			\end{figure}
			
			While this example is certainly no complete parametric study of the convergence properties of all three methods, and especially the absolute value of the timings depends significantly on the specific implementation and CPU used (all experiments here were performed on a single core of an Intel Core i5-8500 CPU), the point to take away is the overall clear trend in the cost as $\omega$ increases: As expected the classical graded Clenshaw--Curtis method requires the fastest, linear, increase in cost. While the composite Filon method performs better, it is, by construction as explained in \eqref{eqn:missmatch_asymptotic_behaviour_dominguez}, still required to increase the cost with frequency, since it does not fully match the asymptotic behaviour of the integral near the singularity. \revisions{The behaviour of this method could be improved if we were to change the grading as $\omega$ increases, but this has to be done in a very specific manner depending on the value of $\alpha$ and at present the available literature considers only the cases $\alpha=0,-0.5$ \cite[]{gao2018adaptive}. We refer the reader to \cite[\S4.4]{deano2017computing} for a detailed discussion on the advantages and disadvantages of adaptive Filon methods.}
			Finally, the direct application of the Filon method \revisions{with recursive moment computation} achieves the approximation, as expected, a\revisions{t} completely frequency-independent cost \revisions{and indeed computes the approximation essentially at cost $\mathcal{O}(\nu\log\nu)$ independently of $\omega$}.
			
			In all of the above numerical examples the reference solution for the true integral was computed with a graded Clenshaw--Curtis method with $M=3000,\nu=10,r=40$.
			
			\section{Application to high-frequency wave scattering}\label{sec:application_to_wave_scattering}

			As a final application of our method for recursive moment computation we consider integrals arising in hybrid numerical-asymptotic methods for high-frequency wave-scattering on a screen \revisions{in two dimensions. For an introduction of the relevant formulation of these types of wave scattering problems we refer the reader to  \cite{chandler2012numerical}, \cite{hewett2014}, \cite{gibbs2019fast}, and \cite{maierhoferthesis2021}.} As will be explained in more detail in \S\ref{sec:numerical_examples_scattering_on_screen} the integrals of relevance in this context are of the form
			\begin{align}\label{eqn:form_of_integrals_linear_hankel_kernel}
				I^{(3)}_{\omega,\beta}[f]&=2\int_{0}^{1}f\left(2x-1\right)H_0^{(1)}(\omega x)e^{i\omega\beta x}dx=\int_{-1}^{1}f(x)H_0^{(1)}\left(\omega \frac{x+1}{2}\right)e^{i\omega\beta (x+1)/2}dx,
			\end{align}
			where $\beta\in\mathbb{R},\,\beta\neq -1,$ and $H_0^{(1)}$ is the Hankel function of first kind and order zero \cite[Eq.~9.1.3]{abramowitz1965handbook}. As a first step in constructing a suitable direct Filon method we aim to understand the asymptotic properties of the integral $I^{(3)}_{\omega,\beta}$. In order to do so we recall the following property of $H_0^{(1)}$:
			\begin{lemma}[Phase extraction of $H_0^{(1)}$, see Lemma 4.6 in \cite{chandler2012numerical}]\label{lem:phase_extraction_in_H_0}
				Let $h_0(z):=\exp(-iz)H_0^{(1)}(z)$, then for each $n\geq 0$ there is a constant $C_n$ such that
				\begin{align*}
					\left|\frac{\D^n}{\D z^n}h_0(z)\right|\leq C_n\begin{cases}
						\max\left\{1+\log(1/z),z^{-n}\right\},&z\in(0,1]\\
						z^{-(n+1/2)},&z\in[1,\infty).
					\end{cases}
				\end{align*}
			\end{lemma}
			\ \newline
			With this control on the oscillations in $h_0(\omega x)$ we can proceed to show:
			\begin{proposition}[Filon paradigm for $I^{(3)}_{\omega,\beta}$]\label{prop:filon_paradigm_for_alpha_zero} For any $k\in\mathbb{N}$ there is a constant $C_k>0$ such that for all $\beta\in\mathbb{R},\beta\neq-1,\omega\geq 1$ and any function $\tilde{f}\in C^{k+2}[0,1]$ with $\tilde{f}^{(j)}(\pm1)=0$ for $j=0,\dots,k$:
				\begin{align*}
					\left|I^{(3)}_{\omega,\beta}[\tilde{f}]\right|\leq C_k\left( \omega^{-(k+2)}\|\tilde{f}^{(k+1)}\|_{\infty}\frac{|\beta+1|^{k+3}-1}{|\beta+1|-1}+\omega^{-(k+2)}\log\omega\|\tilde{f}^{(k+2)}\|_{\infty}|\beta+1|^{-(k+2)}\right).
				\end{align*}
			\end{proposition}
			\begin{proof}The proof is given in Appendix E.
			\end{proof}
			This means that the FCC rule $\mathcal{Q}^{[s,\nu]}_\omega[f]:=I^{(3)}_{\omega,\beta}[p]$ with $p$ satisfying \eqref{eqn:extended_Filon_interpolation_problem} has the asymptotic error
			\begin{align*}
				\left|I^{(3)}_{\omega,\beta}[f]-\mathcal{Q}^{[s,\nu]}_\omega[f]\right|=\mathcal{O}(\omega^{-(k+2)}\log\omega),\quad \omega\rightarrow\infty.
			\end{align*}
			Moreover Prop.~\ref{prop:filon_paradigm_for_alpha_zero} allows us to understand the $\nu$-dependency of the quadrature error through the quality of approximation of $f$ by the interpolating polynomial $p$. There are a number of ways to estimate $\|f^{(j)}-p^{(j)}\|_\infty$: one possibility is via the Hermite interpolation formula as was done in the Filon context for non-stationary oscillators by \cite{melenk2010}, another is to relate the error to the regularity of $f$ in periodic Sobolev norms on $[0,2\pi]$ via the change of variable $x=\cos\theta$, this approach was taken for linear oscillators by \cite{Dominguez2011}. Finally, in our opinion, a very elegant way is via optimal error bounds using the Peano kernel theorem, in particular we can use the following result due to \cite{shadrin199525}: Define the nodal polynomial for the interpolation problem \eqref{eqn:extended_Filon_interpolation_problem} as $\tilde{r}(x)=(x^2-1)^s\prod_{j=1}^\nu(x-c_j)$ then we have the following bounds (where the constants are optimal over $f\in C^{(\nu+2s+1)}([-1,1])$):
			\begin{align}\label{eqn:shadrins_estimate}
				\|f^{(j)}-p^{(j)}\|_\infty\leq\frac{\|\tilde{r}^{(j)}\|_\infty}{(\nu+2s+1)!}\|f^{(\nu+2s+1)}\|_\infty.
			\end{align}
			We can combine \eqref{eqn:shadrins_estimate} with Prop. \ref{prop:filon_paradigm_for_alpha_zero} and the trivial estimate 
			\begin{align*}
				\left|I^{(3)}_{\omega,\beta}[f]-\mathcal{Q}^{[s,\nu]}_\omega[f]\right|\leq \omega^{-1}\|H^{(1)}_0\|_{L^1([0,\omega])}\|f-p\|_\infty,
			\end{align*}
			to find:
			\begin{corollary}\label{cor:nu_explicit_error_estimates}For any $s\in \mathbb{N}$ there is $C_s>0$ such that for all $f\in C^{\infty}([-1,1]),\nu\in\mathbb{N},\beta\neq0,\omega\geq 1$:
				\begin{align*}
					\left|I^{(3)}_{\omega,\beta}[f]-\mathcal{Q}^{[s,\nu]}_\omega[f]\right|\leq \min&\left\{\omega^{-1}\|H^{(1)}_0\|_{L^1([0,\omega])}\|\tilde{r}\|_\infty,C_k \omega^{-(s+2)} \left( \|\tilde{r}^{(s+1)}\|_\infty\frac{|\beta+1|^{s+3}-1}{|\beta+1|-1}\right.\right.\\
					&\quad\quad\quad\quad\quad\quad\quad\quad\quad\quad\left.\left.+\log\omega\|\tilde{r}^{(s+2)}\|_{\infty}|\beta+1|^{-(s+2)}\right)\right\}\frac{\|f^{(\nu+2s+1)}\|_\infty}{(\nu+2s+1)!}.
				\end{align*}
			\end{corollary}
			\subsection{Recursive moment computation}\label{sec:recursive_moment_computation_wave_scattering}
			The results in Prop. \ref{prop:filon_paradigm_for_alpha_zero} and Cor. \ref{cor:nu_explicit_error_estimates} guarantee convergence of the direct Filon method $\mathcal{Q}^{[\nu,s]}_\omega[f]=I^{(3)}_{\omega,\beta}[p]$, where $p$ satisfies \eqref{eqn:extended_Filon_interpolation_problem}. \revisions{Thus it} remains to compute the corresponding quadrature moments $\sigma_n:=\sqrt{s}_n I^{(3)}_{\omega,\beta}[T_n]$. The oscillatory kernel of $I^{(3)}_{\omega,\beta}$ with respect to the Chebyshev weight is
			\begin{align*}
				h_\omega(x)=\sqrt{1-x^2}H_0^{(1)}\left(\omega \frac{x+1}{2}\right)e^{i\omega\beta (x+1)/2}.
			\end{align*}
			The Hankel function satisfies Bessel's equation $(x^2(\D/\D x)^2+x\D/\D x +x^2)H^{\revisions{(1)}}_{\revisions{0}}(x)=0$ \cite[Eq.~9.1.1]{abramowitz1965handbook}. Thus a change of variable and multiplication by $(1-x)$ (to ensure the equation involves a combination of operators from Lemma~\ref{lem:banded_operators_Chebyshev_polynomials}) shows $\mathcal{L}_\omega h_\omega=0$ for
			\begin{align*}
				\mathcal{L}_\omega&=\left((1-x^2)\frac{d}{dx}\right)^2+i  \left(\beta  \omega \left(x^2-1\right)-i (3
				x+1)\right)\left(\left(1-x^2\right) \frac{d}{dx}\right)\\
				&\quad\quad\quad-\frac{1}{4} \left((\beta ^2-1) \omega^2 \left(x^2-1\right)^2-2 i \beta  \omega
				(x-1) (x+1)^2-4 \left(x^2+x+1\right)\right).
			\end{align*}
			By Lem\revisions{ma}~\ref{lem:banded_operators_Chebyshev_polynomials} the action of $\mathcal{L}_\omega$ on the basis $\phi_n=\sqrt{s_n}\,T_n$ has a banded matrix representation. Although $h_\omega\not\in H^{2}([-1,1],(1-x^2)^{-1/2})$ we can use the same integration by parts argument as in \eqref{eqn:general_integration_by_parts_to_take_adjoint} to ensure \eqref{eqn:moving_adjoint_to_other_side} holds and we find after a few steps of algebra the following recursive relation satisfied by the moments valid for all $n\in\mathbb{Z}$, where we again use the notation $\tilde{\sigma}_n=\sigma_n/\sqrt{s_n}$ and $\tilde{\sigma}_{-n}=\tilde{\sigma}_n$:
			\begin{align}\label{eqn:recurrence_for_sigma_1}
				\begin{split}
					(1-\beta^2)\tilde{\sigma}_{n-4}&-\frac{4i\beta(2n-7)}{\omega}\tilde{\sigma}_{n-3}+\frac{16(n-3)^2+8i\beta\omega-4(1-\beta^2)\omega^2}{\omega^2}\tilde{\sigma}_{n-2}\\
					&+\frac{-32(n-2)+4i\beta(6n-7)\omega}{\omega^2}\tilde{\sigma}_{n-1}+\frac{32(3-n^2)-16i\beta\omega+6(1-\beta^2)\omega^2}{\omega^2}\tilde{\sigma}_n\\
					&+\frac{32(n+2)-4i\beta(6n+7)\omega}{\omega^2}\tilde{\sigma}_{n+1}+\frac{16(n+3)^2+8i\beta\omega-4(1-\beta^2)\omega^2}{\omega^2}\tilde{\sigma}_{n+2}\\
					&+\frac{4i\beta(2n+7)}{\omega}\tilde{\sigma}_{n+3}+(1-\beta^2)\tilde{\sigma}_{n+4}=0
				\end{split}
			\end{align}
			\subsection{Initial conditions}\label{sec:initial_conditions}
			Since the recurrence \eqref{eqn:recurrence_for_sigma_1} holds also when $n=0,1,2,3$ already four initial conditions $\tilde{\sigma}_0,\dots,\tilde{\sigma}_{4}$ are sufficient to start the moment computation (and just three initial conditions suffice when $\beta=1$). We begin by proving an \revisions{expression that} allows for efficient and accurate computation of $\tilde{\sigma}_0$:
			\begin{lemma}\label{lem:expression_for_initial_moments_linear_oscillator} For $\omega>0, \beta\neq -1$:
				\begin{align}\label{eqn:closed_form_expression_rho_0}
					\tilde{\sigma}_0&=-\frac{2i}{\pi}\omega^{-1}e^{i(\beta+1)\omega}\int_{0}^{\infty} \frac{1}{\sqrt{t}}\frac{e^{-t\omega}}{\sqrt{2i-t}(1+\beta+it)}dt+\frac{2}{\pi}\omega^{-1}\begin{cases}
						\frac{2}{\sqrt{1-\beta^2}}\revisions{\arctan\sqrt{\frac{1-\beta}{1+\beta}}},&\quad |\beta|<1,\\
						\frac{\revisions{2}}{\sqrt{\beta^2-1}}\revisions{\mathrm{arctanh}\sqrt{\frac{\beta-1}{1+\beta}}},&\quad |\beta|>1,\\
						1,&\quad \beta=1.
					\end{cases}
				\end{align}
			\end{lemma}
			\begin{proof}
				We have the following integral expression \cite[Eq.~10.9.10]{NIST:DLMF}
				\begin{align*}
					H_0^{(1)}(z)=-\frac{2i}{\pi}\int_{0}^{\infty}e^{iz\cosh t}dt, \quad \forall\, 0<\arg z<\pi.
				\end{align*}
				Thus we have:
				\begin{align*}
					\omega\sigma_0&=\lim_{\epsilon\rightarrow 0^+}\int_{0}^{\omega}H^{(1)}_0( x+i\epsilon)e^{i \beta x}dx=-\frac{2i}{\pi}\lim_{\epsilon\rightarrow 0^+}\int_0^\omega\int_{0}^{\infty}e^{(ix-\epsilon)\cosh t}dte^{i\beta x}dx\\
					&=-\frac{2}{\pi}\lim_{\epsilon\rightarrow 0^+}\int_{0}^{\infty}e^{-\epsilon\cosh t} \frac{e^{i\omega(\cosh t+\beta)}-1}{\cosh t+\beta}dt=-\frac{2}{\pi}\int_{0}^{\infty}\frac{e^{i\omega(\cosh t+\beta)}-1}{\cosh t+\beta}dt\\
					&=-\frac{2}{\pi}\int_{1}^{\infty}\frac{e^{i\omega(y+\beta)}}{y+\beta}\frac{dy}{\sqrt{y^2-1}}+\frac{2}{\pi}\int_{0}^{\infty}\frac{1}{\cosh t+\beta}dt,
				\end{align*}
				where in the final line we used the change of variable $y=\cosh t$. The second integral can be evaluated explicitly \cite[Eq.~3.513.2]{gradshteyn2014table} when $\beta\neq \pm1$ and by taking an appropriate limit as $\beta\rightarrow 1^+$, justified by the dominated convergence theorem, we can also deduce its value for $\beta=1$. For the first integral the decay of the integrand is sufficient so that we can deform the contour of integration to $y=1+it$ (noting that the square root singularity at $y=1$ can be dealt with by excluding a small neighbourhood of the origin during the change of contour). Combining those yields precisely the expression \eqref{eqn:closed_form_expression_rho_0}.
			\end{proof}\ \newline
			\newline If we set $g_\beta(t)=(\sqrt{2i-t}(1+\beta+it))^{-1}$ we can write the remaining integral in \eqref{eqn:closed_form_expression_rho_0} in the form
			\begin{align*}
				-\frac{2i}{\pi}\omega^{-1}\int_{0}^{\infty} \frac{1}{\sqrt{t}}\frac{e^{i(\beta+1)\omega}e^{-t\omega}}{\sqrt{2i-t}(1+\beta+it)}dt&=-\frac{2i}{\pi}\omega^{-\frac{3}{2}}e^{i(\beta+1)\omega}\int_{0}^{\infty}g_\beta\left(\frac{t}{\omega}\right) \frac{1}{\sqrt{t}}e^{-t}dt.
			\end{align*}
			Differentiation with respect to $\beta$ (justified by the dominated convergence theorem) allows us to find similar expressions for the remaining initial conditions $\tilde{\sigma}_j, j=1,2,3$. For completeness these are provided in Appendix D. Thus we found expressions for the initial moments in terms of simple functions and integrals of the form
			\begin{align*}
				\int_{0}^{\infty}f\left(\frac{t}{\omega}\right) \frac{1}{\sqrt{t}}e^{-t}dt=\sqrt{\omega}\int_{-\infty}^{\infty}f\left(z^2\right)e^{-\omega z^2}dz,
			\end{align*}
			which have an exponentially decaying integrand (with faster exponential decay in $z$ as $\omega$ increases) and can be evaluated efficiently using either adaptive quadrature or Gauss--Hermite quadrature. We shall not provide a detailed study of the evaluation, but note that we can write
			\begin{align*}
				g_\beta\left(\frac{t^2}{\omega}\right)=\omega^{\frac{3}{2}}\frac{1}{\sqrt{2i\omega-t^2}}\frac{1}{(1+\beta)\omega+it^2}
			\end{align*}
			which means the integrand becomes nearly singular when $\omega,(1+\beta)\omega\ll 1$. Thus for practical purposes we restrict the use of \eqref{eqn:closed_form_expression_rho_0} to the case when $\omega,(1+\beta)\omega\geq1$. Of course, the kernel of $I^{(3)}_{\omega,\beta}[f]$ takes the form $h_0(\omega x)\exp\left(i\omega(1+\beta) x\right)$ (cf. Lemma~\ref{lem:phase_extraction_in_H_0}) which means it is only highly oscillatory when $\omega(1+\beta)\gg 1$, and so in the case when \eqref{eqn:closed_form_expression_rho_0} is near singular we do not need to use Filon methods for the approximation of $I^{(3)}_{\omega,\beta}[f]$ to begin with. The same holds true for the analogous expressions for $\tilde{\sigma}_{j},j=1,2,3,$ as given in Appendix D.
			
			%\begin{remark}
			%	In the special case $\beta=1$ we can also express $\tilde{\sigma}_0$ in terms of Should be able to write in terms of hypergeometric $U$ function Tricomi's (confluent hypergeometric) function; which has the following integral representation:
			%	\begin{align*}
				%		U(a,b,z)=\frac{1}{\Gamma(a)}\int_{0}^\infty e^{-zu}u^{a-1}(1-u)^{b-a-1}du
				%	\end{align*}
			%\end{remark}
			\subsection{Behaviour of homogeneous solutions and initial stability}\label{sec:stability_linear_hankel}
			Note that we can write the recurrence \eqref{eqn:recurrence_for_sigma_1} in the form
			\begin{align}\begin{split}\label{eqn:perturbation_form_recurrence_for_sigma_1_beta_neq1}
					(1-\beta^2)\left(\tilde{\sigma}_{n-4}-4\tilde{\sigma}_{n-2}+\right.&\left.6\tilde{\sigma}_n-4\tilde{\sigma}_{n+2}+\tilde{\sigma}_{n+4}\right)=\frac{n^2}{\omega^2}(-16)(\tilde{\sigma}_{n-2}-2\tilde{\sigma}_n+\tilde{\sigma}_{n+2})\\
					&+\frac{n}{\omega}8i\beta\left(\tilde{\sigma}_{n-3}-3\tilde{\sigma}_{n-1}+3\tilde{\sigma}_{n+1}-\tilde{\sigma}_{n+3}\right)\\
					&+\frac{n}{\omega^2}32\left(-3\tilde{\sigma}_{n-2}+\tilde{\sigma}_{n-1}-\tilde{\sigma}_{n+1}-3\tilde{\sigma}_{n+2}\right)\\
					&+\frac{1}{\omega}4i\beta\left(-7\tilde{\sigma}_{n-3}-2\tilde{\sigma}_{n-2}+7\tilde{\sigma}_{n-1}+4\tilde{\sigma}_{n}+7\tilde{\sigma}_{n+1}-2\tilde{\sigma}_{n+2}-7\tilde{\sigma}_{n+3}\right)\\
					&+\frac{1}{\omega^2}(-32)\left(3\tilde{\sigma}_{n-2}+2\tilde{\sigma}_{n-1}+2\tilde{\sigma}_{n}+2\tilde{\sigma}_{n+1}+3\tilde{\sigma}_{n+2}\right).
				\end{split}
			\end{align}
			A similar argument to the proof of Thm.~\ref{thm:square_root_stability_algebraic_singularity} shows that, whenever $\beta\neq1$, if $\omega$ is sufficiently large compared to $n$, the solutions to \eqref{eqn:perturbation_form_recurrence_for_sigma_1_beta_neq1} grow no faster than algebraically with $n$. For the case $\beta=1$ the recurrence reduces to seven terms and takes the form
			\begin{align}
				\begin{split}\label{eqn:perturbation_form_recurrence_for_sigma_1_beta_1}
					&(-2n+7)\tilde{\sigma}_{n-3}+2\tilde{\sigma}_{n-2}+\left(6 n-7\right)\tilde{\sigma}_{n-1}+\left(-4\right)\tilde{\sigma}_{n}+\left(-6 n-7\right)\tilde{\sigma}_{n+1}+2\tilde{\sigma}_{n+2}+(2n+7)\tilde{\sigma}_{n+3}\\
					&\hspace{3cm}=-\frac{-12n+28+4(n-2)(n-1)}{i\omega}\tilde{\sigma}_{n-2}-\frac{-8n+16}{i\omega}\tilde{\sigma}_{n-1}-\frac{-8n^2+24}{i\omega}\tilde{\sigma}_{n}\\
					&\hspace{3cm}\quad-\frac{8n+16}{i\omega}\tilde{\sigma}_{n+1}-\frac{12n+28+4(n+1)(n+2)}{i\omega}\tilde{\sigma}_{n+2}.
				\end{split}
			\end{align}
			Here we can understand the behaviour of the recurrence operator on the left hand side by substituting $\revisions{\gamma}_n=(2n+3)\tilde{\sigma}_{n+2}+\tilde{\sigma}_n-(2n+1)\tilde{\sigma}_{n-2}$, which yields
			\begin{align*}
				&(-2n+7)\tilde{\sigma}_{n-3}+2\tilde{\sigma}_{n-2}+\left(6 n-7\right)\tilde{\sigma}_{n-1}+\left(-4\right)\tilde{\sigma}_{n}+\left(-6 n-7\right)\tilde{\sigma}_{n+1}+2\tilde{\sigma}_{n+2}+(2n+7)\tilde{\sigma}_{n+3}\\
				&\hspace{11cm}=\gamma_{n-2}-2\gamma_n+\gamma_{n+2},
			\end{align*}
			and shows, by a simple discrete variation of constants argument, that if $\omega$ is sufficiently large compared to $n$ then the solutions to \eqref{eqn:perturbation_form_recurrence_for_sigma_1_beta_1} grow at most linearly in $n$. In both cases $\beta\neq1$ and $\beta=1$ the solutions to \eqref{eqn:recurrence_for_sigma_1} thus have algebraic behaviour in the initial regime. We find from numerical experiments that this behaviour changes as $n$ increases for fixed $\omega$, and that some of the solutions exhibit super-algebraic growth for $n$ sufficiently large, thus leading to instability in \eqref{eqn:recurrence_for_sigma_1}. To understand where this transition occurs we follow the procedure described in \S\ref{sec:tail_stability_algebraic_singularities}. We suspect the change of behaviour occurs when $n\propto \omega$. Thus we let $n=C_{n,\omega}\omega$ and make the Ansatz $\tilde{\sigma}_{n+ j}/\tilde{\sigma}_n=\lambda^j,\, -4\leq j\leq 4$, which when plugged into \eqref{eqn:recurrence_for_sigma_1} results in the following condition at leading order in $\omega$:
			\begin{align}\begin{split}\label{eqn:lambda_condition_linear_hankel}
					(1 - \beta^2)\lambda^{-4}&-8i\beta C_{n,\omega}\lambda^{-3}+\left(16C_{n,\omega}^2-4(1-\beta^2)\right)\lambda^{-2}+\left(24i\beta C_{n,\omega}\right)\lambda^{-1}\\
					&+\left(-32C_{n,\omega}^2+6(1-\beta^2)\right)+\left(-24 i \beta C_{n,\omega}\right)\lambda+\left(16C_{n,\omega}^2-4(1-\beta^2)\right)\lambda^2\\
					&\hspace{6cm}+8i\beta C_{n,\omega}\lambda^3+(1 - \beta^2)\lambda^4=0.
				\end{split}
			\end{align}
			When $\beta\neq1$ the condition has eight solutions for $\lambda$:
			\begin{align*}
				\lambda=\pm1,\pm1,\frac{-2iC_{n,\omega}\pm\sqrt{(1-\beta)^2-4C_{n,\omega}^2}}{1-\beta},\frac{2iC_{n,\omega}\pm\sqrt{(1+\beta)^2-4C_{n,\omega}^2}}{(\beta+1)}.
			\end{align*}
			All of those values are in modulus equal to $1$ if and only if $n/\omega=C_{n,\omega}\leq\min\{|1+\beta|/2,|1-\beta|/2\}$, so we expect algebraic behaviour in this regime and the onset of super-algebraic growth to occur when $n\approx \omega \min\{|1+\beta|/2,|1-\beta|/2\}$. This behaviour is confirmed in Fig. \ref{fig:normgrowthbeta2}. When $n/\omega=C_{n,\omega}\geq\min\{|1+\beta|/2,|1-\beta|/2\}$ at most two of the values for $\lambda$ have modulus greater than 1 and we thus expect Oliver's algorithm \cite{oliver1968numerical} with six initial and two endpoint values to provide a stable way of computing the remainder of the moments. \revisions{These two endpoint values (i.e. values of $\tilde{\sigma}_{N},\tilde{\sigma}_{N+1}$ for some $N\gg \omega$) can be approximated using the asymptotic expansion of $\tilde{\sigma}_N$ as $N\rightarrow\infty$ for fixed $\omega$ which is obtained from the method of stationary phase \cite{bender2013advanced} (cf. also \cite[\S4]{Dominguez2011}).} As we explained in \S\ref{sec:stability_analysis_of_the_recurrences}, in practical applications of Filon methods it is less important to compute moments when $n\gtrsim\omega$ since at that point classical quadrature \revisions{is no more expensive than} the Filon method. In the interest of brevity we therefore omit a discussion of the application of Oliver's algorithm, but note that we performed initial numerical experiments which suggest that this provides indeed a satisfactory way for computing the remaining quadrature moments.
			
			When $\beta=1$ \eqref{eqn:lambda_condition_linear_hankel} has six solutions, $\lambda=\pm1,\pm1, iC_{n,\omega}\pm\sqrt{1-C_{n,\omega}^2}$. These solutions are in modulus equal to $1$ whenever $n/\omega=C_{n,\omega}\leq1$ which suggests the onset of super-algebraic growth lies around $n\approx \omega$. This is confirmed in Fig. \ref{fig:normgrowthbeta1}. Moreover, when $n/\omega>1$ one of the solutions for $\lambda$ is in modulus greater than 1, which indicates that Oliver's algorithm with five initial and one endpoint value can be used to compute the remaining moments in a stable way. \revisions{Again the endpoint value (i.e. $\tilde{\sigma}_{N}$ for some $N\gg \omega$) can be approximated using the method of stationary phase for $N\rightarrow\infty$.}
			\subsection{Numerical evidence of stable forward propagation}
			Similar to \eqref{eqn:algebraic_recurrence_as_matrix_product} we can write the recurrence \eqref{eqn:recurrence_for_sigma_1} in the form
			\begin{align*}
				\bm{x}_{N+1}^{(j)}=\prod_{n=1}^NB^{(j)}_n\bm{x}^{(j)}_1, \quad\forall N\in\mathbb{Z},
			\end{align*}
			where $j=1$ corresponds to the case $\beta=1$ and $j=2$ covers the case $\beta\neq1$. Here $\bm{x}^{(1)}_n=(x_{n+2},\dots,x_{n-3})^T$, $\bm{x}^{(2)}_n=(x_{n+3},\dots,x_{n-4})^T$, and $B^{(j)}_n,\, n\geq 0,j=1,2,$ are $6\times 6$, and $8\times8$ matrices respectively whose entries are, analogously to \eqref{eqn:matrix_form_algebraic_recurrence_as_matrix_product}, given by the coefficients of the recurrence \eqref{eqn:recurrence_for_sigma_1} in the top row and the bottom left $5\times5$ and $7\times7$ entries are given by the identity matrix of respective size $\bm{I_5},\bm{I_7}$. This means that the matrices are of the shape
			\begin{align*}
				B^{(1)}_n&=\begin{pmatrix}
					\star&\star\\
					\bm{I_{5}}&0
				\end{pmatrix},\quad
				B^{(2)}_n=\begin{pmatrix}
					\star&\star\\
					\bm{I_{7}}&0
				\end{pmatrix},
			\end{align*}
			where $\star$ is a placeholder for the non-zero entries given by the coefficients of the recurrence \eqref{eqn:recurrence_for_sigma_1} which are not repeated in the interest of brevity. Similar to \S\ref{sec:numerical_examples_algebraic_singularities} we have the upper bound $\|\bm{x}_N^{(j)}\|\leq \left\|\prod_{n=1}^NB^{(j)}_n\right\|\|\bm{x}_1^{(j)}\|$, thus we can look at the norm of the matrix product to find an upper bound on the growth of solutions to the recurrence \eqref{eqn:recurrence_for_sigma_1}. In Fig.~\ref{fig:norm_growth_linear_hankel} we plot these norms and we see initial algebraic growth which transitions to super-algebraic roughly at the points predicted in \S\ref{sec:stability_linear_hankel}: when $n\approx\omega$ in Fig. \ref{fig:normgrowthbeta1} and when $n\approx 0.4\omega$ in Fig. \ref{fig:normgrowthbeta2}.
			\begin{figure}[h!]
				\centering
				\begin{subfigure}[h]{0.49\linewidth}
					\centering
					\includegraphics[width=0.95\linewidth]{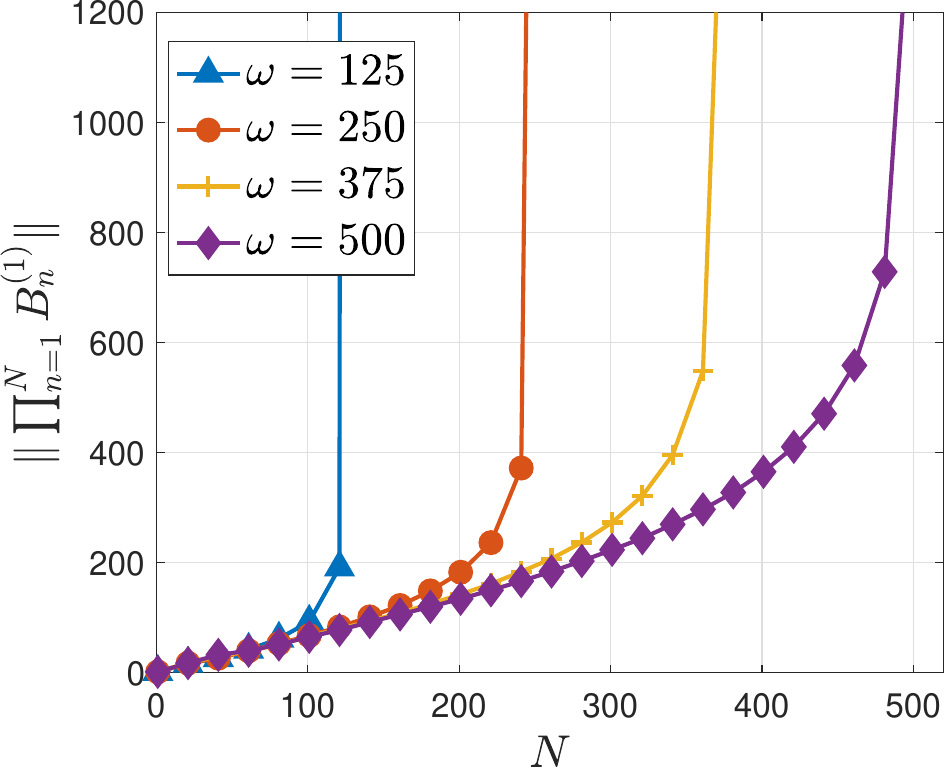}
					\caption{$\beta=1$}
					\label{fig:normgrowthbeta1}
				\end{subfigure}%
				\begin{subfigure}[h]{0.49\linewidth}
					\centering
					\includegraphics[width=0.97\linewidth]{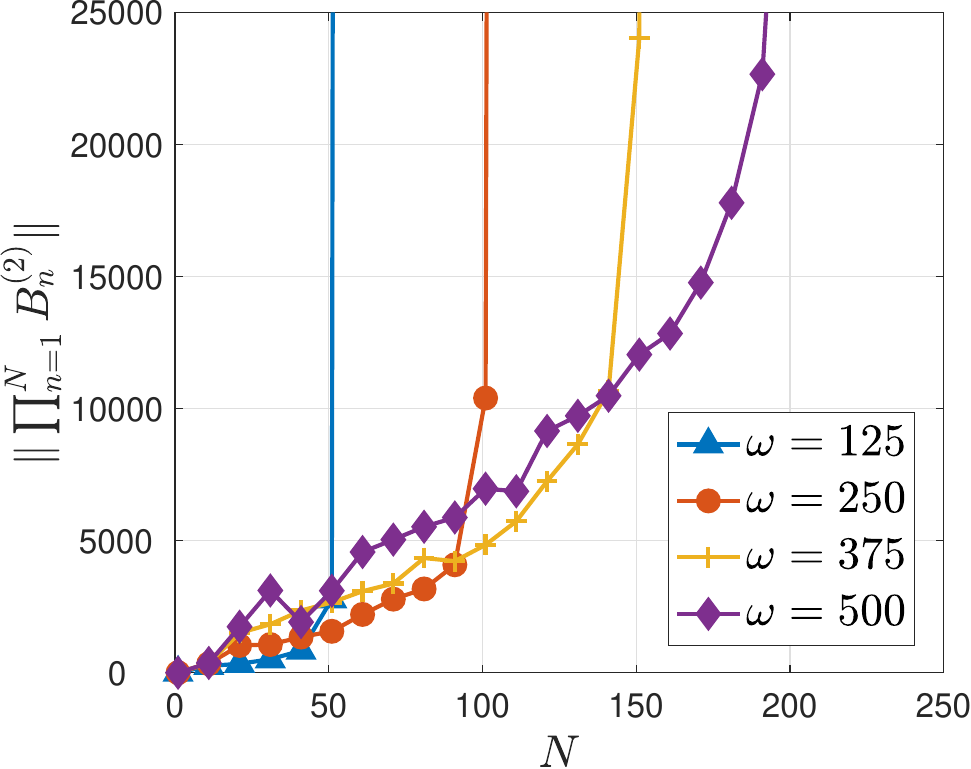}
					\caption{$\beta=0.2$}
					\label{fig:normgrowthbeta2}
				\end{subfigure}
				\caption{The growth of solutions to \eqref{eqn:recurrence_algebraic_singularities_chebyshev_moments} as measured by $\left\|\prod_{n=1}^NB^{(j)}\right\|,j=1,2$. The initial algebraic transitions to super-algebraic growth roughly at the points predicted heuristically in \S\ref{sec:stability_linear_hankel}.}
				\label{fig:norm_growth_linear_hankel}
			\end{figure}

			\subsection{Wave scattering on a screen}
			\label{sec:numerical_examples_scattering_on_screen}
			%\questions{These are piecewise polynomials so need not apply Filon! Can apply to geometric optics approximation though!}
			Integrals of the form \eqref{eqn:form_of_integrals_linear_hankel_kernel} appear in hybrid numerical-asymptotic collocation methods for high-frequency wave scattering on screens in two dimensions (see for instance \cite{hewett2014} and \cite{parolin2015}). Recently, \cite{gibbs2019fast} constructed a very efficient numerical steepest descent (NSD) method that can be used to assemble the matrix and right hand side in the corresponding collocation system at frequency-independent cost. In this example we demonstrate that our direct Filon method with recursive moment computation can be applied to achieve the same goal. As a conceptual difference we highlight that our method relies on evaluations of the integrand strictly in the domain of integration (contrary to NSD where a complex extension of the integrand is evaluated along steepest descent paths). This can be of advantage when the functional form of the incident field is unknown (for instance in a geometric theory of diffraction approximation to multiple scattering) or when a complex extension is not readily available or difficult to evaluate due to the presence of branch cuts in the complex plane. The latter is the case in the example that we consider, although we note that this can be overcome by choosing the steepest descent path in NSD carefully. Let us consider a simple example exhibiting the main features of the direct Filon method when applied in this setting: the scattering of a two-dimensional highly oscillatory Gaussian beam by a finite plate $\Gamma$ extending from $(-1,0)$ to $(1,0)$ in $\mathbb{R}^2$. We follow \cite[Eq.~(55)]{Kravtsov1967} and \cite[Eq.~(17)]{Keller71} and assume an incident field of the form
			\begin{align*}
				\psi_i(x,y)=\left(1+\frac{i\tilde{y}}{\omega a^2}\right)^{-\frac{1}{2}}\exp\left(i\omega \tilde{y}-\frac{\tilde{x}^2}{2a^2}\left(1+\frac{i\tilde{y}}{\omega a^2}\right)^{-\frac{1}{2}}\right)
			\end{align*}
			where $\tilde{x}=x\sin\theta-y\cos\theta,\tilde{y}=x\cos\theta+y\sin\theta$. This describes a Gaussian beam focussed at $(x,y)=(0,0)$, which propagates in the direction $(\cos\theta,\sin\theta)$ at frequency $\omega$ and has width $a$ in the plane $\{\tilde{y}=x\cos\theta+y\sin\theta=0\}$. The scattering problem on a perfectly conducting plate (i.e. with Dirichlet boundary conditions) can be written in the form \cite[Eq.~(7)]{gibbs2019fast}
			\begin{align*}
				\psi_i(s_x)=\mathcal{S}\left(\partial_n\psi\right)(s_x):=\frac{i}{4}\int_{-1}^{1} H_0^{(1)}\left(\omega|s_x-s_y|\right)\partial_n\psi(s_y)ds_y,\quad s_x\in[-1,1],
			\end{align*}
			where $\psi$ is the unknown scattered field and $s_x,s_y$ are coordinates in arclength along the plate. We follow the hybrid Ansatz described by \cite{hewett2014} where the unknown $\partial_n\psi$ is expanded in the form
			\begin{align*}
				\partial_n\psi(s_y)=V_0(s_y;\omega)+\sum_{l=1}^L V_l^+(s_y)e^{i\omega s_y}+V_l^-(s_y)e^{-i\omega s_y}
			\end{align*}
			where $V_0(s_y;\omega)=2\partial\psi_i(s_y)/\partial n$ is the geometrical optics approximation and $V_l^\pm$ are piecewise polynomials of low degree defined on a mesh graded towards the endpoints of the plate. Thus the collocation system for the free parameters in $V_l^\pm$ takes the form
			\begin{align*}
				\sum_{l=1}^N \mathcal{S}\left(V_l^+(\cdot)\,e^{i\omega\, \cdot\,}+V_l^-(\cdot)\,e^{-i\omega \,\cdot\,}\right)(s_m)=\psi_i(s_m)-2\mathcal{S}\left(\partial_n \psi_i\right)(s_m),
			\end{align*}
			for some collocation points $s_m\in [-1,1], m=1,\dots, M$. In the present example we shall focus on the evaluation of the geometrical optics contribution, but we note that the integrals over the basis terms $V_l^\pm(s_y)\exp\revisions{(}\pm i\omega s_y\revisions{)}$ can also be approximated efficiently using expressions of the form \eqref{eqn:closed_form_expression_rho_0}. On the blade the incident Gaussian beam takes the form
			\begin{align*}
				\partial_n \psi_i(s_y,0)&=\left[\frac{-8 i a^2 \omega  s_y \cos\theta +3 s_y^2 \cos(2\theta)+5 s_y^2 }{8a^4\omega^2}\left(1+\frac{i s_y \cos\theta}{a^2 \omega }\right)^{-2}+\frac{-1}{2a^2\omega^2}\left(1+\frac{i s_y \cos\theta}{a^2 \omega }\right)^{-\frac{3}{2}}\right.\\
				&\quad\quad\quad\quad\left.+\left(1+\frac{i s_y \cos\theta}{a^2 \omega }\right)^{-\frac{1}{2}}\right] i \omega  \sin\theta\exp\left(i\omega s_y\cos\theta-\frac{s_y^2\sin^2\theta}{2a^2}\left(1+\frac{is_y\cos\theta}{\omega a^2}\right)^{-\frac{1}{2}}\right)\\
				&=:A(s_y;\omega)\exp\left(i\omega s_y\cos\theta\right)
			\end{align*}
			where we have extracted the amplitude of $\partial_n\psi_i$ in $A$. Note the only $\omega$ dependency of $A$ is via a constant multiplication out front and via the function $f(x)=(1+ix\cos\theta/\revisions{a^2})^{-1/2}$ in the form $f(x/\omega)$. Since
			\begin{align*}
				\left|\frac{d^j}{dx^j}f(x)\right|=\frac{(2j-1)!!}{2^j}|\cos\theta|^ja^{-j}\left|1+\frac{ix\cos\theta}{\omega a^2}\right|^{-\frac{1}{2}-l}\leq C_j x^{-\frac{1}{2}-j}, \quad \forall\, j\geq 0,
			\end{align*}
			where $(2j-1)!!=(2j-1)(2j-3)\cdots1$, the tangential derivatives of $A(s_y;\omega)$ do not grow in $\omega$, i.e. this is a smooth non-oscillatory function that can be well approximated by polynomials on $[-1,1]$ uniformly in $\omega\geq 1$ in the sense of \eqref{eqn:shadrins_estimate}. The geometrical optics approximation thus requires us to compute the following terms for all collocation points $s_m\in [-1,1]$:
			\begin{align}\nonumber
				2\int_{-1}^{1}H^{(1)}_0(\omega |s-s_m|)&\partial_n \psi_i(s,0)ds=\\\nonumber
				&2\int_{-1}^{s_0}H^{(1)}_0(\omega |s-s_m|)A(s)e^{i\omega s\cos\theta}ds+2\int_{s_0}^{1}H^{(1)}_0(\omega |s-s_m|)A(s)e^{i\omega s\cos\theta}ds\\\nonumber
				&=2(1+s_m)e^{i\omega s_0\cos\theta}\int_{0}^{1}H^{(1)}_0(\omega(1+s_m) t)A(s_m-(1+s_m)t)e^{-i\omega(1+s_m) t\cos\theta}dt\\\nonumber
				&\quad+2e^{i\omega s_m\cos\theta}\int_{0}^{s_m}H^{(1)}_0(\omega t)A(t+s_m)e^{i\omega t\cos\theta}dt\\\label{eqn:geometrical_optics_approximation_in_our_form}
				&=e^{i\omega s_m\cos\theta}\left((1+s_m)I^{(3)}_{\omega(1+s_m),-\cos\theta}[A_1]+(1-s_m)I^{(3)}_{\omega(1-s_m),\cos\theta}[A_2]\right),
			\end{align}
			where we took $A_1(x)= A\left((s_0-1)/2-x(s_0+1)/2\right)$, $A_2(x)=A\left((s_0+1)/2+x(1-s_0)/2\right)$. Thus we can consider the approximation of \eqref{eqn:geometrical_optics_approximation_in_our_form} using the direct Filon method with recursive moment computation as described in \S\ref{sec:recursive_moment_computation_wave_scattering}. 
			
			The performance of this method is demonstrated by the results in Fig.~\ref{fig:numerical_experiments_direct_filon_hankel}. Here we choose $a=0.25, \theta=\pi/4,s_m=0$. In Fig.~\ref{fig:asymptotic_convergence_hankel_filon} we see the behaviour of the \textit{relative} error of the direct Filon method as a function of $\omega$ with $\nu=6$ fixed. We recall from Prop.~\ref{prop:filon_paradigm_for_alpha_zero} that the direct Filon method has asymptotic error $\mathcal{O}(\omega^{-s-2}\log\omega)$ and a similar argument shows that $2\mathcal{S}(\partial_n\psi_i)(s_n)$ has asymptotic behaviour $\mathcal{O}(\omega^{-1}\log\omega)$. Thus we expect the relative error to behave like $\mathcal{O}(\omega^{-s-1})$ which is confirmed in Fig.~\ref{fig:asymptotic_convergence_hankel_filon}. This means that the direct Filon method can approximate the integral to a fixed relative error at uniform cost in $\omega$.
			
			\begin{figure}[h!]
				\centering
				\begin{subfigure}[h]{0.49\linewidth}
					\centering
					\includegraphics[width=0.95\textwidth]{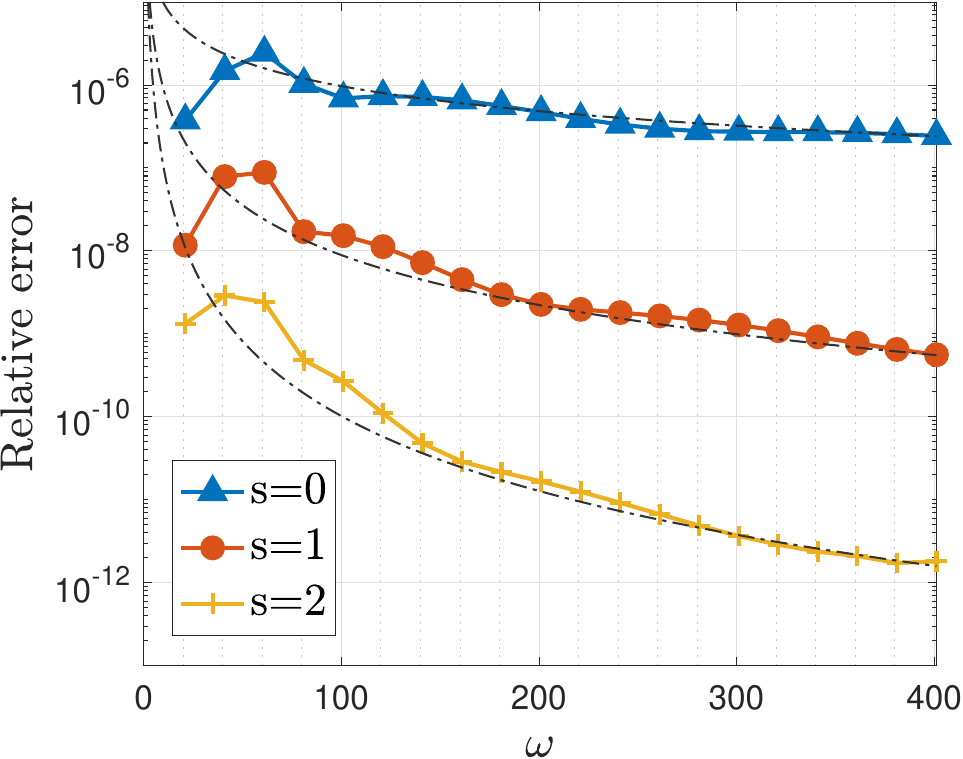}
					\caption{Relative error as function of $\omega$ for fixed $\nu=6$.}
					\label{fig:asymptotic_convergence_hankel_filon}
				\end{subfigure}%
				\begin{subfigure}[h]{0.49\linewidth}
					\centering
					\includegraphics[width=0.95\textwidth]{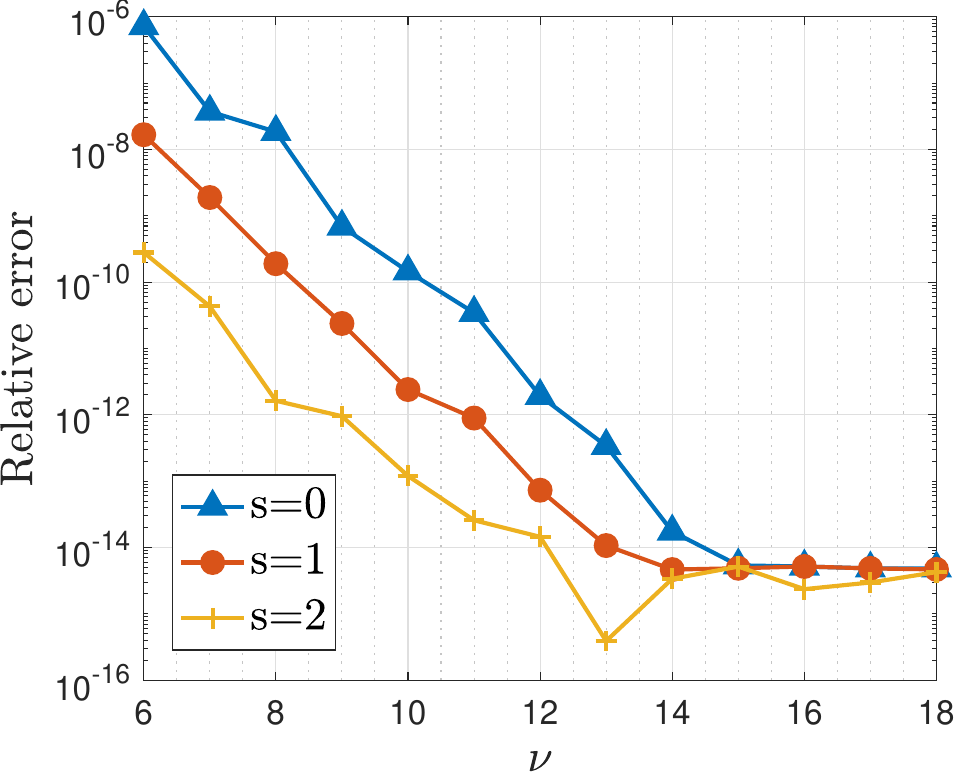}
					\caption{Relative error as function of $\nu$ for fixed $\omega=100$.}
					\label{fig:nu_convergence_hankel_filon}
				\end{subfigure}
				\caption{Relative error of the direct Filon method for evaluating $\mathcal{S}(\partial_n\psi_i)(s_m)$ with $a=0.25,\theta=\pi/4, s_m=0$.}
				\label{fig:numerical_experiments_direct_filon_hankel}
			\end{figure}
			
			In Fig.~\ref{fig:nu_convergence_hankel_filon} we consider the convergence properties of the method for a fixed $\omega=100$ as $\nu$ increases. Since $A_1,A_2$ are smooth we expect, by Corollary~\ref{cor:nu_explicit_error_estimates}, to find spectral convergence in $\nu$ for any fixed value of $\omega$. This is indeed confirmed in Fig.~\ref{fig:nu_convergence_hankel_filon}. In both numerical examples the reference solution for the true integral was computed with a graded Clenshaw--Curtis method (as described in \S\ref{sec:numerical_examples_algebraic_singularities}) with $M=6000,\nu=10,r=40$.
			
			\section{Concluding remarks}
			\label{sec:conclusions}
			In this work we sought to address the `moment-problem' for Filon methods by providing a general framework for constructing recurrences satisfied by the Filon quadrature moments. This framework is based on the observation that many physically relevant oscillatory kernels are in the null space of certain differential operators whose action on the interpolation basis is represented by a banded matrix. The recursive moment computation allowed us to construct direct Filon methods for several examples of interest, two of which we studied in further detail: integrals with algebraic singularities and stationary points and integrals involving a Hankel function. For the former we proved rigorous stability results guaranteeing that the initial moments can be computed with at worst linear error growth. We also demonstrated the advantageous properties of the direct Filon method which perfectly matches the asymptotic behaviour of the integral. The second type of integrals are relevant in evaluating the geometrical optics approximation in high-frequency wave scattering. Based on numerical evidence we found that the recurrences incur only weak (algebraic) error growth as long as $N\lesssim \omega$, meaning the recurrences are a suitable means for computing the quadrature moments for most practical purposes. We provided rigorous convergence results that allow the understanding of both $\nu$- and $\omega$-dependency of the quadrature error and showed an application to high-frequency wave scattering of a Gaussian beam on a finite plate in two dimensions.
			
			In the application to wave scattering problems we found that even when initial moments cannot be expressed explicitly in terms of simple functions, one may still be able to provide an expression that is easy to evaluate numerically (for instance an exponentially decaying integral). This is closely related to the ideas of numerical steepest descent (NSD), where oscillatory integrals are written in terms of exponentially decaying integrals by moving to the complex plane. However, NSD requires analyticity of the integrand at least in a neighbourhood of the domain of integration, which is in contrast to Filon methods that only require the weaker condition that amplitude is well-approximable by polynomials. Therefore we believe future research could focus on combining the two approaches -- by providing a polynomial approximation to an amplitude of limited analyticity followed by the use of NSD to help evaluate the oscillatory integral over the polynomial which is an entire function.
			
			\section*{Acknowledgements}
			The authors would like to thank Alfredo Dea\~no (Universidad Carlos III de Madrid), Andrew Gibbs (University College London), Daan Huybrechs (KU Leuven), Anastasia Kisil (University of Manchester) and Sheehan Olver (Imperial College London) for several interesting discussions about highly oscillatory quadrature and special functions. We also thank Victor Domínguez (Universidad P\'ublica de Navarra) for sharing the Matlab implementation for his Filon method with us. Finally, the authors gratefully acknowledge support from the UK Engineering and Physical Sciences Research Council (EPSRC) grant EP/L016516/1 for the University of Cambridge Centre for Doctoral Training, the Cambridge Centre for Analysis. GM also gratefully acknowledges funding from the European Research Council (ERC) under the European Union’s Horizon 2020 research and innovation programme (grant agreement No. 850941).
			
			\newpage
			\bibliographystyle{siam}
			\bibliography{IMANUM-refs}

			\appendix

			\begin{appendix}
			\section{Fast interpolation at Clenshaw--Curtis points, mid- and endpoint derivatives}\label{app:proof_of_fast_interpolation_at_CC_and_midpoints}
			\setcounter{section}{1}
			Here we provide a some more detail on how the interpolation problem \eqref{eqn:interpolation_problem_p2} can be solved at cost $\mathcal{O}(\nu\log\nu+s\nu+s^3)$ as described in \S\ref{sec:fast_interpolation_at_FCC_points}. Recall that we wish to solve the following interpolation problem:
			\begin{align*}
				q^{(j)}(0)=f^{(j)}(0),\,\, q^{(j)}(\pm1)=f^{(j)}(\pm1),\,\, j=0,\dots,s\,\,\,\,\text{\ and\ }\,\,\,\,q(c_n)=f(c_n),\,\, n=1,\dots,\nu,
			\end{align*}
			using an expansion in Chebyshev polynomials $q(x)=\sum_{n=0}^{\nu+3s+1}q_nT_n(x)$. Let us adopt the notation used by \cite{Gao2017a} and define
			\begin{align*}
				\hat{q}_0&=2q_0, \quad \hat{q}_k=q_k, \quad k=1,\dots,\nu,\quad \hat{q}_{\nu+1}=2q_{\nu+1},\\
				h_j&=f_j-\sum_{m=\nu+2}^{\nu+3s}q_m\cos\left(\frac{jm\pi}{\nu+1}\right),\quad  f_j=f\left(\cos\frac{j\pi}{\nu+1}\right),\quad j=0,\dots,\nu+1.
			\end{align*}
			Then the interpolation conditions $q(c_n)=f(c_n), n=0,\dots,\nu+1$, are equivalent to saying that $\mathcal{C}_{\nu+1}\hat{\bm{q}}=\bm{h}$, where $\mathcal{C}_{\nu+1}$ is the discrete cosine transform DCT-I. The inverse is
			\begin{align}\label{eqn:inverse_DCT-I_for_pmhat}
				\hat{q}_m=\left(\mathcal{C}_{\nu+1}^{-1}\bm{h}\right)_m=\frac{2}{\nu+1}\sideset{}{''}\sum_{j=0}^{\nu+1}h_j \cos \left(\frac{jm\pi}{\nu+1} \right) \quad \text{for\ }m = 0, \dots, \nu+1,
			\end{align}
			where $\sideset{}{''}\sum_{j=0}^{\nu+1}$ means that for $j=0$ and $j=\nu+1$ the terms are halved. We can simplify the expressions for $h_j$ as follows:
			\begin{align}\label{eqn:simplification_hj}
				h_j&=f_j-\sum_{m=\nu+2}^{\nu+3s+1}\cos\left(\frac{jm\pi}{\nu+1}\right)q_m=f_j-\sum_{m=1}^{3s}(-1)^j\cos\left(\frac{jm\pi}{\nu+1}\right)q_{\nu+1+m},\quad j=0,\dots,\nu+1.
			\end{align}
			Using \eqref{eqn:simplification_hj} in \eqref{eqn:inverse_DCT-I_for_pmhat} we find for $m=0,\dots, \nu+1$ (and with $\check{\bm{q}}=\mathcal{C}_{\nu+1}^{-1}\bm{f}$):
			\begin{align*}
				\hat{q}_m&=\check{q}_m-\frac{2}{\nu+1}\sum_{n=1}^{3s}q_{\nu+1+n}\left[\sideset{}{''}\sum_{j=0}^{\nu+1}(-1)^j\cos\left(\frac{jn\pi}{\nu+1}\right) \cos \left(\frac{jm\pi}{\nu+1} \right) \right]\\
				&=\check{q}_m-\frac{1}{\nu+1}\sum_{n=1}^{3s}q_{\nu+1+n}\left[\sideset{}{''}\sum_{j=0}^{\nu+1}(-1)^j\cos\left(\frac{j(n+m)\pi}{\nu+1}\right)+ \sideset{}{''}\sum_{j=0}^{\nu+1}(-1)^j\cos\left(\frac{j(n-m)\pi}{\nu+1}\right) \right].
			\end{align*}
			Now because $\nu$ is odd, one can quickly check using standard trigonometric identities that
			\begin{align*}
				\sideset{}{''}\sum_{j=0}^{\nu+1}(-1)^j\cos\left(\frac{j(n+m)\pi}{\nu+1}\right)=\begin{cases}
					0,&n+m\neq \nu+1\\
					\nu+1,&n+m=\nu+1
				\end{cases}.
			\end{align*}
			Thus we find $\hat{q}_m=\check{q}_m-\sum_{n=1}^{3s}q_{\nu+1+n}(\delta_{n+m,\nu+1}+\delta_{n-m,\nu+1})$ which implies
			\begin{align*}
				q_n&=\frac{1}{2}\check{q}_n,\,\, n=0,\nu+1,\quad	q_n=\check{q}_n,\,\, n=1,\dots,\nu-3s,\quad q_n=\check{q}_n-q_{2\nu-n+2},\,\, n=\nu-3s+1,\dots,\nu.
			\end{align*}
			The remaining interpolation conditions $q^{(j)}(0)=f^{(j)}(0), q^{(j)}(\pm1)=f^{(j)}(\pm1),j=1,\dots,s,$ are equivalent to the following $3s\times 3s$ system allowing us to find $q_{\nu+2},\dots,q_{\nu+3s+1}$:
			\begin{align*}
				\sum_{n=1}^{3s}q_{\nu+1+n}\left[T_{\nu+1+n}^{(j)}(-1)-T_{\nu+1-n}^{(j)}(-1)\right]&=f^{(j)}(-1)-\sideset{}{''}\sum_{n=0}^{\nu+1}\check{q}_nT_n^{(j)}(-1)\\
				\sum_{n=1}^{3s}q_{\nu+1+n}\left[T_{\nu+1+n}^{(j)}(0)-T_{\nu+1-n}^{(j)}(0)\right]&=f^{(j)}(0)-\sideset{}{''}\sum_{n=0}^{\nu+1}\check{q}_nT_n^{(j)}(0)\\
				\sum_{n=1}^{3s}q_{\nu+1+n}\left[T_{\nu+1+n}^{(j)}(1)-T_{\nu+1-n}^{(j)}(1)\right]&=f^{(j)}(1)-\sideset{}{''}\sum_{n=0}^{\nu+1}\check{q}_nT_n^{(j)}(1).
			\end{align*}
			Note that the coefficients in this linear system can be found explicitly:
			\begin{align*}
				T_n^{(j)}(\pm1)&=(\pm 1)^{n-j}\frac{2^j j! n(n+j-1)!}{(2j)!(n-j)!}, \quad \text{for\ }0\leq j\leq n \text{\ and\ }n+j\geq 1,\\
				T_{n}^{(j)}(0)&=\begin{cases}
					(-1)^{r}\frac{n(n-r-1)!}{r!}2^{j-1},&\quad r=(n-j)/2\in \mathbb{N}\cup \{0\}\\
					0,&\quad\text{otherwise,}
				\end{cases}
			\end{align*}
			where the former expression is proved in \cite[Eq.~(2.3)]{Gao2017a} and the latter follows from the expansion of $T_n$ in the usual monomial basis \cite[Eq.~22.3.6]{abramowitz1965handbook}.
			\section{Proof of Thm.~\ref{thm:full_stability_quadratic_oscillator}}\label{app:proof_of_thm_full_stability_quadratic}
			\setcounter{section}{2}
			We recall the statement of Thm.~\ref{thm:full_stability_quadratic_oscillator}:
			\begin{theorem}
				Suppose the moments $\check{\tilde{\sigma}}_n$ are computed using \eqref{eqn:recurrence_quadratic_oscillator_chebyshev_moments1} with slightly perturbed initial conditions: $\check{\tilde{\sigma}}_0=\tilde{\sigma}_0+\epsilon_0$, $\check{\tilde{\sigma}}_2=\tilde{\sigma}_2+\epsilon_2$, for some $|\epsilon_0|,|\epsilon_2|<\epsilon$. Then, for any $n$ with $2n+1<\omega$,
				\begin{align*}
					|\check{\tilde{\sigma}}_{2n}-\tilde{\sigma}_{2n}|< \frac{8n\omega^{\frac{1}{2}}}{3\left(\omega^2-(2n+1)^2\right)^{\frac{1}{4}}}\left(2+\frac{1}{\omega}\right)\epsilon.
				\end{align*}
			\end{theorem}
			\begin{proof}
				Let $x_n=\check{\tilde{\sigma}}_n-\tilde{\sigma}_n$. By linearity it suffices to solve the recurrence \eqref{eqn:recurrence_quadratic_oscillator_chebyshev_moments1} for $x_n$ with initial conditions $x_0=\epsilon_0,\,x_2=\epsilon_2$, and $x_{2n+1}=0$.
				Substitute $x_n=\rho_{n+1}-\rho_{n-1}, n\geq 0,$ with $\rho_{-1}:=-\epsilon_0/2$ and let $\gamma_{n}=\rho_{n+2}+\frac{2n}{i\omega}\rho_n-\rho_{n-2}$ for $n\geq 1$. We also note that $x_{2n+1}=0$ so we may, without loss of generality, choose $\rho_{2n}=0, n\geq 0$. Then the recurrence \eqref{eqn:recurrence_quadratic_oscillator_chebyshev_moments1} is equivalent to solving
				\begin{align*}
					\gamma_{n+2}-2\gamma_n+\gamma_{n-2}=0, \quad n\geq 3,
				\end{align*}
				with the initial conditions $\gamma_{1}=\gamma_3=\rho_{3}+\frac{2}{i\omega}\rho_1-\rho_{-1}=\epsilon_2+\epsilon_0(1+\frac{1}{i\omega}).$
				Here equality of $\gamma_1=\gamma_3$ was achieved by setting $\rho_{-1}=-\epsilon_0/2$ and using \eqref{eqn:recurrence_quadratic_oscillator_chebyshev_moments1} for $n=1$. Hence we have $\gamma_{2n+1}=\gamma_1, \,\forall n\geq1$. Therefore, the problem of finding $x_{2n}$ from given initial conditions is equivalent to
				\begin{align}\label{eqn:error_initial_value_problem}
					\rho_{n+2}+\frac{2n}{i\omega}\rho_n-\rho_{n-2}=\gamma_1, \, n\geq 1,\,\text{subject to\ }\rho_{1}=\epsilon_0/2,\rho_3=\epsilon_{2}+\epsilon_0/2 \text{\ and\ }\rho_{2n}=0,
				\end{align}
				where $\gamma_1=x_2+x_0(1+\frac{1}{i\omega})$. To solve this let us consider the homogeneous recurrence
				\begin{align}\label{eqn:homogeneous_recurrence_for_y}
					a_{2n+3}+\frac{2(2n+1)}{i\omega} a_{2n+1}-a_{2n-1}=0.
				\end{align}
				This has two linearly independent solutions that can be expressed in terms of spherical Bessel functions $j_n,y_n$ (see \cite[\S 10.1]{abramowitz1965handbook}) namely
				\begin{align*}
					a_{2n+1}=A\, \revisions{i^n}j_n\left(\frac{\omega}{2}\right)+B\,\revisions{i^n}y_n\left(\frac{\omega}{2}\right),
				\end{align*}
				for $A,B\in\mathbb{C}$. Let us write $\tilde{j}_n(z):=i^n\revisions{j_n}(z)$ and $\tilde{y}_n(z):=i^n \revisions{y_n}(z)$, then the solution to \eqref{eqn:homogeneous_recurrence_for_y} with initial conditions $a_1,a_3$ is given by
				\begin{align*}
					\begin{pmatrix}
						a_{2n+3}\\
						a_{2n+1}
					\end{pmatrix}=\begin{pmatrix}
						\tilde{j}_{n+1}(\omega/2)&\tilde{y}_{n+1}(\omega/2)\\
						\tilde{j}_{n}(\omega/2)&\tilde{y}_{n}(\omega/2)
					\end{pmatrix}
					\begin{pmatrix}
						\tilde{j}_{1}(\omega/2)&\tilde{y}_{1}(\omega/2)\\
						\tilde{j}_{0}(\omega/2)&\tilde{y}_{0}(\omega/2)
					\end{pmatrix}^{-1}
					\begin{pmatrix}
						a_3\\
						a_1
					\end{pmatrix},\quad n\geq 1
				\end{align*}
				Now we have the following useful identity \cite[Eq.~10.1.31]{abramowitz1965handbook}:
				\begin{align*}
					\tilde{j}_{n}(z)\tilde{y}_{n-1}(z)-\tilde{j}_{n-1}(z)\tilde{y}_{n}(z)=(-1)^{n+1}iz^{-2}, \quad n\geq 1.
				\end{align*}
				Hence we have
				\begin{align*}
					\mathrm{det}\begin{pmatrix}
						\tilde{j}_{n+1}(\omega/2)&\tilde{y}_{n+1}(\omega/2)\\
						\tilde{j}_{n}(\omega/2)&\tilde{y}_{n}(\omega/2)
					\end{pmatrix}=(-1)^{n+1}i\left(\frac{2}{\omega}\right)^{2}.
				\end{align*}
				Thus we can write the solution to \eqref{eqn:error_initial_value_problem}, by discrete variation of constants, as
				\begin{align}
					\hspace{-0.4cm}\begin{split}\label{eqn:discrete_variation_of_constants_quadratic_oscillator}
						\begin{pmatrix}
							\rho_{2n+3}\\
							\rho_{2n+1}
						\end{pmatrix}&=i\left(\frac{\omega}{2}\right)^{2}\begin{pmatrix}
							\tilde{j}_{n+1}(\omega/2)&\tilde{y}_{n+1}(\omega/2)\\
							\tilde{j}_{n}(\omega/2)&\tilde{y}_{n}(\omega/2)
						\end{pmatrix}\left[\sum_{k=1}^n (-1)^{k+1}\begin{pmatrix}
							\tilde{y}_{k}(\omega/2)&-\tilde{y}_{k+1}(\omega/2)\\
							-\tilde{j}_{k}(\omega/2)&\tilde{j}_{k+1}(\omega/2)
						\end{pmatrix}\begin{pmatrix}
							\gamma_1\\
							0
						\end{pmatrix}\right.\\
						&\quad \hspace{6cm}+\left.(-1)\begin{pmatrix}
							\tilde{y}_{0}(\omega/2)&-\tilde{y}_{1}(\omega/2)\\
							-\tilde{j}_{0}(\omega/2)&\tilde{j}_{1}(\omega/2)
						\end{pmatrix}\begin{pmatrix}
							\rho_{3}\\
							\rho_{1}
						\end{pmatrix}\right]
					\end{split}
				\end{align}
				Now we note according to \cite[Eq.~(1) \S 13.74]{watson1995treatise} for $z\geq\nu+1/2\geq 1$ 
				\begin{align*}
					|j_\nu(z)|^2+|y_\nu(z)|^2<\frac{1}{|z|\sqrt{z^2-(\nu+1/2)^2}}.
				\end{align*}
				Thus we can apply Cauchy--Schwarz to \eqref{eqn:discrete_variation_of_constants_quadratic_oscillator} and find
				\begin{align*}
					|\rho_{2n+3}|&\leq \left(\frac{\omega}{2}\right)^2 \left(\frac{1}{\left|\frac{\omega}{2}\right|\sqrt{\left(\frac{\omega}{2}\right)^2-(n+3/2)^2}}\right)^{\frac{1}{2}}\left[\sum_{k=1}^n\left(\frac{1}{\left|\frac{\omega}{2}\right|\sqrt{\left(\frac{\omega}{2}\right)^2-(k+1/2)^2}}\right)^{\frac{1}{2}}|\gamma_1|\right.\\
					&\quad \hspace{2cm}\left.+\left(\frac{1}{\left|\frac{\omega}{2}\right|\sqrt{\left(\frac{\omega}{2}\right)^2-(3/2)^2}}\right)^{\frac{1}{2}}|\rho_1|+ \left(\frac{1}{\left|\frac{\omega}{2}\right|\sqrt{\left(\frac{\omega}{2}\right)^2-(1/2)^2}}\right)^{\frac{1}{2}}|\rho_3|\right]\\
					&\leq \frac{\omega}{\left(\omega^2-(2n+3)^2\right)^{\frac{1}{4}}}\left[\sum_{k=1}^n\frac{1}{\left(\omega^2-(2k+1)^2\right)^{\frac{1}{4}}}|\gamma_1|+\frac{1}{\left(\omega^2-3^2\right)^{\frac{1}{4}}}|\rho_1|+\frac{1}{\left(\omega^2-1\right)^{\frac{1}{4}}}|\rho_3|\right].
				\end{align*}
				Finally, we notice by the integral test for $2n+1<\omega$:
				\begin{align*}
					\sum_{k=1}^n\frac{1}{(\omega^2-(2k+1)^2)^{\frac{1}{4}}}&\leq \int_{0}^{n} \frac{1}{(\omega+2x+1)^{\frac{1}{4}}}\frac{1}{(\omega-2x-1)^{\frac{1}{4}}}dx\leq \omega^{-\frac{1}{4}}\int_0^n\frac{dx}{(\omega-2x-1)^{\frac{1}{4}}}\\
					&\revisions{=}\,\, \frac{2}{3}\omega^{-\frac{1}{4}}\left((\omega-1)^{\frac{3}{4}}-(\omega-2n-1)^{\frac{3}{4}}\right)\\
					&\leq \frac{2}{3}\omega^{-\frac{1}{2}}\left(\omega-(\omega-(2n+1))\right)=\omega^{-\frac{1}{2}}\frac{2(2n+1)}{3}
				\end{align*}
				Thus we have overall
				\begin{align*}
					|\rho_{2n+3}|\leq \frac{\omega^{\frac{1}{2}}}{\left(\omega^2-(2n+3)^2\right)^{\frac{1}{4}}}\left[\frac{2(2n+1)}{3}\left(2+\frac{1}{\omega}\right)+\frac{1}{2\left(1-\frac{\revisions{9}}{\omega^2}\right)^\frac{1}{4}}+\frac{3}{2\left(1-\frac{\revisions{1}}{\omega^2}\right)^\frac{1}{4}}\right]\epsilon
				\end{align*}
				and a similar estimate holds for $\rho_{2n+1}$ and hence the result follows, since $x_{2n+2}=\rho_{2n+3}-\rho_{2n+1}$.
			\end{proof}
			\section{Proof of Thm.~\ref{thm:square_root_stability_algebraic_singularity}}\label{app:proof_of_thm_square_root_stability_algebraic_singularity} \setcounter{section}{3}
			We recall the statement of Thm.~\ref{thm:square_root_stability_algebraic_singularity}:
			\begin{theorem}
				Suppose the moments $\check{\tilde{\sigma}}_n$ are computed using \eqref{eqn:recurrence_algebraic_singularities_chebyshev_moments}
				%\begin{align}\label{eqn:recurrence_algebraic_singularity_chebyshev_moments}
				%	\check{\tilde{\sigma}}_{n-3}+\frac{2(-(n-3)+\alpha)}{	i\omega}\check{\tilde{\sigma}}_{n-2}-\check{\tilde{\sigma}}_{n-1}+\frac{4-4\alpha}{	i\omega}\check{\tilde{\sigma}}_n-\check{\tilde{\sigma}}_{n+1}+\frac{2(n+3+\alpha)}{i\omega}\check{\tilde{\sigma}}_{n+2}+ \check{\tilde{\sigma}}_{n+3}=0,\quad n\geq 0,
				%\end{align}
				with the perturbed initial conditions $\check{\tilde{\sigma}}_0=\tilde{\sigma}+\epsilon_0$, $\check{\tilde{\sigma}}_1=\check{\tilde{\sigma}}_{-1}=\tilde{\sigma}_1+\epsilon_1$, $\check{\tilde{\sigma}}_2=\check{\tilde{\sigma}}_{-2}=\tilde{\sigma}_2+\epsilon_2$, $|\epsilon_j|<\epsilon$ for some $\epsilon>0$, and assume $\check{\tilde{\sigma}}_3=\check{\tilde{\sigma}}_{-3}$. Then, whenever $	n+1<\min\{C\sqrt{\omega},\omega\}$ for a given $C>0$, we have
				\begin{align*}
					|\check{\tilde{\sigma}}_n-\tilde{\sigma}_n|\leq \frac{(K_0+nK_1)}{2}\epsilon \left(\frac{K_2\omega^{\frac{1}{2}}}{K_2\omega^{\frac{1}{2}}-1}\exp\left(\frac{C}{K_2-\omega^{-\frac{1}{2}}}\right)+1\right)
				\end{align*}
				where the constants $K_0,K_1,K_2$ are independent of $n$ and are given by 
				\begin{align*}
					K_0= \frac{2\sqrt{\omega}}{\sqrt{\omega-C^2}},\quad K_1=\frac{\omega+2+|\alpha|}{\sqrt{\omega^2-C^2\omega}},\quad K_2=\frac{\left(\omega-C^2\right)^{\frac{1}{4}}}{\omega^{\frac{1}{4}}\sqrt{2|\alpha|+2}}.
				\end{align*}
			\end{theorem}
			\begin{proof}
				Define $x_n:=\check{\tilde{\sigma}}_n-\tilde{\sigma}_n$ then, by linearity it suffices to solve for $x_n$ which satisfies \eqref{eqn:recurrence_algebraic_singularities_chebyshev_moments} subject to $x_0=\epsilon_0,x_{1}=x_{-1}=\epsilon_1,x_{2}=x_{-2}=\epsilon_2, x_{3}=x_{-3}$. We can formulate the recurrence equivalently in the form
				\begin{align*}
					x_{n-3}-\frac{2(n-2)}{	i\omega}x_{n-2}-x_{n-1}+\frac{4}{	i\omega}x_n-x_{n+1}&+\frac{2(n+2)}{i\omega}x_{n+2}+ x_{n+3}\\
					&=-\frac{2\alpha}{	i\omega}(x_{n-2}-2x_n+x_{n+2})-\frac{2}{i\omega}(x_{n-2}+2x_n+x_{n+2}),
				\end{align*}
				for $n\geq 0$. We can solve the homogeneous difference equation corresponding to the left hand side exactly, and we view the right hand side as a perturbation of the recurrence in the following sense: Let
				\begin{align*}
					x_{n}=\sum_{j=0}^{n-2}\omega^{-j}x_{n}^{(j)},\quad n\geq 2 .
				\end{align*}
				Then $x_n$ is the unique solution of the recurrence \eqref{eqn:recurrence_algebraic_singularities_chebyshev_moments} with the specified initial conditions if we define
				\begin{align}\label{eqn:zeroth_order_equation_algebraic_singularities}
					x_{n-3}^{(0)}-\frac{2(n-2)}{	i\omega}x_{n-2}^{(0)}-x_{n-1}^{(0)}-x_{n+1}^{(0)}+\frac{2(n+2)}{i\omega}x_{n+2}^{(0)}+ x_{n+3}^{(0)}=0,\quad n\geq 0,
				\end{align}
				with the initial conditions $x_n^{(0)}=\epsilon_0,x_{-1}^{(0)}=x_{1}^{(0)}=\epsilon_1,x_{-2}^{(0)}=x_{2}^{(0)}=\epsilon_2$, and
				\begin{align*}
					x_{-3}^{(0)}=x^{(0)}_{3}&=\frac{2\alpha-2}{	i\omega}\epsilon_0+\epsilon_{1}-\frac{2(3+\alpha)}{i\omega}\epsilon_{2},
				\end{align*}
				and if we further choose
				\begin{align}
					\begin{split}\label{eqn:j+1_from_j_recurrence}
						x_{n-3}^{(j+1)}-\frac{2(n-2)}{	i\omega}x_{n-2}^{(j+1)}-x_{n-1}^{(j+1)}-x_{n+1}^{(j+1)}+\frac{2(n+2)}{i\omega}x_{n+2}^{(j+1)}+ x_{n+3}^{(j+1)}&=2i\alpha(x_{n-2}^{(j)}-2x_n^{(j)}+x_{n+2}^{(j)})\\
						&\quad+2i(x_{n-2}^{(j)}+2x_n^{(j)}+x_{n+2}^{(j)})
					\end{split}
				\end{align}
				for $j\geq 0$ and $n\geq j$, under the extra symmetry condition $x^{(j)}_n=x^{(j)}_{-n}$ and with the initial conditions $x_{j+2}^{(j+1)},x_{j+1}^{(j+1)},x_{j}^{(j+1)},x_{j-1}^{(j+1)}=0$.
				
				Let us firstly solve \eqref{eqn:zeroth_order_equation_algebraic_singularities}: We let $\gamma_{n}^{(0)}=x_{n+1}^{(0)}+\frac{2n}{i\omega}x_n^{(0)}-x_{n-1}^{(0)},\,n\geq -2$, which ensures that \eqref{eqn:zeroth_order_equation_algebraic_singularities} is equivalent to
				\begin{align*}
					\gamma_{n+2}^{(j+1)}-\gamma_{n-2}^{(j+1)}&=0,\quad n\geq 0\\
					\gamma_{-2}^{(0)}=\frac{2-2\alpha}{i\omega}\epsilon_0+\frac{2+2\alpha}{i\omega}\epsilon_2,\gamma_{-1}^{(0)}&=\epsilon_0-\frac{2}{i\omega}\epsilon_1-\epsilon_2,\gamma^{(0)}_{0}=0,\gamma^{(0)}_{1}=\epsilon_{2}+\frac{2}{i\omega}\epsilon_1-\epsilon_0,%\gamma^{(0)}_{2}=\epsilon_{3}+\frac{4}{i\omega}\epsilon_2-\epsilon_1,\gamma^{(0)}_{3}=\epsilon_{4}+\frac{6}{i\omega}\epsilon_3-\epsilon_2.
				\end{align*}
				Thus $\gamma_{4n+j}^{(0)}=\gamma_j^{(0)}$ for $j=-2,\dots,1$ and $n\geq 0$. Hence, it remains to solve
				\begin{align}\label{eqn:Bessel_recurrence_for_zero_order_terms}
					x_{n+1}^{(0)}+\frac{2n}{i\omega}x_n^{(0)}-x_{n-1}^{(0)}=\gamma_{n}^{(0)},\quad n\geq 1,
				\end{align}
				with initial conditions $x_0^{(0)}=\epsilon_0,x_{1}^{(0)}=\epsilon_1$. As described by \cite{Dominguez2011} the homogeneous solutions of this recurrence can be expressed in terms of Bessel functions, where it will be convenient to express the solutions in terms of the functions $\tilde{J}_{n}(\omega):=i^n J_n(\omega),\, \tilde{Y}_n(\omega):=i^nY_n(\omega)$, where $J_n(\omega),Y_n(\omega)$ are the standard Bessel functions of the first and second kind, as defined for instance in \cite{abramowitz1965handbook}. The solution to \eqref{eqn:Bessel_recurrence_for_zero_order_terms} can be written using discrete variation of constants as
				\begin{align}
					\begin{split}\label{eqn:full_solution_zeroth_order_terms}
						\begin{pmatrix}
							x_{n+1}^{(0)}\\
							x_{n}^{(0)}
						\end{pmatrix}&=\frac{i\pi \omega}{2}\begin{pmatrix}
							\tilde{J}_{n+1}(\omega)&\tilde{Y}_{n+1}(\omega)\\
							\tilde{J}_n(\omega)&\tilde{Y}_n(\omega)
						\end{pmatrix}
						\left(\sum_{k=1}^n
						(-1)^{k+1}\begin{pmatrix}
							\tilde{Y}_k(\omega)&-\tilde{Y}_{k+1}(\omega)\\
							-\tilde{J}_k(\omega)&\tilde{J}_{k+1}(\omega)
						\end{pmatrix}
						\begin{pmatrix}
							\gamma_{k}^{(0)}\\
							0
						\end{pmatrix}\right.\\
						&\hspace{6.3cm}\left.+(-1)\begin{pmatrix}
							\tilde{Y}_{0}(\omega)&-\tilde{Y}_{1}(\omega)\\
							-\tilde{J}_{0}(\omega)&\tilde{J}_{1}(\omega)
						\end{pmatrix} \begin{pmatrix}
							\epsilon_{1}\\
							\epsilon_{0}
						\end{pmatrix}\right),
					\end{split}
				\end{align}
				for $n\geq 1$. Here, analogously to \cite{Dominguez2011}, we used the identity \cite[Eq.~(9.1.16)]{abramowitz1965handbook}
				\begin{align*}
					det\begin{pmatrix}
						\tilde{J}_{n+1}(\omega)&\tilde{Y}_{n+1}(\omega)\\
						\tilde{J}_n(\omega)&\tilde{Y}_n(\omega)
					\end{pmatrix}=(-1)^{\revisions{n+2}}\frac{\revisions{2i}}{\pi \omega}
				\end{align*}
				We can now perform a similar estimate to \cite[p. 1271]{Dominguez2011} on \eqref{eqn:full_solution_zeroth_order_terms}: Note the upper bound given by \cite[\S 13.74]{watson1995treatise}
				\begin{align}\label{eqn:bound_on_modulus_of_bessel_functions}
					|J_n(\omega)|^2+|Y_n(\omega)|^2\leq \frac{2}{\pi}\frac{1}{\sqrt{\omega^2-n^2}},\quad\text{for\ }\omega>n>1/2.
				\end{align}
				Combining this with Cauchy--Schwarz on \eqref{eqn:full_solution_zeroth_order_terms} yields, for $n\geq 2$,
				\begin{align*}
					|x_n^{(0)}|\leq \frac{\omega}{(\omega^2-n^2)^{\frac{1}{4}}}\left(\sum_{k=1}^{n-1}\frac{1}{(\omega^2-k^2)^{\frac{1}{4}}}|\gamma_k^{(0)}|+\epsilon\sum_{j=0}^1\frac{1}{(\omega^2-j^2)^{\frac{1}{4}}}\right)
				\end{align*}
				Thus, summing these contributions, we obtain the following estimate when $n<\min\{C\sqrt{\omega},\omega\}$:
				\begin{align}\label{eqn:estimate_for_sigma_0}
					|x_n^{(0)}|\leq \frac{\epsilon }{\sqrt{1-\frac{C^2}{\omega}}} \left(n \left(1+\frac{2+|\alpha|}{\omega}\right)+2\right).
				\end{align}
				We now consider the perturbed recurrence \eqref{eqn:j+1_from_j_recurrence} order by order. To do so let us write
				\begin{align*}
					f_n^{(j)}=2i(1+\alpha)\left(x_{n-2}^{(j)}+x_{n+2}^{(j)}\right)+2i(1-\alpha)x_n^{(j)}, \quad n\geq j.
				\end{align*}
				Thus, for $j\geq 0$, we need to solve the recurrence
				\begin{align*}
					x_{n-3}^{(j+1)}-\frac{2(n-2)}{	i\omega}x_{n-2}^{(j+1)}-x_{n-1}^{(j+1)}-x_{n+1}^{(j+1)}+\frac{2(n+2)}{i\omega}x_{n+2}^{(j+1)}+ x_{n+3}^{(j+1)}=f_n^{(j)},\quad n\geq j
				\end{align*}
				subject to the initial conditions $x_{j+2}^{(j+1)},x_{j+1}^{(j+1)},x_{j}^{(j+1)},x_{j-1}^{(j+1)}=0$. We again substitute $\gamma_{n}^{(j)}=x_{n+1}^{(j)}+\frac{2n}{i\omega}x_n^{(j)}-x_{n-1}^{(j)},n\geq j-2$, which ensures that this recurrence is equivalent to
				\begin{align*}
					\gamma_{n+2}^{(j+1)}-\gamma_{n-2}^{(j+1)}=f_n^{(j)},\quad n\geq j-2,\quad \gamma^{(j+1)}_{j-2},\gamma^{(j+1)}_{j-1},\gamma^{(j+1)}_{j},\gamma^{(j+1)}_{j+1}=0
				\end{align*}
				Therefore we easily find
				\begin{align*}
					\gamma_{4n+j}^{(j+1)}=\sum_{m=0}^{n-1}f_{4m+j+2}^{(j)},\quad
					\gamma_{4n+j+1}^{(j+1)}=\sum_{m=0}^{n-1}f_{4m+j+3}^{(j)},\quad
					\gamma_{4n+j-1}^{(j+1)}=\sum_{m=0}^{n-1}f_{4m+j+1}^{(j)},\quad
					\gamma_{4n+j-2}^{(j+1)}=\sum_{m=0}^{n-1}f_{4m+j}^{(j)}.
				\end{align*}
				Now it remains to solve
				\begin{align*}
					\gamma_{n}^{(j)}=x_{n+1}^{(j)}+\frac{2n}{i\omega}x_n^{(j)}-x_{n-1}^{(j)},\quad n\geq j+2
				\end{align*}
				with the initial conditions $x_{j+2}^{(j+1)},x_{j+1}^{(j+1)}=0$. Similarly to the case for $x_n^{(0)}$ we can write the solution in terms of $\tilde{J}_n(\omega),\tilde{Y}_n(\omega)$, which yields

				\begin{align*}
					\begin{pmatrix}
						x_{n+1}^{(j+1)}\\
						x_{n}^{(j+1)}
					\end{pmatrix}=\frac{i\pi \omega}{2}\begin{pmatrix}
						\tilde{J}_{n+1}(\omega)&\tilde{Y}_{n+1}(\omega)\\
						\tilde{J}_n(\omega)&\tilde{Y}_n(\omega)
					\end{pmatrix}
					\sum_{k=j+2}^n
					(-1)^{k+1}\begin{pmatrix}
						\tilde{Y}_k(\omega)&-\tilde{Y}_{k+1}(\omega)\\
						-\tilde{J}_k(\omega)&\tilde{J}_{k+1}(\omega)
					\end{pmatrix}
					\begin{pmatrix}
						\gamma_{k}^{(j+1)}\\
						0
					\end{pmatrix}.
				\end{align*}
				Therefore, we can use Cauchy--Schwarz and \eqref{eqn:bound_on_modulus_of_bessel_functions} similarly to above to estimate
				\begin{align}\label{eqn:estimate_sigma_j+1_in_gamma_j+1}
					|x_n^{(j+1)}|\leq \frac{\pi\omega}{2}\frac{2}{\pi}\frac{1}{(\omega^2-n^2)^{\frac{1}{4}}}\sum_{k=j+2}^{n-1}\frac{1}{(\omega^2-k^2)^{\frac{1}{4}}}|\gamma_{k}^{(j+1)}|
				\end{align}
				Now we recall $f_n^{(j)}=2i\alpha(x_{n-2}^{(j)}-2x_n^{(j)}+x_{n+2}^{(j)})+2i(x_{n-2}^{(j)}+2x_n^{(j)}+x_{n+2}^{(j)})$ which means that
				\begin{align*}
					\gamma_{4n+j}^{(j+1)}&=\sum_{m=0}^{n-1}f_{4m+j+2}^{(j)}=2i\alpha\left(x_{4n+j}^{(j)}+2\sum_{l=1}^{2n-1}(-1)^lx_{j+2l}^{(j)}\right)+2i\left(x_{4n+j}^{(j)}+2\sum_{l=1}^{2n-1}x_{j+2l}^{(j)}\right),\\
					\therefore\quad\quad|\gamma_{4n+j}^{(j+1)}|&\leq (2\alpha+2)\left(|x_{4n+j}^{(j)}|+2\sum_{l=1}^{2n-1}|x_{j+2l}^{(j)}|\right)
				\end{align*}
				Analogously we find for $k=-2,-1,1$:
				\begin{align*}
					|\gamma_{4n+k+j}^{(j+1)}|&\leq (2\alpha+2)\left(|x_{4n+k+j}^{(j)}|+2\sum_{l=1}^{2n-1}|x_{j+2l+k}^{(j)}|\right),
				\end{align*}
				where, of course, $x_{j}^{(j)},x_{j+1}^{(j)}=0$. To complete a total estimate on the size of $x_{n}^{(j+1)}$ we proceed as follows:
				\begin{claim}
					If $|x^{(j)}_{n+j}|\leq n^b$, for $b\geq 0$ and all $n+j+1<C\sqrt{\omega}$, then
					\begin{align*}
						|x_{n+j+1}^{(j+1)}|\leq \frac{2+2|\alpha|}{\sqrt{1-\frac{C^2}{\omega}}}\left(\frac{n^{b+2}}{(b+2)(b+1)}+2\frac{n^{b+1}}{b+1}+n^b\right).
					\end{align*}
				\end{claim}
				\begin{flushleft}
					\textit{Proof of Claim.}
				\end{flushleft}
				\begin{align*}
					|\gamma_{4n+j}^{(j+1)}|&\leq (2|\alpha|+2)\left(|x_{4n+j}^{(j)}|+2\sum_{l=1}^{2n-1}|x_{j+2l}^{(j)}|\right)\leq C(2|\alpha|+2)\left((4n)^b+2\sum_{l=1}^{2n-1} (2l)^b\right)\\
					&\leq C(2|\alpha|+2)\left((4n)^b+2\int_{0}^{2n}(2x)^b\,dx\right)=C(2|\alpha|+2)\left((4n)^b+\frac{1}{b+1}(4n)^{b+1}\right),
				\end{align*}
				where in the final step we used the integral test to find an upper bound. Analogously, we find in general
				\begin{align*}
					|\gamma_{k+j}^{(j+1)}|&\leq C(2|\alpha|+2)\left(\frac{1}{b+1}k^{b+1}+k^b\right), \quad k \geq 2.
				\end{align*}
				Thus we have, based on \eqref{eqn:estimate_sigma_j+1_in_gamma_j+1},
				\begin{align*}
					|x_{n+j+1}^{(j+1)}|&\leq \frac{\omega}{(\omega^2-(n+j+1)^2)^{\frac{1}{4}}}\sum_{k=j+2}^{n+j}\frac{1}{(\omega^2-k^2)^{\frac{1}{4}}}|\gamma_{k}^{(j+1)}|\\
					&\leq \frac{C(2|\alpha|+2)\omega}{(\omega^2-(n+j+1)^2)^{\frac{1}{4}}}\sum_{k=2}^{n}\frac{1}{(\omega^2-(j+k)^2)^{\frac{1}{4}}}\left(\frac{1}{b+1}k^{b+1}+k^b\right).
				\end{align*}
				Thus, if $n+1+j<C\sqrt{\omega}$, we can simplify the above estimate to complete the proof of the claim:
				\begin{align*}
					|x_{n+j+1}^{(j+1)}|&\leq \frac{C(2|\alpha|+2)}{\sqrt{1-\frac{C^2}{\omega}}}\sum_{k=2}^{n}\frac{1}{b+1}k^{b+1}+k^b\leq \frac{C(2|\alpha|+2)}{\sqrt{1-\frac{C^2}{\omega}}}\left[\frac{n^{b+2}}{(b+2)(b+1)}+2\frac{n^{b+1}}{(b+1)}+n^b\right].
				\end{align*}
				\begin{flushright}
					$\blacksquare$
				\end{flushright}
				Now we have shown in \eqref{eqn:estimate_for_sigma_0} that 
				\begin{align*}
					|x_n^{(0)}|\leq n\epsilon \underbrace{ \frac{ \left(1+\frac{2+|\alpha|}{\omega}\right)}{\sqrt{1-\frac{C^2}{\omega}}} }_{K_1}+\epsilon\underbrace{ \frac{2}{\sqrt{1-\frac{C^2}{\omega}}}}_{K_0}.
				\end{align*}
				Thus we have by linearity for $n\geq 2$
				\begin{align}\nonumber
					|x_{n}|\leq \sum_{j=0}^{n-2}\omega^{-j}\left|x_{n}^{(j)}\right|
					&\leq\sum_{j=0}^{n-2}\left(\frac{(2|\alpha|+2)}{\omega\sqrt{1-\frac{C^2}{\omega}}}\right)^{j}\left[\sum_{l=0}^{2j}\binom{2j}{l}\frac{(n-j)^{l}}{l!}\left(K_0\epsilon+K_1\epsilon\frac{n-j}{l+1}\right)\right]\\\label{eqn:detailed_upper_bound_sigma_n}
					&\leq (K_0+nK_1)\epsilon\sum_{j=0}^{n-2}\left(\frac{(2|\alpha|+2)}{\omega\sqrt{1-\frac{C^2}{\omega}}}\right)^{j}L_{2j}(-n+j)
				\end{align}
				where $L_{2j}$ are Laguerre polynomials and the final line follows from the explicit expansion of $L_{2j}$ in the usual monomial basis \cite[Eq.~22.3.9]{abramowitz1965handbook}. Thus we seek to find an upper bound for the function $f(z,n)= \sum_{j=0}^{n-2}z^{-j}L_{2j}(-n+j)$. Note to begin with that Laguerre polynomials are strictly monotonically decreasing for negative arguments, i.e. for $x<y<0$ we have for any $n\geq1$
				\begin{align*}
					L_{n}(x)>L_{n}(y)>L_n(0)=1, \text{\ and\ }L_0\equiv 1.
				\end{align*}
				This follows by induction from the identity \cite[Eqns. 22.5.17 \& 22.7.30]{abramowitz1965handbook}
				\begin{align*}
					\frac{d}{dx}L_{n+1}=\frac{d}{dx}L_{n}-L_{n},\quad n\geq 0,
				\end{align*}
				since we have $L_n(x)>0$ for any $x<0, n\geq0$ (because the zeros of the Laguerre polynomials are located in $[0,\infty)$ and $L_n(0)=1$). Thus we may estimate
				\begin{align*}
					f(z,n)&\leq \sum_{j=0}^{n-2}\left(z^{\frac{1}{2}}\right)^{-2j}L_{2j}(-n)\leq \frac{1}{2}\sum_{j=0}^{n-2}\left(z^{\frac{1}{2}}\right)^{-2j}L_{2j}(-n)+\frac{1}{2}\sum_{j=0}^{n-2}\left(-z^{\frac{1}{2}}\right)^{-2j}L_{2j}(-n)
				\end{align*}
				By the three-term recurrence for Laguerre polynomials \cite[22.7.12]{abramowitz1965handbook}, 
				\begin{align*}
					L_{n+1}(x)=L_n(x)-\frac{x}{n+1}L_n+\frac{n}{n+1}(L_n(x)-L_{n-1}(x)),
				\end{align*}
				and by induction we have $	L_{n+1}(x)>L_n(x)$ for any $x<0,n\geq 0$. Therefore
				\begin{align}\label{eqn:new_estimate_2}
					f(z,n)=f(z,n)+z^{-\frac{1}{2}}f(z,n)-z^{-\frac{1}{2}}f(z,n)\leq \frac{1}{2}\sum_{j=0}^{\infty}z^{-\frac{j}{2}}L_{j}(-n)+\frac{1}{2}\sum_{j=0}^\infty\left(-z^{\frac{1}{2}}\right)^jL_{j}(-n).
				\end{align}
				The generating function of the Laguerre polynomials \cite[Eq.~22.9.15]{abramowitz1965handbook} is
				\begin{align*}
					\sum_{j=0}^{\infty}a^{j}L_{j}(x)=\frac{1}{1-a}\exp\left(\frac{xa}{a-1}\right),\quad \forall x\in\mathbb{R},|a|<1,
				\end{align*}
				which allows us to simplify the estimate \eqref{eqn:new_estimate_2} to
				\begin{align*}
					f(z,n)\leq\frac{1}{2}\frac{z^{\frac{1}{2}}}{z^{\frac{1}{2}}-1}\exp\left(\frac{n}{z^{\frac{1}{2}}-1}\right)+\frac{1}{2}\frac{1}{1+z^{\frac{1}{2}}}\exp\left(-\frac{n}{z^{\frac{1}{2}}+1}\right)
				\end{align*}
				and therefore we conclude by \eqref{eqn:detailed_upper_bound_sigma_n}:
				\begin{align*}
					|x_n|\leq \frac{(K_0+nK_1)}{2}\epsilon \left(\frac{K_2\omega^{\frac{1}{2}}}{K_2\omega^{\frac{1}{2}}-1}\exp\left(\frac{n}{\omega^{\frac{1}{2}}K_2-1}\right)+\frac{1}{K_2\omega^{\frac{1}{2}}+1}\exp\left(-\frac{n}{\omega^{\frac{1}{2}}K_2+1}\right)\right)
				\end{align*}
				with $K_2=(2|\alpha|+2)^{-1/2}(1-C^2/\omega)^{1/4}$ and the result follows.
			\end{proof}
			\section{Expression for initial moments in \S\ref{sec:initial_conditions}}\label{app:expression_initial_moments_linear_hankel_integrals}
			\setcounter{section}{4}
			Let us define the standard moments by $\rho_n:=I^{(3)}_{\omega,\beta}[x^n]$, then we have the expressions
			\begin{align*}
				\tilde{\sigma}_0&=\rho_0,\quad\quad
				\tilde{\sigma}_1=2\rho_1-\rho_0,\quad\quad
				\tilde{\sigma}_2=8\rho_2-8\rho_1+\rho_0,\quad\quad
				\tilde{\sigma}_3=32\rho_3-48\rho_2+18\rho_1-\rho_0.
			\end{align*}
			Differentiating \eqref{eqn:closed_form_expression_rho_0} with respect to $\beta$ yields:
			\begin{align*}
				i\omega\rho_1&=-\frac{2i}{\pi}\omega^{-\frac{3}{2}}e^{i(\beta+1)\omega}\int_{0}^{\infty}\left[i\omega g_\beta\left(\frac{t}{\omega}\right)+\left(\partial_{\beta}g_\beta\left(\frac{t}{\omega}\right)\right)\right] \frac{1}{\sqrt{t}}e^{-t}dt+\begin{cases}\frac{1}{1-\beta ^2}+\frac{2 \beta  \tanh ^{-1}\left(\sqrt{\frac{\beta -1}{\beta
								+1}}\right)}{(\beta^2 -1)^{3/2}},&\beta>-1,\beta\neq 1,\\
					1/3,&\beta=1,\\
					\frac{1}{1-\beta ^2}-\frac{2 \beta  \tanh ^{-1}\left(\sqrt{\frac{\beta -1}{\beta+1}}\right)}{(\beta^2 -1)^{3/2}},&\beta<-1,
				\end{cases}
			\end{align*}
			\begin{align*}
				\hspace{-1.2cm}-\omega^2\rho_2&=-\frac{2i}{\pi}\omega^{-\frac{3}{2}}\int_{0}^{\infty}\left[-\omega^2 g_\beta\left(\frac{t}{\omega}\right)+2i\omega\left(\partial_{\beta}g_\beta\left(\frac{t}{\omega}\right)\right)+\left(\partial_{\beta}^2g_\beta\left(\frac{t}{\omega}\right)\right)\right] \frac{1}{\sqrt{t}}e^{-t}dt\\
				&\hspace{5.4cm}+\begin{cases}
					\frac{3 \beta }{\left(\beta ^2-1\right)^2}-\frac{\left(4 \beta ^2+2\right) \tanh
						^{-1}\left(\sqrt{\frac{\beta -1}{\beta +1}}\right)}{(\beta^2 -1)^{5/2}},&\beta>-1,\beta\neq1,\\
					-4/15,&\beta=1,\\
					\frac{3 \beta }{\left(\beta ^2-1\right)^2}+\frac{\left(4 \beta ^2+2\right) \tanh
						^{-1}\left(\sqrt{\frac{\beta -1}{\beta +1}}\right)}{(\beta^2 -1)^{5/2}},&\beta<-1,
				\end{cases}
			\end{align*}
			and, when $\beta\neq1$, we find additionally:
			\begin{align*}
				-i\omega^3\rho_3=& -\frac{2i}{\pi}\omega^{-\frac{3}{2}}e^{i(\beta+1)\omega}\int_{0}^{\infty}\left[-i\omega^3 g_\beta\left(\frac{t}{\omega}\right)-3\omega^2\left(\partial_{\beta}g_\beta\left(\frac{t}{\omega}\right)\right)\right.\\
				&\quad\quad\quad\quad\quad \quad\quad\quad\quad\quad \left.+3i\omega\left(\partial_{\beta}^2g_\beta\left(\frac{t}{\omega}\right)\right)+\left(\partial_{\beta}^3g_\beta\left(\frac{t}{\omega}\right)\right)\right] \frac{1}{\sqrt{t}}e^{-t}dt\\
				&\quad\quad\quad\quad\quad \quad\quad\quad\quad\quad +\frac{2}{\pi}\omega^{-1}\begin{cases}
					\frac{11 \beta ^2+4}{\left(\beta ^2-1\right)^3}-\frac{6 \beta  \left(2 \beta
						^2+3\right) \tanh ^{-1}\left(\sqrt{\frac{\beta -1}{\beta +1}}\right)}{(\beta
						-1)^{7/2} (\beta +1)^{7/2}},&\beta>-1,\beta\neq1,\\ 
					\frac{11 \beta ^2+4}{\left(\beta ^2-1\right)^3}+\frac{6 \beta  \left(2 \beta
						^2+3\right) \tanh ^{-1}\left(\sqrt{\frac{\beta -1}{\beta +1}}\right)}{(\beta
						-1)^{7/2} (\beta +1)^{7/2}},&\beta<-1.
				\end{cases}
			\end{align*}
			\section{Proof of Prop.~\ref{prop:filon_paradigm_for_alpha_zero}}
			\setcounter{section}{5}
			We recall the statement of Prop.~\ref{prop:filon_paradigm_for_alpha_zero}:
			\begin{proposition}[Filon paradigm for $I^{(3)}_{\omega,\beta}$] For any $k\in\mathbb{N}$ there is a constant $C_k>0$ such that for all $\beta\in\mathbb{R},\beta\neq-1,\omega\geq 1$ and any function $\tilde{f}\in C^{k+2}[0,1]$ with $\tilde{f}^{(j)}(\pm1)=0$ for $j=0,\dots,k$:
				\begin{align*}
					\left|I^{(3)}_{\omega,\beta}[\tilde{f}]\right|\leq C_k\left( \omega^{-(k+2)}\|\tilde{f}^{(k+1)}\|_{\infty}\frac{|\beta+1|^{k+3}-1}{|\beta+1|-1}+\omega^{-(k+2)}\log\omega\|\tilde{f}^{(k+2)}\|_{\infty}|\beta+1|^{-(k+2)}\right).
				\end{align*}
			\end{proposition}
			\begin{proof} We write
				\begin{align}\label{eqn:split_of_integral}
					I^{(3)}_{\omega,\beta}[\tilde{f}]=\int_{0}^{\omega^{-1}} H^{(1)}_0(\omega x)\tilde{f}(2x-1)e^{i\omega \beta x}dx+\int_{\omega^{-1}}^1H^{(1)}_0(\omega x)\tilde{f}(2x-1)e^{i\omega\beta x}dx.
				\end{align}
				By Taylor's theorem we have $|\tilde{f}^{(j)}(2x-1)|\leq \tilde{C}_kx^{k+1-j}\|\tilde{f}^{(k+1)}\|_\infty, j=0,\dots,k+1,$ for all $x\in[-1,1]$ and for some constant $\tilde{C}_k>0$ independent of $x$. To bound the first integral note by Lemma \ref{lem:phase_extraction_in_H_0} for $n=0$ that
				\begin{align*}
					\left|H_0^{(1)}(\omega x)\right|&\leq 2C_0(\omega x)^{-1/2},\quad \forall x>0.
				\end{align*}
				since $1+\log(1/z)\leq 2z^{-1/2}$, $|h_0(z)|\leq 2C_0z^{-1/2}$ when $z\leq 1$. Thus we have
				\begin{align}\label{eqn:estimate_close_to_zero}
					\left|\int_{0}^{\omega^{-1}} H^{(1)}_0(\omega x)\tilde{f}(2x-1)e^{i\omega \beta x}dx\right|\leq  2C_0\tilde{C}_k \omega^{-\frac{1}{2}} \|\tilde{f}^{(k+1)}\|_\infty\int_{0}^{\omega^{-1}}x^{k+\frac{1}{2}}dx\lesssim\omega^{-k-2}\|\tilde{f}^{(k+1)}\|_\infty,
				\end{align}
				where $A(\omega)\lesssim B(\omega)$ means $A(\omega)\leq K B(\omega)$ for a constant $K>0$ independent of $\omega$. Moreover, by integration by parts, we have (noting that $h_0$ is non-singular on $(0,1]$, and $\tilde{f}^{(j)}(1)=0,\, j=0,\dots k$)
				\begin{align}\nonumber
					\hspace{-0.3cm}\int_{\omega^{-1}}^1H^{(1)}_0(\omega x)\tilde{f}(2x-1)e^{i\omega \beta x}dx&=\int_{\omega^{-1}}^1h_0(\omega x)\tilde{f}(2x-1)e^{i\omega(\beta+1) x}dx\\\begin{split}
						&=\revisions{\sum}_{j=0}^{k+1}\left(\frac{-1}{i\omega (\beta+1)}\right)^{j+1}\left[e^{i\omega(\beta+1) x}\frac{\D^j}{\D x^j}\left(h_0(\omega x)\tilde{f}(2x-1)\right)\right]_{x=\omega^{-1}}\\
						&\revisions{-}\left(\frac{-1}{i\omega (\beta+1)}\right)^{k+2}\left[e^{i\omega(\beta+1) x}\frac{\D^{k+1}}{\D x^{k+1}}\left(h_0(\omega x)\tilde{f}(2x-1)\right)\right]_{x=1}\\
						&\revisions{-}\left(\frac{-1}{i\omega (\beta+1)}\right)^{k+2}\int_{\omega^{-1}}^1e^{i\omega(\beta+1) x}\frac{\D^{k+2}}{\D x^{k+2}}\left(h_0(\omega x)\tilde{f}(2x-1)\right)dx.
					\end{split}
				\end{align}
				We bound each term in turn using the Leibniz rule for the derivatives of a product:
				\begin{align}\label{eqn:bound_on_first_terms}
					\left|\left[\frac{\D^j}{\D x^j}\left(h_0(\omega x)\tilde{f}(2x-1)\right)\right]_{x=\omega^{-1}}\right|&\lesssim \sum_{l=0}^j \omega^l \left|\left[\frac{\D^lh_0}{\D x^l}\right]_{x=1}\right|\,\, \left|\tilde{f}^{(j-l)}(2\omega^{-1}-1)\right|\lesssim \omega^{-k-1+j}\|\tilde{f}^{(k+1)}\|_{\revisions{\infty}},
				\end{align}
				where we used $\left|\tilde{f}^{(l)}(-1+2\omega^{-1})\right|\leq \tilde{C}_k \omega^{k+1-l}\|\tilde{f}^{(k+1)}\|_{k+1},\, l=0,\dots,k+1$. Similarly we find
				\begin{align}\label{eqn:bound_on_second_term}
					\left|\left[\frac{\D^{k+1}}{\D x^{k+1}}\left(h_0(\omega x)\tilde{f}(2x-1)\right)\right]_{x=1}\right|&=2^{k+1}\left|\left(h_0(\omega )\tilde{f}^{(k+1)}(1)\right)\right|\lesssim \omega^{-1/2}\|\tilde{f}^{(k+1)}\|_\infty,
				\end{align}
				where the first equality holds because $\tilde{f}^{(j)}(1)=0,j=0,\dots,k$. Finally, we have
				\begin{align*}
					\left|\frac{\D^{k+2}}{\D x^{k+2}}\left(h_0(\omega x)\tilde{f}(2x-1)\right)\right|&\leq \sum_{l=0}^{k+2}\binom{k+2}{l}\left|\frac{\D^{l}}{\D x^l}h_0(\omega x)\right|\left|2^{k+2-l}\tilde{f}^{(k+2-l)}(2x-1)\right|\\
					&\lesssim \omega^{-1/2}x^{-1/2}\left|\tilde{f}^{(k+2)}(2x-1)\right|+\sum_{l=1}^{k+2}\omega^{-1/2}x^{-l-1/2}\left|\tilde{f}^{(k+2-l)}(2x-1)\right|\\
					&\lesssim \omega^{-1/2}x^{-1/2}\|\tilde{f}^{(k+2)}\|_\infty+\sum_{l=1}^{k+2}\omega^{-1/2}x^{-l-1/2}x^{l-1}\|\tilde{f}^{(k+1)}\|_{\infty}\lesssim x^{-1}\|\tilde{f}^{(k+2)}\|_\infty,
				\end{align*}
				where the final estimate holds uniformly in $\omega^{-1}\leq x\leq 1$, since in that region $\omega^{-1/2}x^{-1/2}\leq1$. Therefore,
				\begin{align}\label{eqn:bound_on_final_term}
					\left|\int_{\omega^{-1}}^1e^{i\omega(\beta+1) x}\frac{\D^{k+2}}{\D x^{k+2}}\left(h_0(\omega x)\tilde{f}(2x-1)\right)dx\right|\lesssim \|\tilde{f}^{(k+2)}\|_\infty\int_{\omega^{-1}}^{1}x^{-1}dx\lesssim \|\tilde{f}^{(k+2)}\|_\infty\log\omega.
				\end{align}
				Thus, combining \eqref{eqn:split_of_integral}--\eqref{eqn:bound_on_final_term} yields the estimate
				\begin{align*}
					\left|I^{(3)}_{\omega,\beta}[\tilde{f}]\right|\lesssim \omega^{-(k+2)}\|\tilde{f}^{(k+1)}\|_{\infty}\sum_{j=0}^{k+2}|\beta+1|^{-j}+\omega^{-(k+2)}\log\omega|\beta+1|^{-(k+2)}\|\tilde{f}^{(k+2)}\|_{\infty},
				\end{align*}
				which completes the proof.
			\end{proof}
		\end{appendix}
\end{document}